\pdfoutput=1
\documentclass[hidelinks, reqno]{amsart}
\usepackage[english,french]{babel}
\usepackage{subcaption}
\usepackage{tabularx}
\usepackage{booktabs}
\usepackage{array}
\usepackage{amsmath}
\usepackage{amsfonts}
\usepackage{amssymb}
\usepackage{amsthm}
\usepackage[scr]{rsfso}
\usepackage{enumitem}
\usepackage{permute}
\usepackage{slashed}
\usepackage[usenames,dvipsnames]{xcolor}
\usepackage[pagebackref = true,colorlinks, linkcolor = Red, citecolor = Green, bookmarksdepth=2, linktocpage=true]{hyperref}
\usepackage[capitalize]{cleveref}
\usepackage{comment}
\usepackage[T1]{fontenc}
\usepackage{makecell}

\usepackage{tikz}
\usetikzlibrary{calc}

\allowdisplaybreaks

\DeclareMathOperator{\interior}{int}

\DeclareMathOperator{\Id}{Id}
\DeclareMathOperator{\tr}{tr}
\DeclareMathOperator{\Stab}{Stab}

\DeclareMathOperator{\supp}{supp}

\DeclareMathOperator{\SL}{SL}
\DeclareMathOperator{\GL}{GL}
\DeclareMathOperator{\SO}{SO}
\DeclareMathOperator{\U}{U}

\DeclareMathOperator{\Lip}{Lip}

\DeclareMathOperator{\proj}{proj}
\DeclareMathOperator{\Cay}{Cay}

\DeclareMathOperator{\T}{T}
\DeclareMathOperator{\F}{F}

\DeclareMathOperator{\diam}{diam}
\DeclareMathOperator{\len}{len}
\let\Pr\relax
\DeclareMathOperator{\Pr}{Pr}

\DeclareMathOperator{\Ad}{Ad}
\DeclareMathOperator{\BP}{BP}
\DeclareMathOperator{\inj}{inj}

\DeclareMathOperator{\li}{li}
\DeclareMathOperator{\Core}{Core}
\DeclareMathOperator{\Hull}{Hull}

\DeclareFontFamily{U}{mathb}{\hyphenchar\font45}
\DeclareFontShape{U}{mathb}{m}{n}{
      <5> <6> <7> <8> <9> <10> gen * mathb
      <10.95> mathb10 <12> <14.4> <17.28> <20.74> <24.88> mathb12
      }{}
\DeclareSymbolFont{mathb}{U}{mathb}{m}{n}
\DeclareMathSymbol{\bigast}{2}{mathb}{"06}

\def\XXint#1#2#3{{\setbox0=\hbox{$#1{#2#3}{\int}$}
     \vcenter{\hbox{$#2#3$}}\kern-.5\wd0}}

\theoremstyle{plain}
\newtheorem{theorem}{Theorem}[section]
\newtheorem{proposition}[theorem]{Proposition}
\newtheorem{lemma}[theorem]{Lemma}
\newtheorem{corollary}[theorem]{Corollary}
\newtheorem{conjecture}[theorem]{Conjecture}

\theoremstyle{definition}
\newtheorem{definition}[theorem]{Definition}

\newtheoremstyle{remark}
{}   
{}   
{\normalfont}  
{}       
{\itshape} 
{.}         
{5pt plus 1pt minus 1pt} 
{}          

\theoremstyle{remark}
\newtheorem*{remark}{Remark}

\setlist[enumerate,1]{ref=(\arabic*)}
\setlist[enumerate,2]{ref=(\theenumi)(\alph*)}
\setlist[enumerate,3]{ref=(\theenumi)(\theenumii)(\roman*)}
\setlist[enumerate,4]{ref=(\theenumi)(\theenumii)(\theenumiii)(\Alph*)}

\newlist{alternative}{enumerate}{4}     
\setlist[alternative,1]{label=(\arabic*), ref=(\arabic*)}
\setlist[alternative,2]{label=(\alph*), ref=(\thealternativei)(\alph*)}
\setlist[alternative,3]{label=(\roman*), ref=(\thealternativei)(\thealternativeii)(\roman*)}
\setlist[alternative,4]{label=(\Alph*), ref=(\thealternativei)(\thealternativeii)(\thealternativeiii)(\Alph*)}

\Crefname{enumi}{Property}{Properties}
\Crefname{alternativei}{Alternative}{Alternatives}
\Crefname{subsection}{Subsection}{Subsections}

\makeatletter
\renewenvironment{abstract}{%
  \ifx\maketitle\relax
    \ClassWarning{\@classname}{Abstract should precede
      \protect\maketitle\space in AMS document classes; reported}%
  \fi
  \global\setbox\abstractbox=\vtop \bgroup
    \normalfont\Small
    \list{}{\labelwidth\z@
      \leftmargin3pc \rightmargin\leftmargin
      \listparindent\normalparindent \itemindent\z@
      \parsep\z@ \@plus\p@
      
      \itemsep\medskipamount
    }%
}{%
  \endlist\egroup
  \ifx\@setabstract\relax \@setabstracta \fi
}
\makeatother

\begin{document}
\selectlanguage{english}
\title[Generalization of Selberg's $3/16$ theorem]{Generalization of Selberg's $3/16$ theorem for convex cocompact thin subgroups of $\SO(n, 1)$}
\author{Pratyush Sarkar}
\address{Department of Mathematics, Yale University, New Haven, Connecticut 06511}
\email{pratyush.sarkar@yale.edu}
\date{\today}

\begin{abstract}
\item[\hskip\labelsep\scshape Abstract.]
Let $\Gamma$ be a convex cocompact thin subgroup of an arithmetic lattice in $\SO(n, 1)$. We generalize Selberg's $\frac{3}{16}$ theorem in this setting, i.e., we prove uniform exponential mixing of the frame flow and obtain a uniform resonance-free half plane for the congruence covers of the hyperbolic manifold $\Gamma \backslash \mathbb H^n$. This extends the work of Oh--Winter who established the $n = 2$ case.

The theorem follows from uniform spectral bounds for the congruence transfer operators with holonomy. We employ Sarkar--Winter's frame flow version of Dolgopyat's method uniformly over the congruence covers as well as Golsefidy--Varj\'{u}'s generalization of Bourgain--Gamburd--Sarnak's expansion machinery by using the properties that the return trajectory subgroups are Zariski dense and have trace fields which coincide with that of $\Gamma$. These properties follow by proving that the return trajectory subgroups have finite index in $\Gamma$.

\noindent\hrulefill
\item[\hskip\labelsep\scshape R\'{e}sum\'{e}]
\scshape (G\'{e}n\'{e}ralisation du th\'{e}or\`{e}me $3/16$ de Selberg pour sous-groupes \'{e}pars convexes cocompactes de $\SO(n, 1)$).
\normalfont Soit $\Gamma$ un sous-groupe \'{e}pars convexe cocompact de r\'{e}seau arithm\'{e}tique dans $\SO(n, 1)$. Nous g\'{e}n\'{e}ralisons le th\'{e}or\`{e}me $\frac{3}{16}$ de Selberg dans ce contexte, c'est-\`{a}-dire que nous d\'{e}montrons un m\'{e}lange exponentiel uniforme du flot rep\`{e}re et obtenons un demi-plan uniforme sans r\'{e}sonance pour les rev\^{e}tements de congruence de la vari\'{e}t\'{e} hyperbolique $\Gamma \backslash \mathbb H^n$. Cela \'{e}tend le travail d'Oh--Winter qui ont d\'{e}montr\'{e} le r\'{e}sultat pour le cas $n = 2$.

Le th\'{e}or\`{e}me s'ensuit des bornes spectrales uniformes pour les op\'{e}rateurs de transfert de congruence avec l'holonomie. Nous employons la version flot rep\`{e}re par Sarkar--Winter de la m\'{e}thode de Dolgopyat uniform\'{e}ment sur les rev\^{e}tements de congruence ainsi que la g\'{e}n\'{e}ralisation par Golsefidy--Varj\'{u} de la machinerie d'expansion de Bourgain--Gamburd--Sarnak en utilisant les propri\'{e}t\'{e}s que les sous-groupes de trajectoire de retour sont Zariski denses et ont des corps des traces qui co\"{i}ncident avec celui de $\Gamma$. Ces propri\'{e}t\'{e}s suivent en d\'{e}montrant que les sous-groupes de trajectoire de retour ont l'indice fini dans $\Gamma$.
\end{abstract}

\maketitle

\setcounter{tocdepth}{1}
\tableofcontents


\section{Introduction}
\label{sec:Introduction}
\subsection{Historical background}
\label{subsec:HistoricalBackground}
Selberg's eigenvalue conjecture is one of the most influential conjectures which has connections to spectral theory, number theory, and Riemann surfaces. We recount it in \cref{con:Selberg's_1/4_Conjecture}. Let $\mathbb H^2$ be the hyperbolic plane. The conjecture is concerned with the family of spaces given by the base space $X = \Gamma \backslash \mathbb H^2$ where $\Gamma = \SL_2(\mathbb Z)$ and its \emph{congruence covers} $X_q = \Gamma_q \backslash \mathbb H^2$ of \emph{level} $q$ where $\Gamma_q = \ker(\pi_q)$ is the kernel of the canonical quotient map $\pi_q: \SL_2(\mathbb Z) \to \SL_2(\mathbb Z/q\mathbb Z)$ for all $q \in \mathbb N$. Denote the $L^2$ eigenvalues of the Laplacian on $X_q$ by
\begin{align*}
0 = \lambda_0(X_q) < \lambda_1(X_q) \leq \lambda_2(X_q) \leq \dotsb \qquad \text{for all $q \in \mathbb N$}.
\end{align*}

\begin{conjecture}
\label{con:Selberg's_1/4_Conjecture}
For all $q \in \mathbb N$, we have $\lambda_1(X_q) \geq \frac{1}{4}$.
\end{conjecture}

Selberg proved his $\frac{3}{16}$ theorem which is the above statement with the lower bound $\frac{1}{4}$ replaced by $\frac{3}{16}$. The best known improvement of the lower bound is $\frac{975}{4096}$ by Kim--Sarnak \cite[Appendix 2]{Kim03}. Theorems of this type are called \emph{spectral gap} and they show that the arithmetic nature of the congruence covers have surprising implications on their spectral properties. Using the theory of tempered and spherical representations, we can then derive the \emph{decay of matrix coefficients} reformulation \cref{con:UniformExponentialMixingOfTheGeodesicFlowSL2Z} from \cref{con:Selberg's_1/4_Conjecture}. This is of particular interest from the point of view of homogeneous dynamics because it can be interpreted as a \emph{uniform exponential mixing} statement. Corresponding to the congruence covers of $X$, we have the family of homogeneous spaces given by the unit tangent bundle $\T^1(X) \cong \Gamma \backslash \SL_2(\mathbb R)$ and its congruence covers $\T^1(X_q) \cong \Gamma_q \backslash \SL_2(\mathbb R)$ for all $q \in \mathbb N$. For all $q \in \mathbb N$, let $m_q^{\mathrm{Haar}}$ be the right $\SL_2(\mathbb R)$-invariant finite measure on $\Gamma_q \backslash \SL_2(\mathbb R)$ induced from some fixed Haar measure on $\SL_2(\mathbb R)$. We recall that the right translation action by
\begin{align*}
A = \left\{a_t =
\begin{pmatrix}
e^{\frac{t}{2}} & 0 \\
0 & e^{-\frac{t}{2}}
\end{pmatrix}:
t \in \mathbb R
\right\}
\end{align*}
parametrized in the above fashion corresponds to the geodesic flow which we note is measure preserving. We denote by $\|\cdot\|_{\mathcal{S}^1}$ the $L^2$ Sobolev norm of order $1$.

\begin{conjecture}
\label{con:UniformExponentialMixingOfTheGeodesicFlowSL2Z}
For all $q \in \mathbb N$ and $\epsilon > 0$, there exists $C > 0$ such that for all $\phi, \psi \in C_{\mathrm{c}}^\infty(\Gamma_q \backslash \SL_2(\mathbb R), \mathbb R)$ and $t > 0$, we have
\begin{multline*}
\left|\int_{\Gamma_q \backslash \SL_2(\mathbb R)} \phi(xa_t) \psi(x) \, dm_q^{\mathrm{Haar}}(x) - \frac{1}{m_q^{\mathrm{Haar}}(\Gamma_q \backslash \SL_2(\mathbb R))}m_q^{\mathrm{Haar}}(\phi) \cdot m_q^{\mathrm{Haar}}(\psi)\right| \\
\leq Ce^{-(\frac{1}{2} - \epsilon) t} \|\phi\|_{\mathcal{S}^1} \|\psi\|_{\mathcal{S}^1}.
\end{multline*}
\end{conjecture}

Corresponding to Selberg's $\frac{3}{16}$ theorem, we get the above statement with $\frac{1}{2}$ replaced by $\frac{1}{4}$. This is very powerful because apart from being of intrinsic interest in representation theory and homogeneous dynamics, it can be used to derive all sorts of number theoretic results.

So far in our discussion, $\Gamma = \SL_2(\mathbb Z)$ is a lattice in $\SL_2(\mathbb R)$. However, it is also natural to investigate aforementioned properties in cases when $\Gamma$ is a non-lattice. This was done by Bourgain--Gamburd--Sarnak in their breakthrough papers \cite{BGS10,BGS11} for the appropriate setting where $\Gamma$ is a \emph{thin} group, i.e., a Zariski dense subgroup of infinite index in $\SL_2(\mathbb Z)$. Let $\delta_\Gamma \in (0, 1]$ denote the critical exponent of $\Gamma$. The error term they obtained is exponential when $\delta_\Gamma > \frac{1}{2}$ but not when $\delta_\Gamma \in (0, \frac{1}{2}]$. Nevertheless, they anticipated that an exponential error term is still possible when $\delta_\Gamma \in (0, \frac{1}{2}]$ and this was indeed achieved by Oh--Winter in the important work \cite{OW16}. In the same spirit, this paper is concerned with generalizing the results of \cite{OW16} to higher dimensions, i.e., we establish uniform exponential mixing for frame flows on the congruence covers associated to any convex cocompact thin subgroup $\Gamma < \SO(n, 1)^\circ$ for any $n \geq 2$.

\subsection{Main result}
\label{subsec:Result}
Let $\mathbb H^n$ be the $n$-dimensional hyperbolic space for $n \geq 2$. Let $\mathbb K \subset \mathbb R$ be a totally real number field and $\mathcal{O}_{\mathbb K}$ be the corresponding ring of integers. Let $\mathbf{G} < \GL_N$ for some $N \in \mathbb N$ be an algebraic group defined over $\mathbb K$ such that $\mathbf{G}(\mathbb R) \cong \SO(n, 1)$ and $\mathbf{G}^\sigma(\mathbb R)$ is compact for all nontrivial embeddings $\sigma: \mathbb K \hookrightarrow \mathbb R$. Let $G = \mathbf{G}(\mathbb R)^\circ$ which we recognize as the group of orientation preserving isometries of $\mathbb H^n$. We identify $\mathbb H^n$, its unit tangent bundle $\T^1(\mathbb H^n)$, and its frame bundle $\F(\mathbb H^n)$ with $G/K$, $G/M$, and $G$ respectively where $M < K$ are compact subgroups of $G$. Let $A = \{a_t: t \in \mathbb R\} < G$ be a one-parameter subgroup of semisimple elements such that its right translation action corresponds to the geodesic flow on $G/M$ and the frame flow on $G$.

Let $\tilde{\pi}: \tilde{\mathbf{G}} \to \mathbf{G}$ be a simply connected cover defined over $\mathbb K$. For all ideals $\mathfrak{q} \subset \mathcal{O}_{\mathbb K}$, let $\pi_{\mathfrak{q}}: \tilde{\mathbf{G}}(\mathcal{O}_{\mathbb K}) \to \tilde{\mathbf{G}}(\mathcal{O}_{\mathbb K}/\mathfrak{q})$ be the canonical quotient map. Let $\Gamma < \mathbf{G}(\mathbb K)$ be a Zariski dense torsion-free convex cocompact subgroup such that $\tilde{\pi}^{-1}(\Gamma)$ is contained in $\tilde{\mathbf{G}}(\mathcal{O}_{\mathbb K})$ with trace field $\mathbb Q(\tr(\Ad(\tilde{\pi}^{-1}(\Gamma)))) = \mathbb K$. We impose these conditions so that $\Gamma$ satisfies the strong approximation property. For all nontrivial ideals $\mathfrak{q} \subset \mathcal{O}_{\mathbb K}$, let $\Gamma_\mathfrak{q} < \Gamma$ be the congruence subgroup of level $\mathfrak{q}$, meaning that $\Gamma_\mathfrak{q} = \Gamma \cap \tilde{\pi}(\ker(\pi_\mathfrak{q}))$. For all nontrivial ideals $\mathfrak{q} \subset \mathcal{O}_{\mathbb K}$, let $N_{\mathbb K}(\mathfrak{q})$ be the norm of the ideal $\mathfrak{q}$ and $m^{\mathrm{BMS}}_\mathfrak{q}$ be the Bowen--Margulis--Sullivan measure on $\Gamma_\mathfrak{q} \backslash G$ induced from the one on $\Gamma \backslash G/M$.

\begin{theorem}
\label{thm:TheoremUniformExponentialMixingOfFrameFlow}
There exist $\eta > 0$, $C > 0$, $r \in \mathbb N$, and a nontrivial proper ideal $\mathfrak{q}_0 \subset \mathcal{O}_{\mathbb K}$ such that for all square-free ideals $\mathfrak{q} \subset \mathcal{O}_{\mathbb K}$ coprime to $\mathfrak{q}_0$, $\phi \in C_{\mathrm{c}}^r(\Gamma_\mathfrak{q} \backslash G, \mathbb R)$, $\psi \in C_{\mathrm{c}}^1(\Gamma_\mathfrak{q} \backslash G, \mathbb R)$, and $t > 0$, we have
\begin{multline*}
\left|\int_{\Gamma_\mathfrak{q} \backslash G} \phi(xa_t)\psi(x) \, dm^{\mathrm{BMS}}_\mathfrak{q}(x) - \frac{1}{m^{\mathrm{BMS}}_\mathfrak{q}(\Gamma_\mathfrak{q} \backslash G)} m^{\mathrm{BMS}}_\mathfrak{q}(\phi) \cdot m^{\mathrm{BMS}}_\mathfrak{q}(\psi)\right| \\
\leq CN_{\mathbb K}(\mathfrak{q})^C e^{-\eta t} \|\phi\|_{C^r} \|\psi\|_{C^1}.
\end{multline*}
\end{theorem}

\begin{remark}
A recent work of He--Saxc\'{e} \cite[Theorem 6.1]{HdS19} removes the square-free assumption in the work of Golsefidy--Varj\'{u} \cite[Theorem 1]{GV12} for simple algebraic groups, as is the case for $n \neq 3$. Thus, using the work of He--Saxc\'{e} instead as a black box in the proof of \cref{lem:GV_Expander} in \cref{sec:L2FlatteningLemma}, we can in fact \emph{remove} the square-free hypothesis in \cref{thm:TheoremUniformExponentialMixingOfFrameFlow} for $n \neq 3$.
\end{remark}

\begin{remark}
We mention here some special cases and related works.
\begin{enumerate}
\item For $n = 2$, \cref{thm:TheoremUniformExponentialMixingOfFrameFlow} has been established by Oh--Winter \cite{OW16}. Moreover, if $\delta_\Gamma > \frac{1}{2}$, then spectral gap was established by Bourgain--Gamburd--Sarnak \cite{BGS11}. Moreover, if $\delta_\Gamma > \frac{5}{6}$, then building on the work of Sarnak--Xue \cite{SX91}, Gamburd \cite{Gam02} obtained an \emph{explicit} interval $[\delta_\Gamma(1 - \delta_\Gamma), \frac{5}{36})$ where no new eigenvalues arise.
\item For a single convex cocompact hyperbolic manifold for any $n \geq 2$ without the congruence aspect, exponential mixing for the \emph{geodesic flow} follows from results of Stoyanov \cite{Sto11} which used Dolgopyat's method \cite{Dol98}. Exponential mixing for the \emph{frame flow} was obtained in \cite{SW20}.
\item If $\delta_\Gamma > \frac{n - 1}{2}$ for $n \in \{2, 3\}$, or if $\delta_\Gamma > n - 2$ for $n \geq 4$, then \cref{thm:TheoremUniformExponentialMixingOfFrameFlow} has been established by Mohammadi--Oh \cite{MO15} using the work on expanders by Bourgain--Gamburd--Sarnak \cite{BGS10,BGS11} and unitary representation theory. Recently, Edwards--Oh \cite{EO19} have improved the condition to $\delta_\Gamma > \frac{n - 1}{2}$ for all $n \geq 2$. If $\delta_\Gamma > s_n^0 = (n - 1) - \frac{2(n - 2)}{n(n + 1)}$, then Magee \cite{Mag15} generalized Gamburd's work \cite{Gam02} to arbitrary $n \geq 3$ and obtained an \emph{explicit} interval $[\delta_\Gamma(1 - \delta_\Gamma), s_n^0(n - 1 - s_n^0))$ where no new eigenvalues arise.
\end{enumerate}
\end{remark}

\subsection{Applications}
\label{Applications}
Here we present some immediate applications of \cref{thm:TheoremUniformExponentialMixingOfFrameFlow}.

Let $\mathfrak{q} \subset \mathcal{O}_{\mathbb K}$ be an ideal. Let $m_\mathfrak{q}^{\mathrm{Haar}}$ be the right $G$-invariant measure on $\Gamma_\mathfrak{q} \backslash G$ induced from some fixed Haar measure on $G$. We denote by $m^{\mathrm{BR}}_\mathfrak{q}$ and $m^{\mathrm{BR}_*}_\mathfrak{q}$ the unstable and stable Burger--Roblin measures on $\Gamma_\mathfrak{q} \backslash G$ respectively, compatible with the choice of the Haar measure. The first application is the following theorem regarding decay of matrix coefficients which is derived exactly as in \cite{OS13,OW16} using Roblin's transverse intersection argument \cite{Rob03}.

\begin{theorem}
There exist $\eta > 0$, $c > 0$, $r \in \mathbb N$, and a nontrivial proper ideal $\mathfrak{q}_0 \subset \mathcal{O}_{\mathbb K}$ such that for all square-free ideals $\mathfrak{q} \subset \mathcal{O}_{\mathbb K}$ coprime to $\mathfrak{q}_0$, $\phi \in C_{\mathrm{c}}^r(\Gamma_\mathfrak{q} \backslash G, \mathbb R)$, $\psi \in C_{\mathrm{c}}^1(\Gamma_\mathfrak{q} \backslash G, \mathbb R)$, there exists $C > 0$ such that for all $t > 0$, we have
\begin{multline*}
\left|e^{(n - 1 - \delta_\Gamma)t}\int_{\Gamma_\mathfrak{q} \backslash G} \phi(xa_t)\psi(x) \, dm_\mathfrak{q}^{\mathrm{Haar}}(x) - \frac{1}{m^{\mathrm{BMS}}_\mathfrak{q}(\Gamma_\mathfrak{q} \backslash G)} m^{\mathrm{BR}}_\mathfrak{q}(\phi) \cdot m^{\mathrm{BR}_*}_\mathfrak{q}(\psi)\right| \\
\leq CN_{\mathbb K}(\mathfrak{q})^c e^{-\eta t} \|\phi\|_{C^r} \|\psi\|_{C^1}
\end{multline*}
where $C$ depends only on the projections of $\supp(\phi)$ and $\supp(\psi)$ under the covering map $\Gamma_\mathfrak{q} \backslash G \to \Gamma \backslash G$.
\end{theorem}

The second application is to affine sieve as in \cite{BGS11}. By \cite{MO15}, \cref{thm:TheoremUniformExponentialMixingOfFrameFlow} is sufficient to derive \cref{thm:AffineSieveTheorem}. For this theorem, consider the special case $\mathbf{G} = \SO_Q$ for some integral quadratic form $Q$ of signature $(n, 1)$ and $\Gamma < \mathbf{G}(\mathbb Z)$. Let $\mathbf{G}(\mathbb R)$ act on $\mathbb R^{n + 1}$ by the standard representation, where we equip $\mathbb R^{n + 1}$ with any norm $\|\cdot\|$.

\begin{theorem}
\label{thm:AffineSieveTheorem}
Let $w_0 \in \mathbb Z^{n + 1} \setminus \{0\}$. If $n = 2$ and $Q(w_0) > 0$, we further assume that $w_0$ is not externally $\Gamma$-parabolic (see \cite{MO15} for details). Then, there exist $C_1 > 0$, $C_2 > 0$, and $R > 0$ such that for all $F = F_1 F_2 \dotsb F_r \in \mathbb Q[x_1, x_2, \dotsc, x_{n + 1}]$ factored into irreducible polynomials over $\mathbb Q$ and $T > 0$, we have
\begin{align*}
\#\{x \in w_0\Gamma: \|x\| \leq T,  F_j(x) \text{ is prime for all } j \in \{1, 2, \dotsc, r\}\} &\leq C_1\frac{T^{\delta_\Gamma}}{\log(T)^r}; \\
\#\{x \in w_0\Gamma: \|x\| \leq T, F(x) \text{ has at most $R$ prime factors}\} &\geq C_2\frac{T^{\delta_\Gamma}}{\log(T)^r}.
\end{align*}
\end{theorem}

The third application is regarding uniform exponential equidistribution of holonomy of closed geodesics as in \cite{MMO14} and an asymptotic orbit counting formula as in \cite{MO15}. Let $\mathfrak{q} \subset \mathcal{O}_{\mathbb K}$ be an ideal. Define
\begin{align*}
\mathcal{G}_\mathfrak{q}(T) ={}&\#\{\gamma: \gamma \text{ is a primitive closed geodesic in } \Gamma_\mathfrak{q} \backslash \mathbb H^n \text{ with length at most } T\}
\end{align*}
and $\mathcal N_{\mathfrak{q}}(T; u, v) = \#\{\gamma \in \Gamma_{\mathfrak{q}}: d(u, \gamma v) \leq T\}$ for all $u, v \in \mathbb H^n$ and $T > 0$. For all primitive closed geodesics $\gamma$ in $\Gamma_\mathfrak{q} \backslash \mathbb H^n$, its \emph{holonomy} is a conjugacy class $h_\gamma$ in $M$ induced by parallel transport along $\gamma$. Fix the probability Haar measure on $M$. Recall the function $\li: (2, \infty) \to \mathbb R$ defined by $\li(x) = \int_2^x \frac{1}{\log(t)} \, dt$ for all $x \in (2, \infty)$.

\begin{theorem}
\label{thm:CountingPrimitiveClosedGeodesicsAndOrbits}
There exist $\eta > 0$, $C > 0$, and a nontrivial proper ideal $\mathfrak{q}_0 \subset \mathcal{O}_{\mathbb K}$ such that for all square-free ideals $\mathfrak{q} \subset \mathcal{O}_{\mathbb K}$ coprime to $\mathfrak{q}_0$, we have
\begin{enumerate}
\item for all class functions $\phi \in C^\infty(M, \mathbb R)$ and $T > \frac{\log(2)}{\delta_\Gamma}$, we have
\begin{align*}
\left|\sum_{\gamma \in \mathcal{G}_\mathfrak{q}(T)} \phi(h_\gamma) - \li\bigl(e^{\delta_\Gamma T}\bigr) \int_M \phi \, dm\right| \leq CN_{\mathbb K}(\mathfrak{q})^C e^{(\delta_\Gamma - \eta)T};
\end{align*}
\item\label{itm:TEST} for all $u, v \in \mathbb H^n$, there exists $c > 0$ such that
\begin{align*}
\big|\mathcal N_{\mathfrak{q}}(T; u, v) - ce^{\delta_\Gamma T}\big| \leq CN_{\mathbb K}(\mathfrak{q})^C e^{(\delta_\Gamma - \eta)T} \qquad \text{for all $T > 0$}.
\end{align*}
\end{enumerate}
\end{theorem}

Let $\mathfrak{q} \subset \mathcal{O}_{\mathbb K}$ be an ideal. If $\delta_\Gamma \in (\frac{n - 1}{2}, n - 1]$, then \cref{thm:ResonanceFreeRegions} follows from the spectral gap result in Step II of the proof of \cite[Theorem 11.1]{Kim15} which generalizes \cite[Theorem 1.2]{BGS11}. Thus, we now focus on the case $\delta_\Gamma \in (0, \frac{n - 1}{2}]$ which is exactly when the spectrum of the Laplacian $\Delta_\mathfrak{q}: L^2(\Gamma_\mathfrak{q} \backslash \mathbb H^n, \mathbb C) \to L^2(\Gamma_\mathfrak{q} \backslash \mathbb H^n, \mathbb C)$ is purely continuous, namely $\bigl[\left(\frac{n - 1}{2}\right)^2, \infty\bigr)$ \cite{LP82,Sul82}. The resolvent $\xi \mapsto R_\mathfrak{q}(\xi) = (\Delta_\mathfrak{q} - \xi(n - 1 - \xi))^{-1}$ is then holomorphic on the half plane $\{\xi \in \mathbb C: \Re(\xi) > \frac{n - 1}{2}\}$. Mazzeo--Melrose \cite{MM87} and Guillop\'{e}--Zworski \cite{GZ95} showed that viewing the resolvent $R_\mathfrak{q}$ as a $L_{\mathrm{c}}^2(\Gamma_\mathfrak{q} \backslash \mathbb H^n, \mathbb C) \to L_{\mathrm{loc}}^2(\Gamma_\mathfrak{q} \backslash \mathbb H^n, \mathbb C)$ bounded operator-valued map, it has a meromorphic extension to the entire complex plane, whose poles are called \emph{resonances}. Moreover, Patterson \cite{Pat88} showed that the line $\{\xi \in \mathbb C: \Re(\xi) = \delta_\Gamma\}$ has exactly one simple pole at $\delta_\Gamma$ for the map $\xi \mapsto \Gamma(\xi - \frac{n - 1}{2} + 1) R_\mathfrak{q}(\xi)$. Finally, another application is \cref{thm:ResonanceFreeRegions} regarding a uniform resonance-free half plane for these resolvents. It is proved exactly as in \cite[Section 6]{OW16} using \cite[Proposition 2.2]{GN09}, with the correct factor $f(\xi) = \frac{2^{-1}\pi^{-\frac{n - 1}{2}}\Gamma(\xi)}{\Gamma(\xi - \frac{n - 1}{2} + 1)}$, along with the asymptotic orbit counting formula in \cref{thm:CountingPrimitiveClosedGeodesicsAndOrbits}.

\begin{theorem}
\label{thm:ResonanceFreeRegions}
There exist $\epsilon > 0$ and a nontrivial proper ideal $\mathfrak{q}_0 \subset \mathcal{O}_{\mathbb K}$ such that for all square-free ideals $\mathfrak{q} \subset \mathcal{O}_{\mathbb K}$ coprime to $\mathfrak{q}_0$, $\{\xi \in \mathbb C: \Re(\xi) > \delta_\Gamma - \epsilon\}$ is a resonance-free half plane for $R_\mathfrak{q}$ if $\delta_\Gamma \in \{\frac{n - 1}{2} - k: k \in \mathbb N\}$ but with an exception of a simple pole at $\delta_\Gamma$ if $\delta_\Gamma \notin \{\frac{n - 1}{2} - k: k \in \mathbb N\}$.
\end{theorem}

\subsection{Outline of the proof of \texorpdfstring{\cref{thm:TheoremUniformExponentialMixingOfFrameFlow}}{\autoref{thm:TheoremUniformExponentialMixingOfFrameFlow}}}
First we recount the proof of exponential mixing for the geodesic flow on the single manifold $\T^1(\Gamma \backslash \mathbb H^n) \cong \Gamma \backslash G/M$. Then based on that we describe the proof of uniform exponential mixing for the frame flow in this paper.

As $\Gamma$ is convex cocompact, it is well known from the works of Bowen and Ratner \cite{Bow70,Rat73} that there exists a Markov section for the geodesic flow on the compact invariant subset $\Omega =\supp(m^{\mathrm{BMS}}) \subset \T^1(\Gamma \backslash \mathbb H^n)$. This provides a coding for the geodesic flow and allows us to use tools from symbolic dynamics and thermodynamic formalism. Moreover, by an observation of Pollicott \cite{Pol85} which was used by many others \cite{Dol98,AGY06,Sto11}, the Laplace transform of the correlation function can be written as an infinite series involving the transfer operators $\mathcal{L}_\xi: C(U, \mathbb C) \to C(U, \mathbb C)$ for $\xi = a + ib \in \mathbb C$ defined by
\begin{align*}
\mathcal{L}_\xi(h)(u) = \sum_{u' \in \sigma^{-1}(u)} e^{-(a + \delta_\Gamma - ib)\tau(u')}h(u').
\end{align*}
Here $\sigma$ is the shift map on $U$, $\delta_\Gamma$ is the critical exponent of $\Gamma$ which is known to be equal to the entropy of the geodesic flow, and $\tau$ is the first return time map associated to the Markov section. Now, by a Paley--Wiener type of analysis, it suffices to find spectral bounds for the transfer operators to prove exponential mixing of the geodesic flow.

For small frequencies $|\Im(\xi)| \ll 1$, the spectral bounds for the transfer operators can be derived using the Ruelle--Perron--Frobenius theorem together with perturbation theory. On the other hand, large frequencies $|\Im(\xi)| \gg 1$ are much harder to deal with. However, Dolgopyat \cite{Dol98} invented an ingenious method to obtain spectral bounds by working explicitly on the strong unstable leaves of the Markov section rather than the purely symbolic space. The geometry of the manifold provides a crucial \emph{local non-integrability condition} (LNIC) which implies that $\tau$ is highly oscillating. From this we can get the required spectral bounds. Moreover, he provided the right technical framework for the whole process to work in harmony.

We would like to follow a similar line of argument to prove exponential mixing for the \emph{frame flow} on a family of manifolds $\F(\Gamma_\mathfrak{q} \backslash \mathbb H^n) \cong \Gamma_\mathfrak{q} \backslash G$ \emph{uniformly} in the ideals $\mathfrak{q} \subset \mathcal{O}_{\mathbb K}$. The novelty of this paper in contrast with past works is that we handle \emph{holonomy} and the uniformity in the ideals $\mathfrak{q} \subset \mathcal{O}_{\mathbb K}$ simultaneously. We are also able to do this whilst in arbitrary dimension $n \geq 2$.

Our strategy begins with using instead the \emph{congruence} transfer operators with holonomy $\mathcal{M}_{\xi, \mathfrak{q}, \rho}: C\big(U, L^2(\Gamma_\mathfrak{q} \backslash \Gamma, \mathbb C) \otimes {V_\rho}^{\oplus \dim(\rho)}\big) \to C\big(U, L^2(\Gamma_\mathfrak{q} \backslash \Gamma, \mathbb C) \otimes {V_\rho}^{\oplus \dim(\rho)}\big)$ for $\xi = a + ib \in \mathbb C$ and ideal $\mathfrak{q} \subset \mathcal{O}_{\mathbb K}$ defined by
\begin{align*}
\mathcal{M}_{\xi, \mathfrak{q}, \rho}(H)(u) &= \sum_{u' \in \sigma^{-1}(u)} e^{-(a + \delta_\Gamma - ib)\tau(u')} \big(\mathtt{c}_\mathfrak{q}(u')^{-1} \otimes \rho(\vartheta(u'))^{-1}\big) H(u').
\end{align*}
Here the \emph{congruence cocycle} $\mathtt{c}_\mathfrak{q}$ ``keeps track of the coordinate in the fibers'' of the congruence cover $\T^1(\Gamma_\mathfrak{q} \backslash \mathbb H^n) \to \T^1(\Gamma \backslash \mathbb H^n)$, the \emph{holonomy} $\vartheta$ ``keeps track of the coordinate in the fibers'' of the principal $M$-bundle $\F(\Gamma_\mathfrak{q} \backslash \mathbb H^n) \to \T^1(\Gamma_\mathfrak{q} \backslash \mathbb H^n)$, and $\rho$ is an irreducible representation of $M$. With this formulation, we now require spectral bounds for the congruence transfer operators with holonomy, uniformly in the ideals $\mathfrak{q} \subset \mathcal{O}_{\mathbb K}$ and irreducible representations $\rho$.

But now for small frequencies $|\Im(\xi)| \ll 1$ and trivial $\rho$, since there are countably many ideals $\mathfrak{q} \subset \mathcal{O}_{\mathbb K}$, the Ruelle--Perron--Frobenius theorem and perturbation theory are insufficient to obtain spectral bounds uniformly in the ideals $\mathfrak{q} \subset \mathcal{O}_{\mathbb K}$. Nevertheless, Bourgain--Gamburd--Sarnak demonstrated in their breakthrough works \cite{BGS10,BGS11} that the required bounds are attainable using the entirely new methods of expander graphs. Golsefidy--Varj\'{u} \cite{GV12} generalized their results to perfect algebraic groups. The expander machinery captures the spectral properties of a certain family of Cayley graphs coming from the congruence setting. This gives large cancellations of the summands in the congruence transfer operators due to the cocycle from which we can derive uniform spectral bounds. We mainly follow \cite{OW16,MOW17} to use the expander machinery. In particular, Bourgain--Kontorovich--Magee has shown a method in \cite[Appendix]{MOW17} to use the expander machinery more directly. In order to generalize their techniques in our setting, it is crucial to verify the Zariski density of certain subgroups of $\Gamma$ which we call the \emph{return trajectory subgroups}. Over a general totally real number field $\mathbb K$, we also need to verify that the return trajectory subgroups have full trace field $\mathbb K$. We ensure these properties by proving that the return trajectory subgroups have finite index in $\Gamma$. Although these were trivial points in \cite[Appendix]{MOW17} for the Schottky semigroup and continued fraction settings in $\SL_2(\mathbb Z)$, its higher dimensional analogue is quite intricate, and obtaining this is the main new technical result of this paper. This completes the necessary tools and we obtain spectral bounds for small frequencies $|\Im(\xi)| \ll 1$ and trivial $\rho$, uniformly in the ideals $\mathfrak{q} \subset \mathcal{O}_{\mathbb K}$.

Next, for large frequencies $|\Im(\xi)| \gg 1$ or nontrivial $\rho$, we would like to use Dolgopyat's method as above. We deal with the cocycle with relative ease since it is locally constant on $U$. This was first observed by Oh--Winter in \cite{OW16}. The holonomy on the other hand is trickier to deal with. However, this has already been dealt with in \cite{SW20} where the appropriate LNIC and \emph{non-concentration property} (NCP) was proven and used to carry out Dolgopyat's method. The difference in this paper is to incorporate the cocycle and make sure everything in \cite{SW20} goes through uniformly in the ideals $\mathfrak{q} \subset \mathcal{O}_{\mathbb K}$. We need to modify a few things but there are no complications because of the local constancy of the cocycle. Dolgopyat's method then gives spectral bounds for large frequencies $|\Im(\xi)| \gg 1$ or nontrivial $\rho$, uniformly in the ideals $\mathfrak{q} \subset \mathcal{O}_{\mathbb K}$.

Finally, applying a Paley--Wiener type of analysis, the uniform spectral bounds with holonomy can be converted to uniform exponential mixing of the frame flow.

\subsection{Organization of the paper}
\label{subsec:OrganizationOfThePaper}
We start with the background for our dynamical system in \cref{sec:Preliminaries}. In \cref{sec:CodingTheGeodesicFlow}, we will first code the geodesic flow so that we can use tools from symbolic dynamics and thermodynamic formalism. Then we construct the congruence setting and define cocycles. In \cref{sec:HolonomyAndRepresentationTheory}, we define holonomy and cover some representation theoretic background. In \cref{sec:CongruenceTransferOperatorsWithHolonomyAndTheirUniformSpectralBounds}, we define congruence transfer operators with holonomy and present their uniform spectral bounds which is the main technical theorem in this paper. To prove these bounds, \cref{sec:ReductionToNewInvariantFunctionsAtLevel_q,sec:ApproximatingTheTransferOperator,sec:ZariskiDensityAndTraceFieldOfTheReturnTrajectorySubgroups,sec:L2FlatteningLemma,sec:SupremumAndLipschitzBounds} are devoted to the expander machinery part of the argument which we note does not require holonomy and hence independent of \cref{sec:HolonomyAndRepresentationTheory}. First, we make some reductions in \cref{sec:ReductionToNewInvariantFunctionsAtLevel_q}. In \cref{sec:ApproximatingTheTransferOperator}, we explain how congruence transfer operators can be approximated by convolutions with measures. In \cref{sec:ZariskiDensityAndTraceFieldOfTheReturnTrajectorySubgroups}, we prove that the return trajectory subgroups are Zariski dense and have full trace field $\mathbb K$. This is the main new theorem required to use the expander machinery in \cref{sec:L2FlatteningLemma} to prove a $L^2$-flattening lemma for the approximating measures. The $L^2$-flattening lemma is then used in \cref{sec:SupremumAndLipschitzBounds} to finish the expander machinery part of the argument. Next, we go through the Dolgopyat's method part of the argument in \cref{sec:Dolgopyat'sMethod}. We finish by describing how to convert the uniform spectral bounds with holonomy to uniform exponential mixing of the frame flow.

\subsection*{Acknowledgements}
This work is part of my Ph.D. thesis at Yale University. I am extremely grateful to my advisor Hee Oh for introducing the beautiful field of homogeneous dynamics and suggesting to think about this problem and helpful guidance throughout. I also thank Dale Winter for explaining his work, Wenyu Pan for helpful discussions, and Jialun Li for helpful comments on \cref{thm:ResonanceFreeRegions}.

\section{Preliminaries}
\label{sec:Preliminaries}
We will first introduce the basic setup and fix notations for the rest of the paper.

Let $\mathbb H^n$ be the $n$-dimensional hyperbolic space for $n \geq 2$. We denote by $\langle \cdot, \cdot\rangle$ and $\|\cdot\|$ the inner product and norm respectively on any tangent space of $\mathbb H^n$ induced by the hyperbolic metric. Similarly, we denote by $d$ the distance function on $\mathbb H^n$ induced by the hyperbolic metric. Let $\mathbb K \subset \mathbb R$ be a totally real number field and $\mathcal{O}_{\mathbb K}$ be the corresponding ring of integers. Let $\mathbf{G} < \GL_N$ for some $N \in \mathbb N$ be an algebraic group defined over $\mathbb K$ such that $\mathbf{G}(\mathbb R) \cong \SO(n, 1)$ and $\mathbf{G}^\sigma(\mathbb R)$ is compact for all nontrivial embeddings $\sigma: \mathbb K \hookrightarrow \mathbb R$. Let $G = \mathbf{G}(\mathbb R)^\circ$ which we recognize as the group of orientation preserving isometries of $\mathbb H^n$. Let $\Gamma < G$ be a Zariski dense torsion-free discrete subgroup. Let $o \in \mathbb H^n$ be a reference point and $v_o \in \T^1(\mathbb H^n)$ be a reference tangent vector at $o$. Then we have the stabilizer subgroups $K = \Stab_G(o)$ and $M = \Stab_G(v_o) < K$. Note that $K \cong \SO(n)$ and it is a maximal compact subgroup of $G$ and $M \cong \SO(n - 1)$. Our base hyperbolic manifold is $X = \Gamma \backslash \mathbb H^n \cong \Gamma \backslash G/K$, its unit tangent bundle is $\T^1(X) \cong \Gamma \backslash G/M$ and its frame bundle is $\F(X) \cong \Gamma \backslash G$ which is a principal $\SO(n)$-bundle over $X$ and a principal $\SO(n - 1)$-bundle over $\T^1(X)$. There is a one parameter subgroup of semisimple elements $A = \{a_t: t \in \mathbb R\} < G$, where $C_G(A) = AM$, parametrized such that its canonical right translation action on $G/M$ and $G$ corresponds to the geodesic flow and the frame flow respectively. We choose a left $G$-invariant and right $K$-invariant Riemannian metric on $G$ \cite{Sas58,Mok78} which descends down to the previous hyperbolic metric on $\mathbb H^n \cong G/K$ and again use the notations $\langle \cdot, \cdot\rangle$, $\|\cdot\|$, and $d$ on $G$ and any of its quotient spaces.

To make use of the strong approximation theorem of Weisfeiler \cite{Wei84} later on, we need to work on the simply connected cover $\tilde{\mathbf G}$ endowed with the covering map $\tilde{\pi}: \tilde{\mathbf G} \to \mathbf G$ defined over $\mathbb K$. Let $\tilde{G} = \tilde{\mathbf G}(\mathbb R)$ which is connected and projects down to $\tilde{\pi}(\tilde{G}) = G$. Let $\tilde{\Gamma} < \tilde{G}$ be a Zariski dense subgroup containing the finite central subgroup $\ker(\tilde{\pi})$ as the only torsion elements. To be able to discuss the notion of congruence subgroups, let us suppose that $\tilde{\Gamma} < \tilde{\mathbf G}(\mathcal{O}_{\mathbb K})$. To use the strong approximation theorem, we also need to make the technical assumption that the trace field is $\mathbb Q(\tr(\Ad(\tilde{\Gamma}))) = \mathbb K$. Then we take $\Gamma$ introduced previously to be $\Gamma = \tilde{\pi}(\tilde{\Gamma})$ which also has trace field $\mathbb Q(\tr(\Ad(\Gamma))) = \mathbb K$ by \cite[Corollary 1.4.8]{Mar91}.

\subsection{Convex cocompact groups}
Denote $\partial_\infty\mathbb H^n$ to be the boundary at infinity and $\overline{\mathbb H^n} = \mathbb H^n \cup \partial_\infty\mathbb H^n$ to be the compactification of $\mathbb H^n$.

\begin{definition}[Limit set]
The \emph{limit set} of $\Gamma$ is the set of limit points $\Lambda(\Gamma) = \lim(\Gamma o) \subset \partial_\infty\mathbb H^n \subset \overline{\mathbb H^n}$.
\end{definition}

\begin{definition}[Radial limit set]
\label{def:RadialLimitSet}
A point $\xi \in \Lambda(\Gamma)$ is called a \emph{radial limit point} of $\Gamma$ if for some geodesic $\xi: \mathbb R \to \mathbb H^n$ with $\lim_{t \to \infty} \xi(t) = \xi$, there exists $r > 0$ and sequences $\{\gamma_k\}_{k \in \mathbb N} \subset \Gamma$ and $\{t_k\}_{k \in \mathbb N} \subset \mathbb R$ such that $d(\gamma_k o, \xi(t_k)) < r$ and $\lim_{k \to \infty} \gamma_k o = \xi$. The \emph{radial limit set} is the set of radial limit points $\Lambda_{\mathrm{r}}(\Gamma) \subset \Lambda(\Gamma)$.
\end{definition}

\begin{definition}[Critical exponent]
The \emph{critical exponent} $\delta_\Gamma$ of $\Gamma$ is the abscissa of convergence of the Poincar\'{e} series $\mathscr{P}_\Gamma(s) = \sum_{\gamma \in \Gamma} e^{-s d(o, \gamma o)}$.
\end{definition}

\begin{remark}
The above definitions are independent of the choice of $o \in \mathbb H^n$.
\end{remark}

\begin{definition}[Convex cocompact]
A torsion-free discrete subgroup $\Gamma < G$ is called \emph{convex cocompact} if the \emph{convex core} $\Core(X) = \Gamma \backslash \Hull(\Lambda(\Gamma)) \subset X$, where $\Hull$ denotes the convex hull, is compact.
\end{definition}

The following is a theorem of Bowditch \cite{Bow93}.

\begin{theorem}
\label{thm:ConvexCocompactIFFLimitSetIsRadial}
A torsion-free discrete subgroup $\Gamma < G$ is convex cocompact if and only if $\Lambda(\Gamma) = \Lambda_{\mathrm{r}}(\Gamma)$.
\end{theorem}

We assume that $\Gamma$ is convex cocompact throughout the paper.

\begin{remark}
In our case, $\delta_\Gamma \in (0, n - 1]$ and coincides with the Hausdorff dimension of $\Lambda(\Gamma)$.
\end{remark}

\subsection{Patterson--Sullivan density}
\label{subsec:Patterson--SullivanDensity}
Let $\{\mu^{\mathrm{PS}}_x: x \in \mathbb H^n\}$ denote the \emph{Patterson--Sullivan density} of $\Gamma$ \cite{Pat76,Sul79}, i.e., the set of finite Borel measures on $\partial_\infty\mathbb H^n$ supported on $\Lambda(\Gamma)$ such that
\begin{enumerate}
\item	$g_*\mu^{\mathrm{PS}}_x = \mu^{\mathrm{PS}}_{gx}$ for all $g \in \Gamma$ and $x \in \mathbb H^n$;
\item	$\frac{d\mu^{\mathrm{PS}}_x}{d\mu^{\mathrm{PS}}_y}(\xi) = e^{\delta_\Gamma \beta_{\xi}(y, x)}$ for all $\xi \in \partial_\infty\mathbb H^n$ and $x, y \in \mathbb H^n$
\end{enumerate}
where $\beta_{\xi}$ denotes the \emph{Busemann function} at $\xi \in \partial_\infty\mathbb H^n$ defined by $\beta_{\xi}(y, x) = \lim_{t \to \infty} (d(\xi(t), y) - d(\xi(t), x))$, where $\xi: \mathbb R \to \mathbb H^n$ is any geodesic such that $\lim_{t \to \infty} \xi(t) = \xi$. We also allow tangent vector arguments for the Busemann function in which case we will use their basepoints in the definition.

\subsection{Bowen--Margulis--Sullivan measure}
\label{subsec:BMS_Measure}
For all $u \in \T^1(\mathbb H^n)$, let $u^+$ and $u^-$ denote its forward and backward limit points. Using the Hopf parametrization via the homeomorphism $G/M \cong \T^1(\mathbb H^n) \to \{(u^+, u^-) \in \partial_\infty\mathbb H^n \times \partial_\infty\mathbb H^n: u^+ \neq u^-\} \times \mathbb R$ defined by $u \mapsto (u^+, u^-, t = \beta_{u^-}(o, u))$, we define the \emph{Bowen--Margulis--Sullivan (BMS) measure} $m^{\mathrm{BMS}}$ on $G/M$ \cite{Mar04,Bow71,Kai90} by
\begin{align*}
dm^{\mathrm{BMS}}(u) = e^{\delta_\Gamma \beta_{u^+}(o, u)} e^{\delta_\Gamma \beta_{u^-}(o, u)} \, d\mu^{\mathrm{PS}}_o(u^+) \, d\mu^{\mathrm{PS}}_o(u^-) \, dt
\end{align*}
Note that this definition only depends on $\Gamma$ and not on the choice of reference point $o \in \mathbb H^n$ and moreover $m^{\mathrm{BMS}}$ is left $\Gamma$-invariant. We now define induced measures on other spaces all of which we call the BMS measures and denote by $m^{\mathrm{BMS}}$ by abuse of notation. Since $M$ is compact, we can use the probability Haar measure on $M$ to lift $m^{\mathrm{BMS}}$ to a right $M$-invariant measure on $G$. By left $\Gamma$-invariance, $m^{\mathrm{BMS}}$ now descends to a measure on $\Gamma \backslash G$. By right $M$-invariance, $m^{\mathrm{BMS}}$ descends once more to a measure on $\Gamma \backslash G/M$. We normalize the above measures such that $m^{\mathrm{BMS}}(\Gamma \backslash G) = 1$. It can be checked that the BMS measures are right $A$-invariant. We denote $\Omega = \supp(m^{\mathrm{BMS}}) \subset \Gamma \backslash G/M$ which is compact since $\Gamma$ is convex cocompact.

\section{Coding of the geodesic flow and its associated cocycles}
\label{sec:CodingTheGeodesicFlow}
In this section, we will first review the required background for Markov sections, symbolic dynamics, and thermodynamic formalism. Then, we introduce the congruence setting, define cocycles, and construct compatible Markov sections.

\subsection{Markov sections}
\label{subsec:MarkovSections}
We will use a Markov section on $\Omega \subset \T^1(X) \cong \Gamma \backslash G/M$ to obtain a symbolic coding of the geodesic flow. Recall that the geodesic flow on $\T^1(X)$ is Anosov. Let $W^{\mathrm{su}}(w) \subset \T^1(X)$ and $W^{\mathrm{ss}}(w) \subset \T^1(X)$ denote the leaves through $w \in \T^1(X)$ of the strong unstable and strong stable foliations, and $W_{\epsilon}^{\mathrm{su}}(w) \subset W^{\mathrm{su}}(w)$ and $W_{\epsilon}^{\mathrm{ss}}(w) \subset W^{\mathrm{ss}}(w)$ denote the open balls of radius $\epsilon > 0$ with respect to the induced distance functions $d_{\mathrm{su}}$ and $d_{\mathrm{ss}}$, respectively. We use similar notations for the weak unstable and weak stable foliations by replacing `su' with `wu' and `ss' with `ws' respectively. From the Anosov property we obtain a constant $C_{\mathrm{Ano}} > 0$, such that for all $w \in \T^1(X)$, we have
\begin{align*}
d_{\mathrm{su}}(ua_{-t}, va_{-t}) &\leq C_{\mathrm{Ano}}e^{-t}d_{\mathrm{su}}(u, v); & d_{\mathrm{ss}}(ua_t, va_t) &\leq C_{\mathrm{Ano}}e^{-t}d_{\mathrm{ss}}(u, v)
\end{align*}
for all $t \geq 0$, for all $u, v \in W^{\mathrm{su}}(w)$ or for all $u, v \in W^{\mathrm{ss}}(w)$ respectively. From \cite{Rat73}, we recall that there exist $\epsilon_0, \epsilon_0' > 0$ such that for all $w \in \T^1(X)$, $u \in W_{\epsilon_0}^{\mathrm{wu}}(w)$, and $s \in W_{\epsilon_0}^{\mathrm{ss}}(w)$, there exists a unique intersection denoted by
\begin{align}
\label{eqn:BracketOfUandS}
[u, s] = W_{\epsilon_0'}^{\mathrm{ss}}(u) \cap W_{\epsilon_0'}^{\mathrm{wu}}(s)
\end{align}
and moreover $[\cdot, \cdot]$ defines a homeomorphism from $W_{\epsilon_0}^{\mathrm{wu}}(w) \times W_{\epsilon_0}^{\mathrm{ss}}(w)$ onto its image. Subsets $U \subset W_{\epsilon_0}^{\mathrm{su}}(w) \cap \Omega$ and $S \subset W_{\epsilon_0}^{\mathrm{ss}}(w) \cap \Omega$ for some $w \in \Omega$ are called \emph{proper} if $U = \overline{\interior(U)}$ and $S = \overline{\interior(S)}$, where the interiors and closures are with respect to the topology of $W^{\mathrm{su}}(w) \cap \Omega$ and $W^{\mathrm{ss}}(w) \cap \Omega$ respectively. For any proper subsets $U \subset W_{\epsilon_0}^{\mathrm{su}}(w) \cap \Omega$ and $S \subset W_{\epsilon_0}^{\mathrm{ss}}(w) \cap \Omega$ containing some $w \in \Omega$, we call
\begin{align*}
R = [U, S] = \{[u, s] \in \Omega: u \in U, s \in S\} \subset \Omega
\end{align*}
a \emph{rectangle of size $\hat{\delta}$} if $\diam_{d_{\mathrm{su}}}(U), \diam_{d_{\mathrm{ss}}}(S) \leq \hat{\delta}$ for some $\hat{\delta} > 0$, and $w$ the \emph{center} of $R$. For any rectangle $R = [U, S]$, we generalize the notation and define $[v_1, v_2] = [u_1, s_2]$ for all $v_1 = [u_1, s_1] \in R$ and $v_2 = [u_2, s_2] \in R$.

\begin{definition}[Complete set of rectangles]
A set $\mathcal{R} = \{R_1, R_2, \dotsc, R_N\} = \{[U_1, S_1], [U_2, S_2], \dotsc, [U_N, S_N]\}$ for some $N \in \mathbb N$ consisting of rectangles in $\Omega$ is called a \emph{complete set of rectangles of size $\hat{\delta}$} if
\begin{enumerate}
\item \label{itm:MarkovProperty1} $R_j \cap R_k = \varnothing$ for all $1 \leq j, k \leq N$ with $j \neq k$;
\item \label{itm:MarkovProperty2} $\diam_{d_{\mathrm{su}}}(U_j), \diam_{d_{\mathrm{ss}}}(S_j) \leq \hat{\delta}$ for all $1 \leq j \leq N$;
\item \label{itm:MarkovProperty3} $\Omega = \bigcup_{j = 1}^N \bigcup_{t \in [0, \hat{\delta}]} R_j a_t$.
\end{enumerate}
\end{definition}

Henceforth, we fix
\begin{align}
\label{eqn:DeltaHatCondition}
0 < \hat{\delta} < \min\left(1, \epsilon_0, \epsilon_0', \frac{1}{4}\inj(\T^1(X))\right)
\end{align}
where $\inj(\T^1(X))$ denotes the injectivity radius of $\T^1(X)$ and where $\epsilon_0$ and $\epsilon_0'$ are from \cref{eqn:BracketOfUandS}. We also fix $\mathcal{R} = \{R_1, R_2, \dotsc, R_N\} = \{[U_1, S_1], [U_2, S_2], \dotsc, [U_N, S_N]\}$ to be a complete set of rectangles of size $\hat{\delta}$ in $\Omega$. We define
\begin{align*}
R &= \bigsqcup_{j = 1}^N R_j; & U &= \bigsqcup_{j = 1}^N U_j.
\end{align*}
We introduce the distance function $d$ on $U$ defined by
\begin{align*}
d(u, v) =
\begin{cases}
d_{\mathrm{su}}(u, v), & u, v \in U_j \text{ for some } 1 \leq j \leq N \\
1, & \text{otherwise.}
\end{cases}
\qquad
\text{for all $u, v \in U$}.
\end{align*}
We will use $d_{\mathrm{su}}$ whenever further clarity is required. Define the first return time map $\tau: R \to \mathbb R$ by
\begin{align*}
\tau(u) = \inf\{t \in \mathbb R_{>0}: ua_t \in R\} \qquad \text{for all $u \in R$}.
\end{align*}
Note that $\tau$ is constant on $[u, S_j]$ for all $u \in U_j$ and $1 \leq j \leq N$. Define the Poincar\'{e} first return map $\mathcal{P}: R \to R$ by
\begin{align*}
\mathcal{P}(u) = ua_{\tau(u)} \qquad \text{for all $u \in R$}.
\end{align*}
Let $\sigma = (\proj_U \circ \mathcal{P})|_U: U \to U$ be its projection where $\proj_U: R \to U$ is the projection defined by $\proj_U([u, s]) = u$ for all $[u, s] \in R$. We define the \emph{cores}
\begin{align*}
\hat{R} &= \{u \in R: \mathcal{P}^k(u) \in \interior(R) \text{ for all } k \in \mathbb Z\}; \\
\hat{U} &= \{u \in U: \sigma^k(u) \in \interior(U) \text{ for all } k \in \mathbb Z_{\geq 0}\}.
\end{align*}
We note that the cores are both residual subsets (complements of meager sets) of $R$ and $U$ respectively.

\begin{definition}[Markov section]
Let $\hat{\delta} > 0$ and $N \in \mathbb N$. We call a complete set of rectangles $\mathcal{R}$ of size $\hat{\delta}$ a \emph{Markov section} if in addition to \cref{itm:MarkovProperty1,itm:MarkovProperty2,itm:MarkovProperty3}, the following property
\begin{enumerate}
\setcounter{enumi}{3}
\item $[\interior(U_k), \mathcal{P}(u)] \subset \mathcal{P}([\interior(U_j), u])$ and $\mathcal{P}([u, \interior(S_j)]) \subset [\mathcal{P}(u), \interior(S_k)]$ for all $u \in R$ such that $u \in \interior(R_j) \cap \mathcal{P}^{-1}(\interior(R_k)) \neq \varnothing$, for all $1 \leq j, k \leq N$
\end{enumerate}
called the \emph{Markov property}, is satisfied. This can be understood pictorially in \cref{fig:MarkovProperty}.
\end{definition}

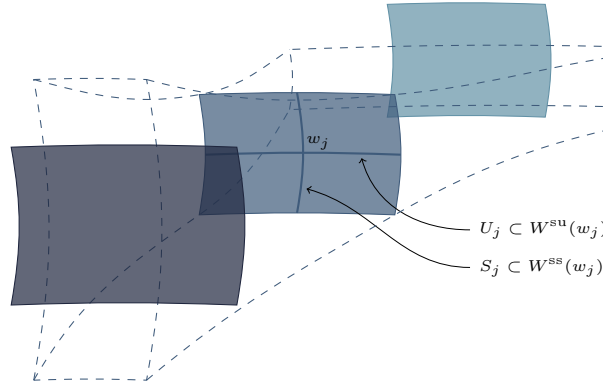
\begin{figure}
\definecolor{front}{RGB}{31, 38, 62}
\definecolor{middle}{RGB}{63,91,123}
\definecolor{back}{RGB}{98,145,166}
\begin{tikzpicture}
\coordinate (E) at (2.7,1.5);
\coordinate (F) at (2.7,3);
\coordinate (G) at (4.8,3);
\coordinate (H) at (4.8,1.5);

\draw[back, fill = back, fill opacity=0.7] (E) to[out=80,in=-80] (F) to[out=2,in=178] (G) to[out=-80,in=80] (H) to[out=178,in=2] (E) -- cycle;

\coordinate (A) at (0.2,0.2);
\coordinate (B) at (0.2,1.8);
\coordinate (C) at (2.8,1.8);
\coordinate (D) at (2.8,0.2);

\draw[middle, fill = middle, fill opacity=0.7] (A) to[out=80,in=-80] (B) to[out=2,in=178] (C) to[out=-80,in=80] (D) to[out=178,in=2] (A) -- cycle;

\coordinate (AB_Mid) at (0.28,1);
\coordinate (CD_Mid) at (2.88,1);

\coordinate (Uj_Label) at (3.8, 0);
\coordinate (Sj_Label) at (3.8, -0.5);

\node[above right, font=\tiny] at (1.5, 0.95) {$w_j$};

\draw[middle,thick] (AB_Mid) to[out=2,in=178] (CD_Mid);
\node[right, font=\tiny] at (Uj_Label) {$U_j \subset W^{\mathrm{su}}(w_j)$};
\draw[->, very thin] (Uj_Label) to[out=180,in=-75] ($0.8*(CD_Mid) + 0.2*(AB_Mid) + (0, -0.05)$);

\coordinate (BC_Mid) at (1.5,1.825);
\coordinate (AD_Mid) at (1.5,0.225);

\draw[middle,thick] (AD_Mid) to[out=80,in=-80] (BC_Mid);
\node[right, font=\tiny] at (Sj_Label) {$S_j \subset W^{\mathrm{ss}}(w_j)$};
\draw[->, very thin] (Sj_Label) to[out=180,in=-15] ($0.8*(AD_Mid) + 0.2*(BC_Mid) + (0.13, 0)$);

\coordinate (A') at (1.4,1.6);
\coordinate (B') at (1.4,2.4);
\coordinate (C') at (7,2.4);
\coordinate (D') at (7,1.6);

\draw[middle,dashed] (A') to[out=80,in=-80] (B') to[out=2,in=178] (C') to[out=-80,in=80] (D') to[out=178,in=2] (A') -- cycle;

\coordinate (A'') at (-2,-2);
\coordinate (B'') at (-2,2);
\coordinate (C'') at (-0.5,2);
\coordinate (D'') at (-0.5,-2);

\draw[middle,dashed] (A'') to[out=80,in=-80] (B'') to[out=2,in=178] (C'') to[out=-80,in=80] (D'') to[out=178,in=2] (A'') -- cycle;

\draw[middle,dashed] (A'') to[out=60,in=215] (A) to[out=35,in=240] (A');
\draw[middle,dashed] (B'') to[out=-15,in=190] (B) to[out=10,in=215] (B');
\draw[middle,dashed] (C'') to[out=-15,in=185] (C) to[out=5,in=180] (C');
\draw[middle,dashed] (D'') to[out=40,in=210] (D) to[out=30,in=190] (D');

\coordinate (I) at (-2.3,-1);
\coordinate (J) at (-2.3,1.1);
\coordinate (K) at (0.7,1.1);
\coordinate (L) at (0.7,-1);

\draw[front, fill = front, fill opacity=0.7] (I) to[out=80,in=-80] (J) to[out=2,in=178] (K) to[out=-80,in=80] (L) to[out=178,in=2] (I) -- cycle;
\end{tikzpicture}
\caption{The Markov property.}
\label{fig:MarkovProperty}
\end{figure}

The existence of Markov sections of arbitrarily small size for Anosov flows was proved by Bowen and Ratner \cite{Bow70,Rat73}. Thus, we assume henceforth that $\mathcal{R}$ is a Markov section.

\subsection{Symbolic dynamics}
\label{subsec:SymbolicDynamics}
Denote $\mathcal A = \{1, 2, \dotsc, N\}$ to be the \emph{alphabet} for the coding given by the Markov section. Define the $N \times N$ \emph{transition matrix} $T$ by
\begin{align*}
T_{j, k} =
\begin{cases}
1, & \interior(R_j) \cap \mathcal{P}^{-1}(\interior(R_k)) \neq \varnothing \\
0, & \text{otherwise}
\end{cases}
\qquad
\text{for all $1 \leq j, k \leq N$}.
\end{align*}
The transition matrix $T$ is \emph{topologically mixing} \cite[Theorem 4.3]{Rat73}, i.e., there exists $N_T \in \mathbb N$ such that $T^{N_T}$ consists only of positive entries. Define the spaces of bi-infinite and infinite \emph{admissible sequences}
\begin{align*}
\Sigma &= \{(\dotsc, x_{-1}, x_0, x_1, \dotsc) \in \mathcal A^{\mathbb Z}: T_{x_j, x_{j + 1}} = 1 \text{ for all } j \in \mathbb Z\}; \\
\Sigma^+ &= \{(x_0, x_1, \dotsc) \in \mathcal A^{\mathbb Z_{\geq 0}}: T_{x_j, x_{j + 1}} = 1 \text{ for all } j \in \mathbb Z_{\geq 0}\}
\end{align*}
respectively. We will use the term \emph{admissible sequences} for finite sequences as well in the natural way. For any $\theta \in (0, 1)$, we can endow $\Sigma$ with the distance function $d_\theta$ defined by $d_\theta(x, y) = \theta^{\inf\{|j| \in \mathbb Z_{\geq 0}: x_j \neq y_j\}}$ for all $x, y \in \Sigma$. We can similarly endow $\Sigma^+$ with a distance function which we also denote by $d_\theta$.

\begin{definition}[Cylinder]
For all admissible sequences $x = (x_0, x_1, \dotsc, x_k)$ of some \emph{length} $\len(x) = k \in \mathbb Z_{\geq 0}$, we define the corresponding \emph{cylinder} to be
\begin{align*}
\mathtt{C}[x] = \{u \in U: \sigma^j(u) \in \interior(U_{x_j}) \text{ for all } 0 \leq j \leq k\}
\end{align*}
with \emph{length} $\len(\mathtt{C}[x]) = k$. We will denote cylinders simply by $\mathtt{C}$ (or other typewriter style letters) when we do not need to specify the corresponding admissible sequence.
\end{definition}

\begin{remark}
For all admissible pairs $(j, k)$, the restricted maps $\sigma|_{\mathtt{C}[j, k]}: \mathtt{C}[j, k] \to \interior(U_k)$, $(\sigma|_{\mathtt{C}[j, k]})^{-1}: \interior(U_k) \to \mathtt{C}[j, k]$, and $\tau|_{\mathtt{C}[j, k]}: \mathtt{C}[j, k] \to \mathbb R$ are Lipschitz.
\end{remark}

By a slight abuse of notation, let $\sigma$ also denote the shift map on $\Sigma$ or $\Sigma^+$. There exist natural continuous surjections $\zeta: \Sigma \to R$ and $\zeta^+: \Sigma^+ \to U$ defined by $\zeta(x) = \bigcap_{j = -\infty}^\infty \overline{\mathcal{P}^{-j}(\interior(R_{x_j}))}$ for all $x \in \Sigma$ and $\zeta^+(x) = \bigcap_{j = 0}^\infty \overline{\sigma^{-j}(\interior(U_{x_j}))}$ for all $x \in \Sigma^+$. Define $\hat{\Sigma} = \zeta^{-1}(\hat{R})$ and $\hat{\Sigma}^+ = (\zeta^+)^{-1}(\hat{U})$. Then the restrictions $\zeta|_{\hat{\Sigma}}: \hat{\Sigma} \to \hat{R}$ and $\zeta^+|_{\hat{\Sigma}^+}: \hat{\Sigma}^+ \to \hat{U}$ are bijective and satisfy $\zeta|_{\hat{\Sigma}} \circ \sigma|_{\hat{\Sigma}} = \mathcal{P}|_{\hat{R}} \circ \zeta|_{\hat{\Sigma}}$ and $\zeta^+|_{\hat{\Sigma}^+} \circ \sigma|_{\hat{\Sigma}^+} = \sigma|_{\hat{U}} \circ \zeta^+|_{\hat{\Sigma}^+}$.

For $\theta \in (0, 1)$ sufficiently close to $1$, the maps $\zeta$ and $\zeta^+$ are Lipschitz \cite[Lemma 2.2]{Bow75} with some Lipschitz constant $C_\theta > 0$. We fix $\theta$ to be any such constant. Let $C^{\Lip(d_\theta)}(\Sigma, \mathbb R)$ denote the space of Lipschitz functions $f: \Sigma \to \mathbb R$ which is a Banach space with the norm $\|f\|_{\Lip(d_\theta)} = \|f\|_\infty + \Lip_{d_\theta}(f)$ where
\begin{align*}
\Lip_{d_\theta}(f) = \sup_{\substack{x, y \in \Sigma\\ \text{such that } x \neq y}} \frac{|f(x) - f(y)|}{d_\theta(x, y)}
\end{align*}
is the Lipschitz seminorm, for all $f \in C^{\Lip(d_\theta)}(\Sigma, \mathbb R)$. We use similar notations for function spaces with domain space $\Sigma^+$ or target space $\mathbb C$. We use the convention henceforth that for any complex-valued function space, we omit the target space $\mathbb C$.

Since $(\tau \circ \zeta)|_{\hat{\Sigma}}$ and $(\tau \circ \zeta^+)|_{\hat{\Sigma}^+}$ are Lipschitz, there are unique Lipschitz extensions $\tau_\Sigma \in C^{\Lip(d_\theta)}(\Sigma, \mathbb R)$ and $\tau_{\Sigma^+} \in C^{\Lip(d_\theta)}(\Sigma^+, \mathbb R)$ respectively. Note that the resulting maps are distinct from $\tau \circ \zeta$ and $\tau \circ \zeta^+$ because they may differ precisely on $x \in \Sigma$ for which $\zeta(x) \in \partial\mathtt{C}$ and $x \in \Sigma^+$ for which $\zeta^+(x) \in \partial\mathtt{C}$ respectively, for some cylinder $\mathtt{C} \subset U$ with $\len(\mathtt{C}) = 1$. Then the previous properties extend to $\zeta(\sigma(x)) = \zeta(x)a_{\tau_\Sigma(x)}$ for all $x \in \Sigma$ and $\zeta^+(\sigma(x)) = \proj_U(\zeta^+(x)a_{\tau_{\Sigma^+}(x)})$ for all $x \in \Sigma^+$.

\subsection{Thermodynamics}
\label{subsec:Thermodynamics}
\begin{definition}[Pressure]
\label{def:Pressure}
For all $f \in C^{\Lip(d_\theta)}(\Sigma, \mathbb R)$, called the \emph{potential}, the \emph{pressure} is defined by
\begin{align*}
\Pr_\sigma(f) = \sup_{\nu \in \mathcal{M}^1_\sigma(\Sigma)}\left\{\int_\Sigma f \, d\nu + h_\nu(\sigma)\right\}
\end{align*}
where $\mathcal{M}^1_\sigma(\Sigma)$ is the set of $\sigma$-invariant Borel probability measures on $\Sigma$ and $h_\nu(\sigma)$ is the measure theoretic entropy of $\sigma$ with respect to $\nu$.
\end{definition}

For all $f \in C^{\Lip(d_\theta)}(\Sigma, \mathbb R)$, there is in fact a unique $\sigma$-invariant Borel probability measure $\nu_f$ on $\Sigma$ which attains the supremum in \cref{def:Pressure} called the \emph{$f$-equilibrium state} \cite[Theorems 2.17 and 2.20]{Bow08} and it satisfies $\nu_f(\hat{\Sigma}) = 1$ \cite[Corollary 3.2]{Che02}.

In particular, we will consider the probability measure $\nu_{-\delta_\Gamma\tau_\Sigma}$ on $\Sigma$ which we will denote simply by $\nu_\Sigma$ and has corresponding pressure $\Pr_\sigma(-\delta_\Gamma\tau_\Sigma) = 0$. We have $\nu_\Sigma(\hat{\Sigma}) = 1$. Define the corresponding probability measure $\nu_R = \zeta_*(\nu_\Sigma)$ on $R$ and note that $\nu_R(\hat{R}) = 1$. Now consider the suspension space $R^\tau = (R \times \mathbb R_{\geq 0})/\mathord{\sim}$ where $\sim$ is the equivalence relation on $R \times \mathbb R_{\geq 0}$ defined by $(u, t + \tau(u)) \sim (\mathcal{P}(u), t)$. Then we have a bijection $R^\tau \to \Omega$ defined by $(u, t) \mapsto ua_t$. We can define the measure $\nu^\tau$ on $R^\tau$ as the product measure $\nu_R \times m^{\mathrm{Leb}}$ on $\{(u, t) \in R \times \mathbb R_{\geq 0}: 0 \leq t < \tau(u)\}$. Then using the aforementioned bijection we have the pushforward measure which, by abuse of notation, we also denote by $\nu^\tau$ on $\T^1(X)$ supported on $\Omega$. By \cite{Sul84} and \cite[Theorem 4.4]{Che02}, we have $m^{\mathrm{BMS}} = \frac{\nu^\tau}{\nu_R(\tau)}$ since they are the unique measure of maximal entropy for the geodesic flow on $\T^1(X)$. Finally, we define the probability measure $\nu_U = (\proj_U)_*(\nu_R)$. Note that $\nu_U(\hat{U}) = 1$ and $\nu_U(\tau) = \nu_R(\tau)$.

\subsection{Cocycles and compatible Markov sections}
\label{subsec:CocyclesAndCompatibleMarkovSections}
Noting that $\T^1(\mathbb H^n)$ is a locally isometric cover of $\T^1(X)$, for all $j \in \mathcal{A}$, choose homeomorphic lifts $\mathsf{R}_j = [\mathsf{U}_j, \mathsf{S}_j] \subset \T^1(\mathbb H^n) \cong G/M$ of $R_j$. Define $\mathsf{R} = \bigsqcup_{j = 1}^N \mathsf{R}_j$ and $\mathsf{U} = \bigsqcup_{j = 1}^N \mathsf{U}_j$. For all $u \in R$, let $\tilde{u} \in \mathsf{R}$ denote the unique lift in $\mathsf{R}$. We can also lift the Poincar\'{e} first return map to $\mathcal{P}: \Gamma \mathsf{R} \to \Gamma \mathsf{R}$ defined by $\mathcal{P}(\gamma \tilde{u}) = \gamma \tilde{u}a_{\tau(u)}$ for all $\gamma \in \Gamma$ and $u \in R$.

\begin{definition}[Cocycle]
The \emph{cocycle} is a map $\mathtt{c}: R \to \Gamma$ such that for all $u \in R$, we have $\mathcal{P}(\tilde{u}) \in \mathtt{c}(u)\mathsf{R}$.
\end{definition}

The following lemma is proved similar to \cite[Lemma 2.14]{OW16}.

\begin{lemma}
\label{lem:CocyclesLocallyConstant}
The cocycle $\mathtt{c}$ is \emph{locally constant}, i.e., if $u_1, u_2 \in R_x \cap \mathcal{P}^{-1}(R_y)$ for some $x, y \in \mathcal{A}$, then $\mathtt{c}(u_1) = \mathtt{c}(u_2)$.
\end{lemma}

\begin{corollary}
\label{cor:CocyclesLocallyConstantCorollary}
If $u_1, u_2 \in \mathtt{C}$ for some cylinder $\mathtt{C} \subset U$ with $\len(\mathtt{C}) = k \in \mathbb N$, then $\mathtt{c}^k(u_1) = \mathtt{c}^k(u_2)$.
\end{corollary}

For all ideals $\mathfrak{q} \subset \mathcal{O}_{\mathbb K}$, we have the canonical quotient map $\pi_{\mathfrak{q}}: \tilde{\mathbf{G}}(\mathcal{O}_{\mathbb K}) \to \tilde{\mathbf{G}}(\mathcal{O}_{\mathbb K}/\mathfrak{q})$ and we define the principal congruence subgroup of level $\mathfrak{q}$ to be $\ker(\pi_{\mathfrak{q}})$. We would like to define the congruence subgroup of $\tilde{\Gamma}$ of level $\mathfrak{q}$ to be the normal subgroup $\tilde{\Gamma}_{\mathfrak{q}} = \ker(\pi_{\mathfrak{q}}|_{\tilde{\Gamma}}) \vartriangleleft \tilde{\Gamma}$. However, we make a minor modification and assume as before that $\tilde{\Gamma}_{\mathfrak{q}} \vartriangleleft \tilde{\Gamma}$ contains $\ker(\tilde{\pi})$ as the only torsion elements, i.e., we define $\tilde{\Gamma}_{\mathfrak{q}} = \langle\ker(\pi_{\mathfrak{q}}|_{\tilde{\Gamma}}), \ker(\tilde{\pi})\rangle \vartriangleleft \tilde{\Gamma}$. Again we define $\Gamma_{\mathfrak{q}} = \tilde{\pi}(\tilde{\Gamma}_{\mathfrak{q}})$. For all \emph{nontrivial} $\mathfrak{q} \subset \mathcal{O}_{\mathbb K}$, also define the finite group $F_{\mathfrak{q}} = \Gamma_{\mathfrak{q}} \backslash \Gamma \cong \tilde{\Gamma}_{\mathfrak{q}} \backslash \tilde{\Gamma}$. By the strong approximation theorem of Weisfeiler (see \cite[Theorem 10.1]{Wei84} and its proof), there exists a nontrivial proper ideal $\mathfrak{q}_0 \subset \mathcal{O}_{\mathbb K}$ such that for all ideals $\mathfrak{q} \subset \mathcal{O}_{\mathbb K}$ coprime to $\mathfrak{q}_0$, the map $\pi_{\mathfrak{q}}|_{\tilde{\Gamma}}$ is in fact surjective and hence induces the isomorphism $\overline{\pi_{\mathfrak{q}}|_{\tilde{\Gamma}}}: F_{\mathfrak{q}} \to \tilde{G}_{\mathfrak{q}}$ where we define the finite group $\tilde{G}_{\mathfrak{q}} = \pi_{\mathfrak{q}}(\ker(\tilde{\pi})) \backslash \tilde{\mathbf{G}}(\mathcal{O}_{\mathbb K}/\mathfrak{q})$. Without loss of generality, we assume that the same ideal $\mathfrak{q}_0 \subset \mathcal{O}_{\mathbb K}$ is sufficient to apply both the strong approximation theorem and \cite[Corollary 6]{GV12} for the return trajectory subgroups introduced in \cref{sec:ZariskiDensityAndTraceFieldOfTheReturnTrajectorySubgroups}.

\begin{definition}[Congruence cocycle]
For all nontrivial ideals $\mathfrak{q} \subset \mathcal{O}_{\mathbb K}$, define the \emph{congruence cocycle} $\mathtt{c}_{\mathfrak{q}}: R \to F_{\mathfrak{q}}$ by $\mathtt{c}_{\mathfrak{q}}(u) = \Gamma_{\mathfrak{q}}\mathtt{c}(u)$ for all $u \in R$.
\end{definition}

Let $\mathfrak{q} \subset \mathcal{O}_{\mathbb K}$ be a nontrivial ideal. Denote $X_\mathfrak{q} = \Gamma_\mathfrak{q} \backslash \mathbb H^n$ to be the \emph{congruence cover} of $X$ of level $\mathfrak{q}$. We have the isometries $\T^1(X_\mathfrak{q}) \cong \Gamma_\mathfrak{q} \backslash G/M$ and $\F(X_\mathfrak{q}) \cong \Gamma_\mathfrak{q} \backslash G$. Recall that $m^{\mathrm{BMS}}$ is left $\Gamma$-invariant, so in particular it is left $\Gamma_\mathfrak{q}$-invariant. Thus, it descends to the BMS measure $m^{\mathrm{BMS}}_\mathfrak{q}$ on $\Gamma_\mathfrak{q} \backslash G$ which, by right $M$-invariance, descends once more to the BMS measure $m^{\mathrm{BMS}}_\mathfrak{q}$ on $\Gamma_\mathfrak{q} \backslash G/M$ by abuse of notation. Note that $m^{\mathrm{BMS}}_\mathfrak{q}(\Gamma_\mathfrak{q} \backslash G/M) = \#F_{\mathfrak{q}}$. Define $p_\mathfrak{q}: \T^1(X_\mathfrak{q}) \to \T^1(X)$ to be the locally isometric covering map and $\Omega_\mathfrak{q} = p_\mathfrak{q}^{-1}(\Omega) = \supp\bigl(m^{\mathrm{BMS}}_\mathfrak{q}\bigr) \subset \Gamma_\mathfrak{q} \backslash G/M$. Clearly,
\begin{align*}
\mathcal{R}^\mathfrak{q} = \{g\mathsf{R}_x \subset \Omega_\mathfrak{q}: x \in \mathcal{A}, g \in F_\mathfrak{q}\}
\end{align*}
is a complete set of rectangles of size $\hat{\delta}$. Define $R^\mathfrak{q} = F_\mathfrak{q}\mathsf{R}$ and $U^\mathfrak{q} = F_\mathfrak{q}\mathsf{U}$. We make the identification $R \times F_\mathfrak{q} \cong R^\mathfrak{q}$ via the isometry $(u, g) \mapsto g\tilde{u}$. Similarly $U \times F_\mathfrak{q} \cong U^\mathfrak{q}$. In fact, $\mathcal{R}^\mathfrak{q}$ is a Markov section with the first return time map $\tau_\mathfrak{q} = \tau \circ p_\mathfrak{q}$ and the Poincar\'{e} first return map $\mathcal{P}_\mathfrak{q}: R^\mathfrak{q} \to R^\mathfrak{q}$ defined by $\mathcal{P}_\mathfrak{q}(u, g) = (\mathcal{P}(u), g\mathtt{c}_\mathfrak{q}(u))$ for all $(u, g) \in R^\mathfrak{q}$. The induced shift map $\sigma_\mathfrak{q}: U^\mathfrak{q} \to U^\mathfrak{q}$ is defined similarly by $\sigma_\mathfrak{q}(u, g) = (\sigma(u), g\mathtt{c}_\mathfrak{q}(u))$ for all $(u, g) \in U^\mathfrak{q}$. Define the measure $\nu_{R^\mathfrak{q}} = \nu_R \times m_{F_\mathfrak{q}}$ on $R^\mathfrak{q}$ where $m_{F_\mathfrak{q}}$ is the counting measure on $F_\mathfrak{q}$. Like $R^\tau$ and $\nu^\tau$, we define $R^{\mathfrak{q}, \tau}$ and the measure $\nu^{\mathfrak{q}, \tau}$ in a similar fashion. Again, we regard $\nu^{\mathfrak{q}, \tau}$ as a measure on $\T^1(X_\mathfrak{q})$ supported on $\Omega_\mathfrak{q}$. Then $\nu^{\mathfrak{q}, \tau}(\T^1(X_\mathfrak{q})) = \#F_\mathfrak{q} \cdot \nu_R(\tau)$ and so $m^{\mathrm{BMS}}_\mathfrak{q} = \frac{\nu^{\mathfrak{q}, \tau}}{\nu_R(\tau)}$.

\section{Holonomy and representation theory}
\label{sec:HolonomyAndRepresentationTheory}
In this section, we define the holonomy which is required in addition to the Markov section to deal with the frame flow. Since it is $M$-valued, we are naturally lead to consider $L^2(M)$ and so we also cover the required representation theory.

We do not have a Markov section available for the frame flow. Thus, similar to $\tau$ and $\mathtt{c}$, we need a map $\vartheta$ which ``keeps track of the $M$-coordinate''. Similar to defining $\mathtt{c}$, we first require an appropriate choice of a section $F$ on $R$ of the frame bundle $\F(X)$ over $\T^1(X)$. Let $w_j$ be the center of $R_j$ for all $j \in \mathcal{A}$. For convenience later on, we will actually define a \emph{smooth} section
\begin{align*}
F: \bigsqcup_{j = 1}^N [W_{\epsilon_0}^{\mathrm{su}}(w_j), W_{\epsilon_0}^{\mathrm{ss}}(w_j)] \to \F(X)
\end{align*}
where without loss of generality we assume $\epsilon_0$ is sufficiently small so that the union is indeed a disjoint union. Define $N^+ < G$ and $N^- < G$ to be the expanding and contracting horospherical subgroups, i.e.,
\begin{align*}
N^\pm = \Big\{n^\pm \in G: \lim_{t \to \pm\infty} a_tn^\pm a_{-t} = e\Big\}.
\end{align*}
First we choose arbitrary frames $F(w_j) \in \F(X)$ based at the tangent vector $w_j \in \T^1(X)$ for all $j \in \mathcal{A}$. Then we extend the section $F$ such that for all $j \in \mathcal{A}$ and $u, u' \in W_{\epsilon_0}^{\mathrm{su}}(w_j)$, we have that the frames $F(u)$ and $F(u')$ are backwards asymptotic, i.e., $\lim_{t \to -\infty} d(F(u)a_t, F(u')a_t) = 0$. Then we must have $F(u') = F(u)n^+$ for some unique $n^+ \in N^+$. We again extend the section $F$ such that for all $j \in \mathcal{A}$, $u \in W_{\epsilon_0}^{\mathrm{su}}(w_j)$, and $s, s' \in W_{\epsilon_0}^{\mathrm{ss}}(w_j)$, we have that the frames $F([u, s])$ and $F([u, s'])$ are forwards asymptotic, i.e., $\lim_{t \to \infty} d(F([u, s])a_t, F([u, s'])a_t) = 0$. Then we must have $F([u, s']) = F([u, s])n^-$ for some unique $n^- \in N^-$. This completes the construction.

\begin{definition}[Holonomy]
The \emph{holonomy} is a map $\vartheta: R \to M$ such that for all $u \in R$, we have $F(u)a_{\tau(u)} = F(\mathcal{P}(u))\vartheta(u)^{-1}$.
\end{definition}

We recall \cite[Lemma 4.2]{SW20} here.

\begin{lemma}
The holonomy $\vartheta$ is constant on $[u, S_j]$ for all $u \in U_j$ and $j \in \mathcal{A}$.
\end{lemma}

Let $\mathfrak{q} \subset \mathcal{O}_{\mathbb K}$ be a nontrivial ideal. Define $\Omega_{\mathfrak{q}, \vartheta} = \supp\bigl(m^{\mathrm{BMS}}_\mathfrak{q}\bigr) \subset \Gamma_\mathfrak{q} \backslash G$. Denote $R^\vartheta \subset \F(X)$ and $R^{\mathfrak{q}, \vartheta} \subset \F(X_\mathfrak{q})$ to be the subsets of frames over $R$ and $R^\mathfrak{q}$ respectively. Similarly define $U^\vartheta$ and $U^{\mathfrak{q}, \vartheta}$. We lift the section $F$ to $F^\mathfrak{q}: \bigsqcup_{g \in F_\mathfrak{q}, j \in \mathcal{A}} [W_{\epsilon_0}^{\mathrm{su}}(g\overline{w}_j), W_{\epsilon_0}^{\mathrm{ss}}(g\overline{w}_j)] \to \F(X_\mathfrak{q})$ in the natural way. Via these sections, we have the identifications $R^\vartheta \cong R \times M$ and $R^{\mathfrak{q}, \vartheta} \cong R \times F_\mathfrak{q} \times M$. Similarly $U^\vartheta \cong U \times M$ and $U^{\mathfrak{q}, \vartheta} \cong U \times F_\mathfrak{q} \times M$. Define the measures $\nu_{R^\vartheta}$ on $R^\vartheta$ and $\nu_{R^{\mathfrak{q}, \vartheta}}$ on $R^{\mathfrak{q}, \vartheta}$ by lifting $\nu_{R}$ and $\nu_{R^{\mathfrak{q}}}$ using the probability Haar measure on $M$. Define the suspension space $R^{\mathfrak{q}, \vartheta, \tau} = R^{\mathfrak{q}, \vartheta} \times \mathbb R_{\geq 0}/{\sim}$ where $\sim$ is the equivalence relation on $R^{\mathfrak{q}, \vartheta} \times \mathbb R_{\geq 0}$ defined by $(u, g, m, t + \tau(u)) \sim (\mathcal{P}(u), g\mathtt{c}_{\mathfrak{q}}(u), \vartheta(u)^{-1}m, t)$. Like $\nu^\tau$, define the measure $\nu^{\mathfrak{q}, \vartheta, \tau}$ on $R^{\mathfrak{q}, \vartheta, \tau}$. We regard $\nu^{\mathfrak{q}, \vartheta, \tau}$ as a measure on $\F(X_\mathfrak{q})$ supported on $\Omega_{\mathfrak{q}, \vartheta}$. Then $\nu^{\mathfrak{q}, \vartheta, \tau}(\F(X_\mathfrak{q})) = \nu_{R^\mathfrak{q}}(\tau_\mathfrak{q}) = \#F_\mathfrak{q} \cdot \nu_R(\tau)$ and $m^{\mathrm{BMS}}_\mathfrak{q} = \frac{\nu^{\mathfrak{q}, \vartheta, \tau}}{\nu_R(\tau)}$.

Let $\mathfrak{q} \subset \mathcal{O}_{\mathbb K}$ be a nontrivial ideal. We need to deal with the function space $C(U^{\mathfrak{q}, \vartheta})$. We note that
\begin{align*}
&C(U^{\mathfrak{q}, \vartheta}) \cong C(U \times F_\mathfrak{q} \times M) \cong C(U, C(F_\mathfrak{q} \times M)) \\
\subset{}&C(U, L^2(F_\mathfrak{q} \times M)) \cong C(U, L^2(F_\mathfrak{q}) \otimes L^2(M)).
\end{align*}
Define $\varrho: M \to \U(L^2(M))$ to be the unitary left regular representation, i.e., $\varrho(h)(\phi)(m) = \phi(h^{-1}m)$ for all $m \in M$, $\phi \in L^2(M)$, and $h \in M$. Let $\widehat{M}$ denote the unitary dual of $M$. We denote the trivial irreducible representation by $1 \in \widehat{M}$. Define $\widehat{M}_0 = \widehat{M} \setminus \{1\}$. By the Peter--Weyl theorem, we obtain an orthogonal Hilbert space decomposition
\begin{align*}
L^2(M) = \operatorname*{\widehat{\bigoplus}}_{\rho \in \widehat{M}} {V_\rho}^{\oplus \dim(\rho)}
\end{align*}
corresponding to the decomposition $\varrho = \operatorname*{\widehat{\bigoplus}}_{\rho \in \widehat{M}} \rho^{\oplus \dim(\rho)}$. For brevity, we denote the isotypic components ${V_\rho}^{\oplus \dim(\rho)}$ simply by $V_\rho^\oplus$ for all $\rho \in \widehat{M}$.

Let $\rho \in \widehat{M}$, $b \in \mathbb R$, and $\mathfrak{q} \subset \mathcal{O}_{\mathbb K}$ be a nontrivial ideal. We define the tensored unitary representation $\rho_b: AM \to \U(V_\rho)$ by
\begin{align*}
\rho_b(a_tm)(z) = e^{-ibt}\rho(m)(z) \qquad \text{for all $z \in V_\rho$, $t \in \mathbb R$, and $m \in M$}.
\end{align*}
We define the tensored unitary representation $\rho_{b, \mathfrak{q}}: AM \to \U(L^2(F_\mathfrak{q}) \otimes V_\rho)$ by
\begin{align*}
\rho_{b, \mathfrak{q}}(a_tm)(z) = (\Id_{L^2(F_\mathfrak{q})} \otimes \rho_b(a_tm))(z)
\end{align*}
for all $z \in L^2(F_\mathfrak{q}) \otimes V_\rho$, $t \in \mathbb R$, and $m \in M$. The reason for simply using $\Id_{L^2(F_\mathfrak{q})}$ for the action on the $L^2(F_\mathfrak{q})$ tensor component is precisely because cocycles are locally constant which will become clear in \cref{subsec:ChangesRequiredForLemma9.10}.

We introduce some notations related to Lie algebras. We denote Lie algebras corresponding to Lie groups by the corresponding Fraktur letters, e.g., $\mathfrak{a} = \T_e(A), \mathfrak{m} = \T_e(M), \mathfrak{n}^+ = \T_e(N^+)$, and $\mathfrak{n}^- = \T_e(N^-)$. For any unitary representation $\rho: M \to \U(V)$ for some Hilbert space $V$, we denote the differential at $e \in M$ by $d\rho = (d\rho)_e: \mathfrak{m} \to \mathfrak{u}(V)$, and define the norm
\begin{align*}
\|\rho\| = \sup_{\substack{z \in \mathfrak{m}\\ \text{such that } \|z\| = 1}} \|d\rho(z)\|_{\mathrm{op}}
\end{align*}
and similarly for any unitary representation $\rho: AM \to \U(V)$.

\begin{remark}
The norms remain the same if we replace $V_\rho$ with $V_\rho^\oplus$ since the $M$-action is identical across all components.
\end{remark}

We recall \cite[Lemma 4.3]{SW20} here which records some useful facts regarding the Lie theoretic norms.

\begin{lemma}
\label{lem:LieTheoreticNormBounds}
For all $b \in \mathbb R$ and $\rho \in \widehat{M}$, we have
\begin{align*}
\sup_{a \in A, m \in M} \sup_{\substack{z \in \T_{am}(AM)\\ \textnormal{such that } \|z\| = 1}} \|(d\rho_b)_{am}(z)\|_{\mathrm{op}} = \|\rho_b\|
\end{align*}
and $\max(|b|, \|\rho\|) \leq \|\rho_b\| \leq |b| + \|\rho\|$.
\end{lemma}

The \emph{source} of the oscillations needed in Dolgopyat's method is provided by the \emph{local non-integrability condition (LNIC)} (see \cite{SW20} for details) and the oscillations are \emph{propagated} when $\rho_b$ has a sufficiently large norm. This occurs exactly when $|b|$ is sufficiently large or $\rho \in \widehat{M}$ is nontrival. Let $b_0 > 0$ which we fix later. We define
\begin{align*}
\widehat{M}_0(b_0) = \{(b, \rho) \in \mathbb R \times \widehat{M}: |b| > b_0 \text{ or } \rho \neq 1\}.
\end{align*}
We fix some related constants. Fix $\delta_{\varrho} = \inf_{b \in \mathbb R, \rho \in \widehat{M}_0} \|\rho_b\| \geq \inf_{\rho \in \widehat{M}_0} \|\rho\|$. Note that $\delta_{\varrho} > 0$ as $M$ is compact. Furthermore, we can deduce that $\inf_{(b, \rho) \in \widehat{M}_0(b_0)} \|\rho_b\| \geq \min(b_0, \delta_{\varrho})$. Hence we fix $\delta_{1, \varrho} = \min(1, \delta_{\varrho})$.

The Killing form $B$ on $\mathfrak{m}$ is nondegenerate and negative definite because $M$ is a compact semisimple Lie group. We denote the corresponding inner product and norm on both $\mathfrak{m}$ and $\mathfrak{m}^*$ by $\langle \cdot, \cdot \rangle_B$ and $\|\cdot\|_B$. By construction of the Riemannian metric on $G$, the induced inner product on $\mathfrak{m}$ satisfies $\langle \cdot, \cdot \rangle_B = C_B\langle \cdot, \cdot \rangle$ for some constant $C_B > 0$.

We generalize \cite[Lemma 4.4]{SW20} for our setting in the following lemma.

\begin{lemma}
\label{lem:maActionLowerBound}
There exists $\delta > 0$ such that for all $b \in \mathbb R$, $\rho \in \widehat{M}$, nontrivial ideals $\mathfrak{q} \subset \mathcal{O}_{\mathbb K}$, and $\omega \in L^2(F_{\mathfrak{q}}) \otimes V_\rho^\oplus$ with $\|\omega\|_2 = 1$, there exists $z \in \mathfrak{a} \oplus \mathfrak{m}$ with $\|z\| = 1$ such that $\|d\rho_{b, \mathfrak{q}}(z)(\omega)\|_2 \geq \delta \|\rho_b\|$.
\end{lemma}

\begin{proof}
Fix $\delta = \frac{1}{2}$ if $M$ is trivial and $\delta = \frac{1}{2 \dim(\mathfrak{m})}$ otherwise. Let $b \in \mathbb R$, $\rho \in \widehat{M}$, $\mathfrak{q} \subset \mathcal{O}_{\mathbb K}$ be a nontrivial ideal, and $\omega = \sum_{g \in F_{\mathfrak{q}}} \delta_g \otimes \omega_g \in L^2(F_{\mathfrak{q}}) \otimes V_\rho^\oplus$ with $\|\omega\|_2^2 = \sum_{g \in F_{\mathfrak{q}}} \|\omega_g\|_2^2 = 1$. For any $z \in \mathfrak{a} \subset \mathfrak{a} \oplus \mathfrak{m}$ with $\|z\| = 1$, we have
\begin{align*}
\|d\rho_{b, \mathfrak{q}}(z)(\omega)\|_2 = \left\|\sum_{g \in F_{\mathfrak{q}}} \delta_g \otimes d\rho_b(z)(\omega_g)\right\|_2 = \left(\sum_{g \in F_{\mathfrak{q}}} \|ib\omega_g\|_2^2\right)^{\frac{1}{2}} = |b|.
\end{align*}
If $M$ is trivial, then $|b| = \|\rho_b\| \geq \delta \|\rho_b\|$ so the lemma follows. Otherwise, first consider the case $|b| \geq \|\rho\|$. By \cref{lem:LieTheoreticNormBounds}, we have $|b| \geq \frac{1}{2}(|b| + \|\rho\|) \geq \delta \|\rho_b\|$ which proves the lemma in this case.

Now consider the case $|b| \leq \|\rho\|$. By \cref{lem:LieTheoreticNormBounds}, we have $\|\rho_b\| \leq 2\|\rho\|$. Let $\Phi_\rho$ be the set of weights corresponding to the Lie algebra representation $d\rho$ and $\lambda \in \Phi_\rho$ be the highest weight. By the same argument as in \cite[Lemma 4.4]{SW20}, we have
\begin{align}
\label{eqn:dRhoZ_OperatorNormBound}
\|d\rho(z)\|_{\mathrm{op}} \leq \max_{\eta \in \Phi_\rho} \|\eta\|_B \|z\|_B \leq C_B \|\lambda\|_B
\end{align}
for all $z \in \mathfrak{m} \subset \mathfrak{a} \oplus \mathfrak{m}$ with $\|z\| = 1$. Hence, $\|\rho\| \leq C_B \|\lambda\|_B$ which implies $\|\rho_b\| \leq 2C_B\|\lambda\|_B$. Now, with respect to the inner product $\langle \cdot, \cdot \rangle_B$, let $(z_1, z_2, \dotsc, z_{\dim(\mathfrak{m})})$ be an orthonormal basis of $\mathfrak{m}$ so that it is its own dual basis. Then the negative Casimir element in the center of the universal enveloping algebra of $\mathfrak{m}$ is given by $\varsigma = \sum_{j = 1}^{\dim(\mathfrak{m})} z_j^2 \in Z(\mathfrak{m}) \subset U(\mathfrak{m})$. Its action on $V_\rho$ via $d\rho$ and hence also via $d\rho_b$ is simply by the scalar $\|\lambda\|_B^2 + 2\langle \lambda, \upsilon \rangle_B$ where $\upsilon = \frac{1}{2}\sum_{\eta \in R^+} \eta$ and $R^+$ is the set of positive roots. But $\langle \lambda, \upsilon \rangle_B \geq 0$ since $\lambda \in \Phi_\rho$ is the highest weight. Thus, we have
\begin{align*}
\|d\rho_{b, \mathfrak{q}}(\varsigma)(\omega)\|_2 = \left\|\sum_{g \in F_{\mathfrak{q}}} \delta_g \otimes d\rho_b(\varsigma)(\omega_g)\right\|_2 \geq  \|\lambda\|_B^2 \left\|\sum_{g \in F_{\mathfrak{q}}} \delta_g \otimes \omega_g\right\|_2 = \|\lambda\|_B^2.
\end{align*}
This implies that $\sum_{j = 1}^{\dim(\mathfrak{m})} \|d\rho_{b, \mathfrak{q}}(z_j^2)(\omega)\|_2 \geq \|\lambda\|_B^2$. Hence, there exists $z_0 \in \{z_1, z_2, \dotsc, z_{\dim(\mathfrak{m})}\}$ such that $\|d\rho_{b, \mathfrak{q}}(z_0^2)(\omega)\|_2 \geq \frac{\|\lambda\|_B^2}{\dim(\mathfrak{m})}$. Using $\|z_0\|_B = 1$ and a similar bound as in \cref{eqn:dRhoZ_OperatorNormBound}, we have $\|d\rho_{b, \mathfrak{q}}(z_0)(\omega)\|_2 \geq \frac{\|\lambda\|_B}{\dim(\mathfrak{m})}$. Let $z = \frac{z_0}{\|z_0\|} \in \mathfrak{m} \subset \mathfrak{a} \oplus \mathfrak{m}$ so that $\|z\| = 1$. Along with the above bound $\|\rho_b\| \leq 2C_B\|\lambda\|_B$, we have
\begin{align*}
\|d\rho_{b, \mathfrak{q}}(z)(\omega)\|_2 \geq \frac{\|\lambda\|_B}{\dim(\mathfrak{m})\|z_0\|} \geq \frac{1}{2 \dim(\mathfrak{m})}\|\rho_b\| \geq \delta \|\rho_b\|
\end{align*}
which proves the lemma in this case also.
\end{proof}

Fix $\varepsilon_1 > 0$ to be the $\delta$ provided by \cref{lem:maActionLowerBound}.

\section{Congruence transfer operators with holonomy and their uniform spectral bounds}
\label{sec:CongruenceTransferOperatorsWithHolonomyAndTheirUniformSpectralBounds}
The goal of this section is to define the congruence transfer operators with holonomy and state the main technical theorems about their uniform spectral bounds.

\subsection{Modified constructions using the smooth structure on $G$}
\label{sec:ModifiedConstructionsUsingTheSmoothStructureOnG}

Since we would like to follow \cite{SW20} to use Dolgopyat's method, we need to enlarge $U$ to an open set in the strong unstable foliation which allows us to use the smooth structure from $G$ and define smooth counterparts to $\sigma$, $\tau$, $\mathtt{c}$, and $\vartheta$. Except for $\mathtt{c}$, we recall the notations and refer to \cite[Subsection 5.1]{SW20} for details of the constructions.

Thanks to \cite[Lemma 1.2]{Rue89}, there exist open sets $U_j \subset \tilde{U}_j$ such that $\overline{\tilde{U}_j} \subset W_{\epsilon_0}^{\mathrm{su}}(w_j)$ with $\diam_{d_{\mathrm{su}}}(\tilde{U}_j) \leq \hat{\delta}$ for all $j \in \mathcal{A}$ such that for all admissible pairs $(j, k)$, the inverse $(\sigma|_{\mathtt{C}[j, k]})^{-1}: \interior(U_k) \to \mathtt{C}[j, k]$ can be extended to a smooth injective map $\sigma^{-(j, k)}: \tilde{U}_k \to \tilde{U}_j$. We define $\tilde{U} = \bigsqcup_{j = 1}^N \tilde{U}_j$. Also define the measure $\nu_{\tilde{U}}$ on $\tilde{U}$ simply by $\nu_{\tilde{U}}(B) = \nu_U(B \cap U)$ for all Borel sets $B \subset \tilde{U}$. Let $j \in \mathbb Z_{\geq 0}$ and $\alpha = (\alpha_0, \alpha_1, \dotsc, \alpha_j)$ be an admissible sequence. Define $\sigma^{-\alpha} = \sigma^{-(\alpha_0, \alpha_1)} \circ \sigma^{-(\alpha_1, \alpha_2)} \circ \cdots \circ \sigma^{-(\alpha_{j - 1}, \alpha_j)}: \tilde{U}_{\alpha_j} \to \tilde{U}_{\alpha_0}$ if $j > 0$ and $\sigma^{-\alpha} = \Id_{\tilde{U}_{\alpha_0}}$ if $j = 0$. Define the cylinder $\tilde{\mathtt{C}}[\alpha] = \sigma^{-\alpha}(\tilde{U}_{\alpha_j}) \supset \mathtt{C}[\alpha]$. Define the smooth maps $\sigma^\alpha = (\sigma^{-\alpha})^{-1}: \tilde{\mathtt{C}}[\alpha] \to \tilde{U}_{\alpha_j}$.

Let $(j, k)$ be an admissible pair. The maps $\tau|_{\mathtt{C}[j, k]}$ and $\vartheta|_{\mathtt{C}[j, k]}$ extend to smooth maps $\tau_{(j, k)}: \tilde{\mathtt{C}}[j, k] \to \mathbb R$ and $\vartheta^{(j, k)}: \tilde{\mathtt{C}}[j, k] \to M$. Since $\mathtt{c}$ is locally constant by \cref{lem:CocyclesLocallyConstant}, we can extend $\mathtt{c}|_{\mathtt{C}[j, k]}$ to $\mathtt{c}^{(j, k)}: \tilde{\mathtt{C}}[j, k] \to \Gamma$ as a constant map. For all $k \in \mathbb N$ and admissible sequences $\alpha = (\alpha_0, \alpha_1, \dotsc, \alpha_k)$, we define the smooth maps $\tau_\alpha: \tilde{\mathtt{C}}[\alpha] \to \mathbb R_{>0}$, $\mathtt{c}^\alpha: \tilde{\mathtt{C}}[\alpha] \to \Gamma$, $\vartheta^\alpha: \tilde{\mathtt{C}}[\alpha] \to M$, and $\Phi^\alpha: \tilde{\mathtt{C}}[\alpha] \to AM$ by
\begin{align*}
\tau_\alpha(u) &= \sum_{j = 0}^{k - 1} \tau_{(\alpha_j, \alpha_{j + 1})}(\sigma^{(\alpha_0, \alpha_1, \dotsc, \alpha_j)}(u)); \\
\mathtt{c}^\alpha(u) &= \prod_{j = 0}^{k - 1} \mathtt{c}^{(\alpha_j, \alpha_{j + 1})}(\sigma^{(\alpha_0, \alpha_1, \dotsc, \alpha_j)}(u)); \\
\vartheta^\alpha(u) &= \prod_{j = 0}^{k - 1} \vartheta^{(\alpha_j, \alpha_{j + 1})}(\sigma^{(\alpha_0, \alpha_1, \dotsc, \alpha_j)}(u)); \\
\Phi^\alpha(u) &= a_{\tau_\alpha(u)}\vartheta^\alpha(u) = \prod_{j = 0}^{k - 1} \Phi^{(\alpha_j, \alpha_{j + 1})}(\sigma^{(\alpha_0, \alpha_1, \dotsc, \alpha_j)}(u))
\end{align*}
for all $u \in \tilde{\mathtt{C}}[\alpha] \subset \tilde{U}$, where the terms of the products are to be in \emph{ascending} order from left to right. For all admissible sequences $\alpha$ with $\len(\alpha) = 0$, we define $\tau_\alpha(u) = 0$ and $\mathtt{c}^\alpha(u) = \vartheta^\alpha(u) = \Phi^\alpha(u) = e \in G$ for all $u \in \tilde{\mathtt{C}}[\alpha]$. For all $u \in U$, there is a unique corresponding admissible sequence in $\Sigma^+$ and hence we can instead use the notations $\tau_k(u)$, $\mathtt{c}^k(u)$, $\vartheta^k(u)$, and $\Phi^k(u)$ for all $k \in \mathbb Z_{\geq 0}$.

\begin{remark}
Since $\mathtt{c}^\alpha$ is constant for all admissible sequences $\alpha$, we will often omit writing the argument. We can also construct the extended congruence cocycles $\mathtt{c}_{\mathfrak{q}}^\alpha: \tilde{\mathtt{C}}[\alpha] \to F_{\mathfrak{q}}$ for all nontrivial ideals $\mathfrak{q} \subset \mathcal{O}_{\mathbb K}$ in a similar way as before.
\end{remark}

\subsection{Transfer operators}
For all nontrivial ideals $\mathfrak{q} \subset \mathcal{O}_{\mathbb K}$, let $\Gamma$ and $\Gamma_{\mathfrak{q}}$ act on $L^2(F_\mathfrak{q})$ from the left by the \emph{right} regular representation which we will simply denote by juxtaposition. Throughout the paper, we will use the notation $\xi = a + ib \in \mathbb C$ for the complex parameter for the transfer operators and use the convention that sums over sequences will be understood to be sums \emph{only} over \emph{admissible} sequences.

\begin{definition}[Congruence transfer operator with holonomy]
For all $\xi \in \mathbb C$, $\rho \in \widehat{M}$, and nontrivial ideals $\mathfrak{q} \subset \mathcal{O}_{\mathbb K}$, the \emph{congruence transfer operator with holonomy} $\tilde{\mathcal{M}}_{\xi\tau, \mathfrak{q}, \rho}: C(\tilde{U}, L^2(F_\mathfrak{q}) \otimes V_\rho^\oplus) \to C(\tilde{U}, L^2(F_\mathfrak{q}) \otimes V_\rho^\oplus)$ is defined by
\begin{align*}
\tilde{\mathcal{M}}_{\xi\tau, \mathfrak{q}, \rho}(H)(u) = \sum_{\substack{(j, k)\\ u' = \sigma^{-(j, k)}(u)}} e^{\xi\tau_{(j, k)}(u')} ((\mathtt{c}_{\mathfrak{q}}^{(j, k)})^{-1} \otimes \rho(\vartheta^{(j, k)}(u'))^{-1}) H(u')
\end{align*}
for all $u \in \tilde{U}$ and $H \in C(\tilde{U}, L^2(F_\mathfrak{q}) \otimes V_\rho^\oplus)$.
\end{definition}

Let $\xi \in \mathbb C$. We denote $\tilde{\mathcal{L}}_{\xi\tau} = \tilde{\mathcal{M}}_{\xi\tau, \mathcal{O}_{\mathbb K}, 1}$ and simply call it the \emph{transfer operator}. Let $\rho \in \widehat{M}$ and $\mathfrak{q} \subset \mathcal{O}_{\mathbb K}$ be a nontrivial ideal, and denote $|_U: C(\tilde{U}, L^2(F_\mathfrak{q}) \otimes V_\rho^\oplus) \to C(U, L^2(F_\mathfrak{q}) \otimes V_\rho^\oplus)$ to be the restriction map. Then we also define the \emph{congruence transfer operator with holonomy} $\mathcal{M}_{\xi\tau, \mathfrak{q}, \rho} = |_U \circ \tilde{\mathcal{M}}_{\xi\tau, \mathfrak{q}, \rho} \circ (|_U)^{-1}$ where $(|_U)^{-1}$ denotes taking any preimage using Tietze extension theorem and denote the \emph{transfer operator} $\mathcal{L}_{\xi\tau} = \mathcal{M}_{\xi\tau, \mathcal{O}_{\mathbb K}, 1}$. We also denote $\mathcal{M}_{\xi\tau, \mathfrak{q}} = \mathcal{M}_{\xi\tau, \mathfrak{q}, 1}$.

\begin{remark}
Let $\xi \in \mathbb C$, $\rho \in \widehat{M}$, and $\mathfrak{q} \subset \mathcal{O}_{\mathbb K}$ be a nontrivial ideal. Then $\tilde{\mathcal{M}}_{\xi\tau, \mathfrak{q}, \rho}$ preserves $C^k(\tilde{U}, L^2(F_\mathfrak{q}) \otimes V_\rho^\oplus)$ for all $k \in \mathbb N$, viewing the target space as a real vector space, and $\mathcal{M}_{\xi\tau, \mathfrak{q}, \rho}$ preserves $C^{\Lip(d)}(U, L^2(F_\mathfrak{q}) \otimes V_\rho^\oplus)$.
\end{remark}

We recall the Ruelle--Perron--Frobenius (RPF) theorem along with the theory of Gibbs measures in this setting \cite{Bow08,PP90}.

\begin{theorem}
\label{thm:RPFonU}
For all $a \in \mathbb R$, the operator $\mathcal{L}_{a\tau}: C(U) \to C(U)$ and its dual $\mathcal{L}_{a\tau}^*: C(U)^* \to C(U)^*$ has eigenvectors with the following properties. There exist a unique positive function $h \in C^{\Lip(d)}(U, \mathbb R)$ and a unique Borel probability measure $\nu$ on $U$ such that
\begin{enumerate}
\item	$\mathcal{L}_{a\tau}(h) = e^{\Pr_\sigma(a\tau_{\Sigma})}h$;
\item	$\mathcal{L}_{a\tau}^*(\nu) = e^{\Pr_\sigma(a\tau_{\Sigma})}\nu$;
\item	the eigenvalue $e^{\Pr_\sigma(a\tau_{\Sigma})}$ is maximal simple and the rest of the spectrum of $\mathcal{L}_f|_{C^{\Lip(d)}(U)}$ is contained in a disk of radius strictly less than $e^{\Pr_\sigma(a\tau_{\Sigma})}$;
\item	$\nu(h) = 1$ and the Borel probability measure $\mu$ defined by $d\mu = h \, d\nu$ is $\sigma$-invariant and is the projection of the $a\tau_{\Sigma}$-equilibrium state to $U$, i.e., $\mu = (\proj_U \circ \zeta)_*(\nu_{a\tau_{\Sigma}})$.
\end{enumerate}
\end{theorem}

In light of \cref{thm:RPFonU}, it is convenient to normalize the transfer operators defined above. Let $a \in \mathbb R$. Define $\lambda_a = e^{\Pr_\sigma(-(\delta_\Gamma + a)\tau_{\Sigma})}$ which is the maximal simple eigenvalue of $\mathcal{L}_{-(\delta_\Gamma + a)\tau}$ by \cref{thm:RPFonU} and recall that $\lambda_0 = 1$. Define the eigenvectors, the unique positive function $h_a \in C^{\Lip(d)}(U, \mathbb R)$ and the unique probability measure $\nu_a$ on $U$ with $\nu_a(h_a) = 1$ such that
\begin{align*}
\mathcal{L}_{-(\delta_\Gamma + a)\tau}(h_a) &= \lambda_a h_a; & \mathcal{L}_{-(\delta_\Gamma + a)\tau}^*(\nu_a) &= \lambda_a \nu_a
\end{align*}
provided by \cref{thm:RPFonU}. Note that $d\nu_U = h_0 \, d\nu_0$. Now by \cite[Theorem A.2]{SW20}, the eigenvector $h_a \in C^{\Lip(d)}(U, \mathbb R)$ extends to an eigenvector $h_a \in C^\infty(\tilde{U}, \mathbb R)$ with bounded derivatives for $\tilde{\mathcal{L}}_{-(\delta_\Gamma + a)\tau}$. For all admissible pairs $(j, k)$, we define the smooth map
\begin{align}
\label{eqn:f^(a)}
f_{(j, k)}^{(a)} = -(a + \delta_\Gamma)\tau_{(j, k)} + \log(h_0) - \log(h_0 \circ \sigma^{(j, k)}) - \log(\lambda_a).
\end{align}
For all $k \in \mathbb N$ and admissible sequences $\alpha = (\alpha_0, \alpha_1, \dotsc, \alpha_k)$, we define the smooth map $f_\alpha^{(a)}: \tilde{\mathtt{C}}[\alpha] \to \mathbb R$ by
\begin{align*}
f_\alpha^{(a)}(u) = \sum_{j = 0}^{k - 1} f_{(\alpha_j, \alpha_{j + 1})}^{(a)}(\sigma^{(\alpha_0, \alpha_1, \dotsc, \alpha_j)}(u)) \qquad \text{for all $u \in \tilde{\mathtt{C}}[\alpha]$}.
\end{align*}
For all admissible sequences $\alpha$ with $\len(\alpha) = 0$, we define $f_\alpha^{(a)}(u) = 0$. As before, for all $u \in U$, we can also use the notation $f_k^{(a)}(u)$ for any $k \in \mathbb Z_{\geq 0}$.

We now normalize the transfer operators. Let $\xi \in \mathbb C$, $\rho \in \widehat{M}$, and $\mathfrak{q} \subset \mathcal{O}_{\mathbb K}$ be a nontrivial ideal. We define $\tilde{\mathcal{M}}_{\xi, \mathfrak{q}, \rho}: C(\tilde{U}, L^2(F_\mathfrak{q}) \otimes V_\rho^\oplus) \to C(\tilde{U}, L^2(F_\mathfrak{q}) \otimes V_\rho^\oplus)$ by
\begin{align*}
\tilde{\mathcal{M}}_{\xi, \mathfrak{q}, \rho}(H)(u) = \sum_{\substack{(j, k)\\ u' = \sigma^{-(j, k)}(u)}} e^{(f_{(j, k)}^{(a)} + ib\tau_{(j, k)})(u')} ((\mathtt{c}_{\mathfrak{q}}^{(j, k)})^{-1} \otimes \rho(\vartheta^{(j, k)}(u'))^{-1}) H(u')
\end{align*}
for all $u \in \tilde{U}$ and $H \in C(\tilde{U}, L^2(F_q) \otimes V_\rho^\oplus)$. For all $k \in \mathbb N$, its $k$\textsuperscript{th} iteration is
\begin{align}
\label{eqn:k^thIterationOfCongruenceTransferOperatorOfType_rho}
\tilde{\mathcal{M}}_{\xi, \mathfrak{q}, \rho}^k(H)(u) = \sum_{\substack{\alpha: \len(\alpha) = k\\ u' = \sigma^{-\alpha}(u)}} e^{f_\alpha^{(a)}(u')} ((\mathtt{c}_{\mathfrak{q}}^\alpha)^{-1} \otimes \rho_b(\Phi^\alpha(u'))^{-1}) H(u')
\end{align}
for all $u \in \tilde{U}$ and $H \in C(\tilde{U}, L^2(F_q) \otimes V_\rho^\oplus)$. Denote $\tilde{\mathcal{L}}_\xi = \tilde{\mathcal{M}}_{\xi, \mathcal{O}_{\mathbb K}, 1}$. Using the restriction map $|_U$, we get the corresponding normalized operator $\mathcal{M}_{\xi, \mathfrak{q}, \rho}: C(U, V_\rho^\oplus) \to C(U, V_\rho^\oplus)$. Denote $\mathcal{M}_{\xi, \mathfrak{q}} = \mathcal{M}_{\xi, \mathfrak{q}, 1}$ and $\mathcal{L}_\xi = \mathcal{M}_{\xi, \mathcal{O}_{\mathbb K}, 1}$. With this normalization, for all $a \in \mathbb R$, the maximal simple eigenvalue of $\mathcal{L}_a$ is $1$ with normalized eigenvector $\frac{h_a}{h_0}$. Moreover, we have $\mathcal{L}_0^*(\nu_U) = \nu_U$.

\begin{remark}
For all nontrivial ideals $\mathfrak{q} \subset \mathcal{O}_{\mathbb K}$, since the left actions of $\Gamma$ and $\Gamma_{\mathfrak{q}}$ on $L^2(F_\mathfrak{q})$ coincide, we can drop the subscript of the cocycle in the definition of the congruence transfer operators with holonomy whenever required. We prefer to keep the subscript in \cref{sec:ReductionToNewInvariantFunctionsAtLevel_q,sec:ApproximatingTheTransferOperator,sec:ZariskiDensityAndTraceFieldOfTheReturnTrajectorySubgroups,sec:L2FlatteningLemma,sec:SupremumAndLipschitzBounds} where we use the expander machinery because it is important that the cocycle takes values in a finite group and prefer to drop the subscript in \cref{sec:Dolgopyat'sMethod} where we use Dolgopyat's method because it is unnecessary.
\end{remark}

We fix some related constants. By perturbation theory for operators as in \cite[Chapter 8]{Kat95} and \cite[Proposition 4.6]{PP90}, we can fix $a_0' > 0$ such that the map $[-a_0', a_0'] \to \mathbb R$ defined by $a \mapsto \lambda_a$ and the map $[-a_0', a_0'] \to C(\tilde{U}, \mathbb R)$ defined by $a \mapsto h_a$ are Lipschitz. We then fix $A_f > 0$ such that $|f^{(a)}(u) - f^{(0)}(u)| \leq A_f|a|$ for all $u \in U$ and $|a| \leq a_0'$. Fix $\overline{\tau} = \max_{(j, k)} \sup_{u \in \tilde{\mathtt{C}}[j, k]} \tau_{(j, k)}(u)$ and $\underline{\tau} = \min_{(j, k)} \inf_{u \in \tilde{\mathtt{C}}[j, k]} \tau_{(j, k)}(u)$. Fix
\begin{align*}
T_0 > \max(C_\theta, 1) \cdot \max\bigg(\max_{(j, k)} \|\tau_{(j, k)}\|_{C^1}, \max_{(j, k)} \sup_{|a| \leq a_0'} \left\|f_{(j, k)}^{(a)}\right\|_{C^1}, \max_{(j, k)} \|\vartheta^{(j, k)}\|_{C^1}\bigg)
\end{align*}
which is possible by \cite[Lemma 4.1]{PS16}.

\subsection{Uniform spectral bounds with holonomy}
\label{sec:UniformSpectralBoundsWithHolonomy}
We first introduce some inner products, norms, and seminorms. Let $\rho \in \widehat{M}$, $\mathfrak{q} \subset \mathcal{O}_{\mathbb K}$ be a nontrivial ideal, and $H \in C(U, L^2(F_\mathfrak{q}) \otimes V_\rho^\oplus)$. We will denote $\|H\| \in C(U, \mathbb R)$ to be the function defined by $\|H\|(u) = \|H(u)\|_2$ for all $u \in U$, and if $\rho = 1$, we will denote $|H| \in C(U, L^2(F_\mathfrak{q}, \mathbb R))$ to be the function defined by $|H|(u) = |H(u)| \in L^2(F_\mathfrak{q}, \mathbb R)$ for all $u \in U$. We define $\|H\|_\infty = \sup \|H\|$. We use similar notations if the domain is $\tilde{U}$. We define the Lipschitz seminorm and the Lipschitz norm by
\begin{align*}
\Lip_d(H) &= \sup_{\substack{u, u' \in U\\ \text{such that } u \neq u'}} \frac{\|H(u) - H(u')\|_2}{d(u, u')}; & \|H\|_{\Lip(d)} &= \|H\|_\infty + \Lip_d(H)
\end{align*}
respectively. We also denote $\|H\|_{\Lip(d_\theta)} = \|H \circ \zeta^+\|_{\Lip(d_\theta)}$.

Let $Y$ be a Riemannian manifold and $H \in C^1(\tilde{U}, Y)$. We define the $C^1$ seminorm and the $C^1$ norm by
\begin{align*}
|H|_{C^1} &= \sup_{u \in \tilde{U}}\|(dH)_u\|_{\mathrm{op}}; & \|H\|_{C^1} &= \|H\|_\infty + |H|_{C^1}
\end{align*}
respectively. It is also useful to define the norm
\begin{align*}
\|H\|_{1, b} &= \|H\|_\infty + \frac{1}{\max(1, |b|)} |H|_{C^1}.
\end{align*}
Henceforth, by differentiable function spaces on $\tilde{U}$ or its derived suspension spaces, such as $C^1(\tilde{U}, Y)$, we will always mean the space of $C^1$ functions whose $C^1$ norm is \emph{bounded}.

Let $\rho \in \widehat{M}$ and $\mathfrak{q} \subset \mathcal{O}_{\mathbb K}$ be a nontrivial ideal. We define the Banach spaces
\begin{align*}
\mathcal{V}_{\mathfrak{q}, \rho}(U) &= C^{\Lip(d)}(U, L^2(F_\mathfrak{q}) \otimes V_\rho^\oplus);
&
\mathcal{V}_{\mathfrak{q}, \rho}(\tilde{U}) &= C^1(\tilde{U}, L^2(F_\mathfrak{q}) \otimes V_\rho^\oplus).
\end{align*}
We define $\mathcal{W}_{\mathfrak{q}, \rho}(U) \subset \mathcal{V}_{\mathfrak{q}, \rho}(U)$ and $\mathcal{W}_{\mathfrak{q}, \rho}(\tilde{U}) \subset \mathcal{V}_{\mathfrak{q}, \rho}(\tilde{U})$ in a similar fashion with $L^2(F_\mathfrak{q})$ above replaced by $L_0^2(F_\mathfrak{q}) = \big\{\phi \in L^2(F_\mathfrak{q}): \sum_{g \in F_\mathfrak{q}} \phi(g) = 0\big\}$. We denote $\mathcal{V}_\mathfrak{q}(U) = \mathcal{V}_{\mathfrak{q}, 1}(U)$ and $\mathcal{W}_\mathfrak{q}(U) = \mathcal{W}_{\mathfrak{q}, 1}(U)$.

For all ideals $\mathfrak{q} \subset \mathcal{O}_{\mathbb K}$, we denote the norm of the ideal by $N_{\mathbb K}(\mathfrak{q}) = \#(\mathcal{O}_{\mathbb K}/\mathfrak{q})$ and we say $\mathfrak{q}$ is \emph{square-free} if it is a nontrivial proper ideal without any square prime ideal factors.

Now we can state \cref{thm:TheoremFrameFlow} which is the main technical theorem regarding uniform spectral bounds of congruence transfer operators with holonomy.

\begin{theorem}
\label{thm:TheoremFrameFlow}
There exist $\eta > 0$, $C \geq 1$, $a_0 > 0$, and a nontrivial proper ideal $\mathfrak{q}_0' \subset \mathcal{O}_{\mathbb K}$ such that for all $\xi \in \mathbb C$ with $|a| < a_0$, $\rho \in \widehat{M}$, square-free ideals $\mathfrak{q} \subset \mathcal{O}_{\mathbb K}$ coprime to $\mathfrak{q}_0\mathfrak{q}_0'$, $k \in \mathbb N$, and $H \in \mathcal{W}_{\mathfrak{q}, \rho}(U)$ with an extension $\tilde{H} \in \mathcal{W}_{\mathfrak{q}, \rho}(\tilde{U})$, we have
\begin{align*}
\left\|\mathcal{M}_{\xi, \mathfrak{q}, \rho}^k(H)\right\|_2 \leq CN_{\mathbb K}(\mathfrak{q})^C e^{-\eta k} \|\tilde{H}\|_{1, \|\rho_b\|}.
\end{align*}
\end{theorem}

\cref{thm:TheoremFrameFlow} follows from the following \cref{thm:TheoremSmall|b|,thm:TheoremLarge|b|OrNontrivial_rho}.

\begin{theorem}
\label{thm:TheoremSmall|b|}
There exist $\eta > 0$, $C \geq 1$, $a_0 > 0$, $b_0 > 0$, and a nontrivial proper ideal $\mathfrak{q}_0' \subset \mathcal{O}_{\mathbb K}$ such that for all $\xi \in \mathbb C$ with $|a| < a_0$ and $|b| \leq b_0$, square-free ideals $\mathfrak{q} \subset \mathcal{O}_{\mathbb K}$ coprime to $\mathfrak{q}_0\mathfrak{q}_0'$, $k \in \mathbb N$, and $H \in \mathcal{W}_\mathfrak{q}(U)$, we have
\begin{align*}
\left\|\mathcal{M}_{\xi, \mathfrak{q}}^k(H)\right\|_2 \leq CN_{\mathbb K}(\mathfrak{q})^C e^{-\eta k} \|H\|_{\Lip(d)}.
\end{align*}
\end{theorem}

\begin{theorem}
\label{thm:TheoremLarge|b|OrNontrivial_rho}
There exist $\eta > 0$, $C > 0$, $a_0 > 0$, and $b_0 > 0$ such that for all $\xi \in \mathbb C$ with $|a| < a_0$, if $(b, \rho) \in \widehat{M}_0(b_0)$, then for all nontrivial ideals $\mathfrak{q} \subset \mathcal{O}_{\mathbb K}$, $k \in \mathbb N$, and $H \in \mathcal{V}_{\mathfrak{q}, \rho}(\tilde{U})$, we have
\begin{align*}
\big\|\tilde{\mathcal{M}}_{\xi, \mathfrak{q}, \rho}^k(H)\big\|_2 \leq Ce^{-\eta k} \|H\|_{1, \|\rho_b\|}.
\end{align*}
\end{theorem}

We will first prove \cref{thm:TheoremSmall|b|}. We fix $b_0 > 0$ to be the one from \cref{thm:TheoremLarge|b|OrNontrivial_rho} where it is clear from \cref{sec:Dolgopyat'sMethod} that we can assume $b_0 = 1$ from \cite[Eq. (11)]{SW20}. To prove \cref{thm:TheoremSmall|b|}, we first make some reductions as in \cite{OW16}.

\section{Reduction to new invariant functions at level $\mathfrak{q}$}
\label{sec:ReductionToNewInvariantFunctionsAtLevel_q}
In this section, we reduce \cref{thm:TheoremSmall|b|} to \cref{thm:ReducedTheoremSmall|b|}. This is done by introducing the concept of new invariant functions.

Let $\mathfrak{q} \subset \mathfrak{q}' \subset \mathcal{O}_{\mathbb K}$ be ideals. The canonical quotient map $\pi_{\mathfrak{q}, \mathfrak{q}'}: \tilde{\mathbf{G}}(\mathcal{O}_{\mathbb K}/\mathfrak{q}) \to \tilde{\mathbf{G}}(\mathcal{O}_{\mathbb K}/\mathfrak{q}')$ induces the pull back $\pi_{\mathfrak{q}, \mathfrak{q}'}^*: L^2(\tilde{\mathbf{G}}(\mathcal{O}_{\mathbb K}/\mathfrak{q}')) \to L^2(\tilde{\mathbf{G}}(\mathcal{O}_{\mathbb K}/\mathfrak{q}))$. Define $\hat{E}_{\mathfrak{q}'}^\mathfrak{q} = \pi_{\mathfrak{q}, \mathfrak{q}'}^*(L^2(\tilde{\mathbf{G}}(\mathcal{O}_{\mathbb K}/\mathfrak{q}')))$ and $\dot{E}_{\mathfrak{q}'}^\mathfrak{q} = \hat{E}_{\mathfrak{q}'}^\mathfrak{q} \cap \big(\bigoplus_{\mathfrak{q}' \subsetneq \mathfrak{q}''} \hat{E}_{\mathfrak{q}''}^\mathfrak{q}\big)^\perp$. Then, we have the orthogonal decomposition
\begin{align*}
L_0^2(\tilde{\mathbf{G}}(\mathcal{O}_{\mathbb K}/\mathfrak{q})) = \bigoplus_{\mathfrak{q} \subset \mathfrak{q}' \subsetneq \mathcal{O}_{\mathbb K}} \dot{E}_{\mathfrak{q}'}^\mathfrak{q} \qquad \text{for all ideals $\mathfrak{q} \subset \mathcal{O}_{\mathbb K}$}.
\end{align*}
Similarly, using the induced quotient map $\overline{\pi_{\mathfrak{q}, \mathfrak{q}'}}: \tilde{G}_\mathfrak{q} \to \tilde{G}_{\mathfrak{q}'}$, we have the orthogonal decomposition
\begin{align*}
L_0^2(\tilde{G}_\mathfrak{q}) = \bigoplus_{\mathfrak{q} \subset \mathfrak{q}' \subsetneq \mathcal{O}_{\mathbb K}} E_{\mathfrak{q}'}^\mathfrak{q} \qquad \text{for all ideals $\mathfrak{q} \subset \mathcal{O}_{\mathbb K}$}.
\end{align*}

\begin{remark}
We exclude $\mathfrak{q}' = \mathcal{O}_{\mathbb K}$ above because the subspaces $\dot{E}_{\mathcal{O}_{\mathbb K}}^\mathfrak{q} \subset L^2(\tilde{\mathbf{G}}(\mathcal{O}_{\mathbb K}/\mathfrak{q}))$ and $E_{\mathcal{O}_{\mathbb K}}^\mathfrak{q} \subset L^2(\tilde{G}_{\mathfrak{q}})$ consists of constant functions.
\end{remark}

Let $\mathfrak{q} \subset \mathfrak{q}' \subset \mathcal{O}_{\mathbb K}$ be ideals. For all $\phi \in L^2(\tilde{G}_\mathfrak{q})$, define $\dot{\phi} \in L^2(\tilde{\mathbf{G}}(\mathcal{O}_{\mathbb K}/\mathfrak{q}))$ by $\dot{\phi}(g) = \phi(\pi_{\mathfrak{q}}(\ker(\tilde{\pi}))g)$ for all $g \in \tilde{\mathbf{G}}(\mathcal{O}_{\mathbb K}/\mathfrak{q})$. Then $E_{\mathfrak{q}'}^\mathfrak{q} = \{\phi \in L^2(\tilde{G}_\mathfrak{q}): \dot{\phi} \in \dot{E}_{\mathfrak{q}'}^\mathfrak{q}\}$. In the relevant case when $\mathfrak{q}$ is coprime to $\mathfrak{q}_0$, the subspace $E_{\mathfrak{q}'}^\mathfrak{q} \subset L^2(\tilde{G}_\mathfrak{q})$ can be thought of as consisting of new functions which are invariant under $\Gamma_{\mathfrak{q}'}$ but not invariant under $\Gamma_{\mathfrak{q}''}$ for any $\mathfrak{q}' \subsetneq \mathfrak{q}''$, using the isomorphism $\overline{\pi_\mathfrak{q}|_{\tilde{\Gamma}}}: F_\mathfrak{q} \to \tilde{G}_\mathfrak{q}$. We define $\mathcal{W}_{\mathfrak{q}'}^\mathfrak{q}(U) = \{H \in \mathcal{W}_\mathfrak{q}(U): H(u) \in E_{\mathfrak{q}'}^\mathfrak{q} \text{ for all } u \in U\}$, so that we have the orthogonal decomposition
\begin{align*}
\mathcal{W}_\mathfrak{q}(U) = \bigoplus_{\mathfrak{q} \subset \mathfrak{q}' \subsetneq \mathcal{O}_{\mathbb K}} \mathcal{W}_{\mathfrak{q}'}^\mathfrak{q}(U) \qquad \text{for all ideals $\mathfrak{q} \subset \mathcal{O}_{\mathbb K}$}.
\end{align*}
We have the canonical projection operator $e_{\mathfrak{q}, \mathfrak{q}'}: \mathcal{W}_\mathfrak{q}(U) \to \mathcal{W}_{\mathfrak{q}'}^\mathfrak{q}(U)$ and we can define the canonical projection map $\proj_{\mathfrak{q}, \mathfrak{q}'} = (\overline{\pi_{\mathfrak{q}, \mathfrak{q}'}}^*)^{-1}\big|_{E_{\mathfrak{q}'}^\mathfrak{q}}: E_{\mathfrak{q}'}^\mathfrak{q} \to E_{\mathfrak{q}'}^{\mathfrak{q}'}$ since $\overline{\pi_{\mathfrak{q}, \mathfrak{q}'}}^*$ is injective and $E_{\mathfrak{q}'}^\mathfrak{q} \subset \overline{\pi_{\mathfrak{q}, \mathfrak{q}'}}^*(L^2(\tilde{G}_{\mathfrak{q}'}))$. This is equivalent to defining $\proj_{\mathfrak{q}, \mathfrak{q}'}(\phi)(g) = \phi(\tilde{g})$ where $\tilde{g}$ is any lift of $g$ with respect to $\overline{\pi_{\mathfrak{q}, \mathfrak{q}'}}: \tilde{G}_\mathfrak{q} \to \tilde{G}_{\mathfrak{q}'}$, for all $g \in \tilde{G}_{\mathfrak{q}'}$ and $\phi \in E_{\mathfrak{q}'}^\mathfrak{q}$. We use the same notation $\proj_{\mathfrak{q}, \mathfrak{q}'}: \mathcal{W}_{\mathfrak{q}'}^\mathfrak{q}(U) \to \mathcal{W}_{\mathfrak{q}'}^{\mathfrak{q}'}(U)$ for the induced projection map defined pointwise. The congruence transfer operator $\mathcal{M}_{\xi, \mathfrak{q}}$ preserves $\mathcal{W}_{\mathfrak{q}'}^\mathfrak{q}(U)$ for all $\xi \in \mathbb C$. The projection operator commutes with the congruence transfer operator, i.e., $e_{\mathfrak{q}, \mathfrak{q}'} \circ \mathcal{M}_{\xi, \mathfrak{q}} = \mathcal{M}_{\xi, \mathfrak{q}} \circ e_{\mathfrak{q}, \mathfrak{q}'}$ for all $\xi \in \mathbb C$, and the projection map is equivariant with respect to the congruence transfer operator, i.e., $\proj_{\mathfrak{q}, \mathfrak{q}'} \circ \mathcal{M}_{\xi, \mathfrak{q}} = \mathcal{M}_{\xi, \mathfrak{q}'} \circ \proj_{\mathfrak{q}, \mathfrak{q}'}$ for all $\xi \in \mathbb C$. By surjectivity of $\overline{\pi_{\mathfrak{q}, \mathfrak{q}'}}$, we can denote $\spadesuit_{\mathfrak{q}, \mathfrak{q}'} = \#\ker(\overline{\pi_{\mathfrak{q}, \mathfrak{q}'}}) = \frac{\#\tilde{G}_\mathfrak{q}}{\#\tilde{G}_{\mathfrak{q}'}} = [\Gamma_{\mathfrak{q}'}: \Gamma_\mathfrak{q}]$ and by direct calculation it can be checked that $\|H\|_2 = \sqrt{\spadesuit_{\mathfrak{q}, \mathfrak{q}'}} \|\proj_{\mathfrak{q}, \mathfrak{q}'}(H)\|_2$ and $\|H\|_{\Lip(d)} = \sqrt{\spadesuit_{\mathfrak{q}, \mathfrak{q}'}} \|\proj_{\mathfrak{q}, \mathfrak{q}'}(H)\|_{\Lip(d)}$.

\begin{theorem}
\label{thm:Reduced'TheoremSmall|b|}
There exist $\eta > 0$, $C \geq 1$, $a_0 > 0$, and $q_1 \in \mathbb N$ such that for all $\xi \in \mathbb C$ with $|a| < a_0$ and $|b| \leq b_0$, square-free ideals $\mathfrak{q} \subset \mathcal{O}_{\mathbb K}$ coprime to $\mathfrak{q}_0$ with $N_{\mathbb K}(\mathfrak{q}) > q_1$, $k \in \mathbb N$, and $H \in \mathcal{W}_\mathfrak{q}^\mathfrak{q}(U)$, we have
\begin{align*}
\left\|\mathcal{M}_{\xi, \mathfrak{q}}^k(H)\right\|_2 \leq CN_{\mathbb K}(\mathfrak{q})^C e^{-\eta k} \|H\|_{\Lip(d)}.
\end{align*}
\end{theorem}

\begin{proof}[Proof that \cref{thm:Reduced'TheoremSmall|b|} implies \cref{thm:TheoremSmall|b|}]
Fix $\eta > 0$, $C_1 \geq 1$, $a_0 > 0$, and $q_1 \in \mathbb N$ to be the $\eta$, $C$, $a_0$, and $q_1$ from \cref{thm:Reduced'TheoremSmall|b|}. Fix $C = C_1 + \frac{1}{2}$. Set $\mathfrak{q}_0' \subset \mathcal{O}_{\mathbb K}$ to be the product of all nontrivial prime ideals $\mathfrak{p} \subset \mathcal{O}_{\mathbb K}$ with $N_{\mathbb K}(\mathfrak{p}) \leq q_1$ so that if $\mathfrak{q} \subset \mathcal{O}_{\mathbb K}$ is an ideal coprime to $\mathfrak{q}_0\mathfrak{q}_0'$ and $\mathfrak{q} \subset \mathfrak{q}' \subsetneq \mathcal{O}_{\mathbb K}$ is a proper ideal, then $N_{\mathbb K}(\mathfrak{q}') > q_1$. Let $\xi \in \mathbb C$ with $|a| < a_0$ and $|b| \leq b_0$. Let $\mathfrak{q} \subset \mathcal{O}_{\mathbb K}$ be a square-free ideal coprime to $\mathfrak{q}_0\mathfrak{q}_0'$, $k \in \mathbb N$, and $H \in \mathcal{W}_\mathfrak{q}(U)$. We have
\begin{align*}
\left\|\mathcal{M}_{\xi, \mathfrak{q}}^k(H)\right\|_2^2 &= \sum_{\mathfrak{q} \subset \mathfrak{q}' \subsetneq \mathcal{O}_{\mathbb K}} \left\|e_{\mathfrak{q}, \mathfrak{q}'}(\mathcal{M}_{\xi, \mathfrak{q}}^k (H))\right\|_2^2 \\
&= \sum_{\mathfrak{q} \subset \mathfrak{q}' \subsetneq \mathcal{O}_{\mathbb K}} \spadesuit_{\mathfrak{q}, \mathfrak{q}'} \left\|\proj_{\mathfrak{q}, \mathfrak{q}'}(e_{\mathfrak{q}, \mathfrak{q}'}(\mathcal{M}_{\xi, \mathfrak{q}}^k(H)))\right\|_2^2 \\
&= \sum_{\mathfrak{q} \subset \mathfrak{q}' \subsetneq \mathcal{O}_{\mathbb K}} \spadesuit_{\mathfrak{q}, \mathfrak{q}'} \left\|\mathcal{M}_{\xi, \mathfrak{q}'}^k(\proj_{\mathfrak{q}, \mathfrak{q}'}(e_{\mathfrak{q}, \mathfrak{q}'}(H)))\right\|_2^2.
\end{align*}
Since $\proj_{\mathfrak{q}, \mathfrak{q}'}(e_{\mathfrak{q}, \mathfrak{q}'}(H)) \in \mathcal{W}_{\mathfrak{q}'}^{\mathfrak{q}'}(U)$, we can now apply \cref{thm:Reduced'TheoremSmall|b|} to get
\begin{align*}
\left\|\mathcal{M}_{\xi, \mathfrak{q}}^k(H)\right\|_2^2 &\leq \sum_{\mathfrak{q} \subset \mathfrak{q}' \subsetneq \mathcal{O}_{\mathbb K}} C_1^2N_{\mathbb K}(\mathfrak{q}')^{2C_1} e^{-2\eta k} \spadesuit_{\mathfrak{q}, \mathfrak{q}'} \left\|\proj_{\mathfrak{q}, \mathfrak{q}'}(e_{\mathfrak{q}, \mathfrak{q}'}(H))\right\|_{\Lip(d)}^2 \\
&= \sum_{\mathfrak{q} \subset \mathfrak{q}' \subsetneq \mathcal{O}_{\mathbb K}} C_1^2N_{\mathbb K}(\mathfrak{q}')^{2C_1} e^{-2\eta k} \left\|e_{\mathfrak{q}, \mathfrak{q}'}(H)\right\|_{\Lip(d)}^2 \\
&\leq C_1^2N_{\mathbb K}(\mathfrak{q})^{2C_1} e^{-2\eta k} \left\|H\right\|_{\Lip(d)}^2 \sum_{\mathfrak{q} \subset \mathfrak{q}' \subsetneq \mathcal{O}_{\mathbb K}} 1 \\
&\leq C^2N_{\mathbb K}(\mathfrak{q})^{2C}e^{-2\eta k} \left\|H\right\|_{\Lip(d)}^2.
\end{align*}
\end{proof}

\begin{theorem}
\label{thm:ReducedTheoremSmall|b|}
There exist $C_s > 0$, $a_0 > 0$, $\kappa \in (0, 1)$, and $q_1 \in \mathbb N$ such that for all $\xi \in \mathbb C$ with $|a| < a_0$ and $|b| \leq b_0$, square-free ideals $\mathfrak{q} \subset \mathcal{O}_{\mathbb K}$ coprime to $\mathfrak{q}_0$ with $N_{\mathbb K}(\mathfrak{q}) > q_1$, there exists an integer $s \in (0, C_s\log(N_{\mathbb K}(\mathfrak{q})))$ such that for all $j \in \mathbb Z_{\geq 0}$, and $H \in \mathcal{W}_\mathfrak{q}^\mathfrak{q}(U)$, we have
\begin{align*}
\big\|\mathcal{M}_{\xi, \mathfrak{q}}^{js}(H)\big\|_2 \leq N_{\mathbb K}(\mathfrak{q})^{-j\kappa} \|H\|_{\Lip(d_\theta)}.
\end{align*}
\end{theorem}

\begin{proof}[Proof that \cref{thm:ReducedTheoremSmall|b|} implies \cref{thm:Reduced'TheoremSmall|b|}]
Fix $C_s$, $a_0 > 0$, $\kappa \in (0, 1)$, and $q_1 \in \mathbb N$ to be the ones from \cref{thm:ReducedTheoremSmall|b|}. Fix
\begin{align*}
B' =  \sup_{|a| \leq a_0, \{0\} \subsetneq \mathfrak{q} \subset \mathcal{O}_{\mathbb K}} \log\left\|\mathcal{M}_{\xi, \mathfrak{q}}\right\|_{\mathrm{op}} \leq \sup_{|a| \leq a_0} \log\left\|\mathcal{L}_\xi\right\|_{\mathrm{op}} \leq \log(Ne^{T_0})
\end{align*}
viewing the transfer operators as operators on $L^2(U, L^2(F_\mathfrak{q}))$ and $L^2(U, \mathbb R)$ respectively. Fix $B = \max(0, B')$, $C = \max(BC_s + 1, C_\theta) \geq 1$, and $\eta = \frac{\kappa}{C_s} > 0$, recalling $C_\theta$ from \cref{subsec:SymbolicDynamics}. Let $\xi \in \mathbb C$ with $|a| < a_0$ and $|b| \leq b_0$. Let $\mathfrak{q} \subset \mathcal{O}_{\mathbb K}$ be a square-free ideal coprime to $\mathfrak{q}_0$ with $N_{\mathbb K}(\mathfrak{q}) > q_1$ and $s \in (0, C_s\log(N_{\mathbb K}(\mathfrak{q})))$ be the corresponding integer provided by \cref{thm:ReducedTheoremSmall|b|}. Then $\left\|\mathcal{M}_{\xi, \mathfrak{q}}\right\|_{\mathrm{op}}^m \leq N_{\mathbb K}(\mathfrak{q})^{BC_s}$ for all $0 \leq m < s$. Let $k \in \mathbb N$ and $H \in \mathcal{W}_\mathfrak{q}^\mathfrak{q}(U)$. Writing $k = js + m$ for some $j \in \mathbb Z_{\geq 0}$ and $0 \leq m < s$, and using \cref{thm:ReducedTheoremSmall|b|}, we have
\begin{align*}
\left\|\mathcal{M}_{\xi, \mathfrak{q}}^k(H)\right\|_2 &\leq \left\|\mathcal{M}_{\xi, \mathfrak{q}}\right\|^m \cdot \big\|\mathcal{M}_{\xi, \mathfrak{q}}^{js}(H)\big\|_2 \\
&\leq N_{\mathbb K}(\mathfrak{q})^{BC_s} N_{\mathbb K}(\mathfrak{q})^{-j\kappa} \|H\|_{\Lip(d_\theta)} \\
&\leq N_{\mathbb K}(\mathfrak{q})^{BC_s + \frac{m\kappa}{s}} e^{-(js + m) \frac{\kappa}{s}\log(N_{\mathbb K}(\mathfrak{q}))} \|H\|_{\Lip(d_\theta)} \\
&\leq CN_{\mathbb K}(\mathfrak{q})^C e^{-\eta k} \|H\|_{\Lip(d)}.
\end{align*}
\end{proof}

The rest of the section is devoted to obtaining strong bounds which are crucial for the proof of \cref{thm:ReducedTheoremSmall|b|} in \cref{sec:SupremumAndLipschitzBounds}.

\section{Approximating the congruence transfer operators}
\label{sec:ApproximatingTheTransferOperator}
The aim of this section is to approximate the congruence transfer operators by convolutions with measures to mimic a random walk.

Let $\mathfrak{q} \subset \mathcal{O}_{\mathbb K}$ be an ideal coprime to $\mathfrak{q}_0$. For all $H \in C(U, L^2(\tilde{G}_\mathfrak{q}))$, we define $H_{\Sigma^+} \in C(\Sigma^+, L^2(\tilde{G}_\mathfrak{q}))$ by $H_{\Sigma^+} = H \circ \zeta^+$. Define $f_{\Sigma^+}^{(a)}$ as in \cref{eqn:f^(a)} by replacing $\sigma|_U$ and $\tau$ with $\sigma|_{\Sigma^+}$ and $\tau_{\Sigma^+}$. In light of \cref{cor:CocyclesLocallyConstantCorollary}, we define $\mathtt{c}_{\mathfrak{q}, \Sigma^+}: \Sigma^+ \to \tilde{G}_\mathfrak{q}$ simply by $\mathtt{c}_{\mathfrak{q}, \Sigma^+}(x) = \mathtt{c}_\mathfrak{q}(u)$ for all $x \in \Sigma^+$, for any choice of $u \in \mathtt{C}[x_0, x_1]$. As there is dependence only on the first two entries, we use the notation $\mathtt{c}_{\mathfrak{q}, \Sigma^+}(j, k)$ for all admissible pairs $(j, k)$ in the natural way. For all $k \in \mathbb N$, we then have
\begin{align*}
\mathtt{c}_{\mathfrak{q}, \Sigma^+}^k(x) = \mathtt{c}_{\mathfrak{q}, \Sigma^+}(x_0, x_1) \mathtt{c}_{\mathfrak{q}, \Sigma^+}(x_1, x_2) \dotsb \mathtt{c}_{\mathfrak{q}, \Sigma^+}(x_{k - 1}, x_k)
\end{align*}
and define $\mathtt{c}_{\mathfrak{q}, \Sigma^+}^0(x) = e \in \Gamma$, for all $x \in \Sigma^+$. We now define another \emph{congruence transfer operator} $\mathcal{M}_{\xi, \mathfrak{q}, \Sigma^+}: C(\Sigma^+, L^2(\tilde{G}_\mathfrak{q})) \to C(\Sigma^+, L^2(\tilde{G}_\mathfrak{q}))$ as in \cref{eqn:k^thIterationOfCongruenceTransferOperatorOfType_rho} by replacing $\sigma|_U$, $\tau$, $f^{(a)}$, and $\mathtt{c}_{\mathfrak{q}}$ with $\sigma|_{\Sigma^+}$, $\tau_{\Sigma^+}$, $f_{\Sigma^+}^{(a)}$, and $\mathtt{c}_{\mathfrak{q}, \Sigma^+}$, and taking $k = 1$ and $\rho = 1$. By a slight abuse of notation, we will drop the subscripts $\Sigma^+$ henceforth.

We introduce some useful notations. Let $j \in \mathbb N$ and $(\alpha_j, \alpha_{j - 1}, \dotsc, \alpha_1)$ be an admissible sequence. We denote $\alpha^j = (\alpha_j, \alpha_{j - 1}, \dotsc, \alpha_1)$. Also, when sequences are themselves written in a sequence, they are to be concatenated. For all $y \in \mathcal A$, denote $\omega(y) \in \Sigma^+$ to be any sequence such that $(y, \omega(y))$ is admissible. We naturally extend the notation for admissible sequences as well so that $\omega(\alpha^j) = \omega(\alpha_1)$.

Let $\mathfrak{q} \subset \mathcal{O}_{\mathbb K}$ be an ideal coprime to $\mathfrak{q}_0$. For any measure $\mu$ on $\tilde{G}_\mathfrak{q}$ and $\phi \in L^2(\tilde{G}_\mathfrak{q})$ the convolution $\mu * \phi \in L^2(\tilde{G}_\mathfrak{q})$ is defined by
\begin{align*}
(\mu * \phi)(g) = \sum_{h \in \tilde{G}_\mathfrak{q}} \mu(h) \phi(gh^{-1}) \qquad \text{for all $g \in \tilde{G}_\mathfrak{q}$}.
\end{align*}

For all $\xi \in \mathbb C$, for all ideals $\mathfrak{q} \subset \mathcal{O}_{\mathbb K}$ coprime to $\mathfrak{q}_0$, $x \in \Sigma^+$, integers $0 < r < s$, and admissible sequences $(\alpha_s, \alpha_{s - 1}, \dotsc, \alpha_{r + 1})$, we define the measures
\begin{align*}
\mu_{(\alpha_s, \alpha_{s - 1}, \dotsc, \alpha_{r + 1})}^{\xi, \mathfrak{q}, x} &= \sum_{\alpha^r} e^{(f_s^{(a)} + ib\tau_s)(\alpha^s, x)} \delta_{\mathtt{c}_\mathfrak{q}^{r + 1}(\alpha_{r + 1}, \alpha^r, x)}; \\
\nu_0^{a, \mathfrak{q}, x, r} &= \sum_{\alpha^r} e^{f_r^{(a)}(\alpha^r, x)} \delta_{\mathtt{c}_\mathfrak{q}^{r + 1}(\alpha_{r + 1}, \alpha^r, x)}; \\
\hat{\mu}_{(\alpha_s, \alpha_{s - 1}, \dotsc, \alpha_{r + 1})}^{a, \mathfrak{q}, x} &= \sum_{\alpha^r} e^{f_s^{(a)}(\alpha^s, x)} \delta_{\mathtt{c}_\mathfrak{q}^{r + 1}(\alpha_{r + 1}, \alpha^r, x)} = e^{f_{s - r}^{(a)}(\alpha^s, x)} \nu_0^{a, \mathfrak{q}, x, r}; \\
\nu_{(\alpha_s, \alpha_{s - 1}, \dotsc, \alpha_{r + 1})}^{a, \mathfrak{q}, x} &= e^{f_{s - r}^{(a)}(\alpha_s, \alpha_{s - 1}, \dotsc, \alpha_{r + 1}, \omega(\alpha_{r + 1}))} \nu_0^{a, \mathfrak{q}, x, r}
\end{align*}
on $\tilde{G}_\mathfrak{q}$ and also for all $H \in C(U, L^2(\tilde{G}_\mathfrak{q}))$, define the function
\begin{align*}
\phi_{(\alpha_s, \alpha_{s - 1}, \dotsc, \alpha_{r + 1})}^{\mathfrak{q}, H} = \delta_{\mathtt{c}_\mathfrak{q}^{s - r - 1}(\alpha_s, \alpha_{s - 1}, \dotsc, \alpha_{r + 1}, \omega(\alpha_{r + 1}))} * H(\alpha_s, \alpha_{s - 1}, \dotsc, \alpha_{r + 1}, \omega(\alpha_{r + 1}))
\end{align*}
in $L^2(\tilde{G}_\mathfrak{q})$ where we note that $\big\|\phi_{(\alpha_s, \alpha_{s - 1}, \dotsc, \alpha_{r + 1})}^{\mathfrak{q}, H}\big\|_2 \leq \|H\|_\infty$.

Now we present a bound which will be used often.

\begin{lemma}
\label{lem:SumExpf^aBound}
There exists $C > 1$ such that for all $|a| < a_0'$, $x \in \Sigma^+$, and $k \in \mathbb N$, we have $\sum_{\alpha^k} e^{f_k^{(a)}(\alpha^k, x)} \leq C$.
\end{lemma}

\begin{proof}
Fix $C = e^{A_fa_0'} > 1$. Let $|a| < a_0'$, $x \in \Sigma^+$, and $k \in \mathbb N$. Then $|f^{(a)} - f^{(0)}| \leq A_f|a| < A_fa_0'$ and so we have
\begin{align*}
\sum_{\alpha^k} e^{f_k^{(a)}(\alpha^k, x)} \leq e^{A_fa_0'}\sum_{\alpha^k} e^{f_k^{(0)}(\alpha^k, x)} = e^{A_fa_0'}\mathcal{L}_0^k(\chi_{\Sigma^+})(x) = e^{A_fa_0'} = C.
\end{align*}
\end{proof}

Fix $C_f$ to be the $C$ provided by \cref{lem:SumExpf^aBound}.

We record an easy lemma here which relate the measures defined above.

\begin{lemma}
\label{lem:muHatLessThanCnu}
There exists $C > 0$ such that for all $\xi \in \mathbb C$ with $|a| < a_0'$, ideals $\mathfrak{q} \subset \mathcal{O}_{\mathbb K}$ coprime to $\mathfrak{q}_0$, $x \in \Sigma^+$, integers $0 < r < s$, and admissible sequences $(\alpha_s, \alpha_{s - 1}, \dotsc, \alpha_{r + 1})$, we have $\left|\mu_{(\alpha_s, \alpha_{s - 1}, \dotsc, \alpha_{r + 1})}^{\xi, \mathfrak{q}, x}\right| \leq \hat{\mu}_{(\alpha_s, \alpha_{s - 1}, \dotsc, \alpha_{r + 1})}^{a, \mathfrak{q}, x}$ and
\begin{align*}
C^{-1}\nu_{(\alpha_s, \alpha_{s - 1}, \dotsc, \alpha_{r + 1})}^{a, \mathfrak{q}, x} \leq \hat{\mu}_{(\alpha_s, \alpha_{s - 1}, \dotsc, \alpha_{r + 1})}^{a, \mathfrak{q}, x} \leq C\nu_{(\alpha_s, \alpha_{s - 1}, \dotsc, \alpha_{r + 1})}^{a, \mathfrak{q}, x}.
\end{align*}
\end{lemma}

\begin{proof}
Fix $C = e^{\frac{T_0 \theta}{1 - \theta}}$. Let $\xi \in \mathbb C$ with $|a| < a_0'$, $\mathfrak{q} \subset \mathcal{O}_{\mathbb K}$ be an ideal coprime to $\mathfrak{q}_0$, $x \in \Sigma^+$, $0 < r < s$ be integers, and $(\alpha_s, \alpha_{s - 1}, \dotsc, \alpha_{r + 1})$ be an admissible sequence. Denote $\mu_{(\alpha_s, \alpha_{s - 1}, \dotsc, \alpha_{r + 1})}^{\xi, \mathfrak{q}, x}$ by $\mu$, $\hat{\mu}_{(\alpha_s, \alpha_{s - 1}, \dotsc, \alpha_{r + 1})}^{a, \mathfrak{q}, x}$ by $\hat{\mu}$, and $\nu_{(\alpha_s, \alpha_{s - 1}, \dotsc, \alpha_{r + 1})}^{a, \mathfrak{q}, x}$ by $\nu$. It is easy to check $f_s^{(a)}(\alpha^s, x) = f_{s - r}^{(a)}(\alpha^s, x) + f_r^{(a)}(\alpha^r, x)$ from which $|\mu| \leq \hat{\mu}$ follows. We also have
\begin{align*}
&\big|f_s^{(a)}(\alpha^s, x) - \big(f_{s - r}^{(a)}(\alpha_s, \alpha_{s - 1}, \dotsc, \alpha_{r + 1}, \omega(\alpha_{r + 1})) + f_r^{(a)}(\alpha^r, x)\big)\big| \\
={}&\big|f_{s - r}^{(a)}(\alpha^s, x) - f_{s - r}^{(a)}(\alpha_s, \alpha_{s - 1}, \dotsc, \alpha_{r + 1}, \omega(\alpha_{r + 1}))\big| \\
\leq{}&\sum_{k = 0}^{s - r - 1} \big|f^{(a)}(\sigma^k(\alpha^s, x)) - f^{(a)}(\sigma^k(\alpha_s, \alpha_{s - 1}, \dotsc, \alpha_{r + 1}, \omega(\alpha_{r + 1})))\big| \\
\leq{}&\sum_{k = 0}^{s - r - 1}\Lip_{d_\theta}(f^{(a)}) \cdot d_\theta(\sigma^k(\alpha^s, x), \sigma^k(\alpha_s, \alpha_{s - 1}, \dotsc, \alpha_{r + 1}, \omega(\alpha_{r + 1}))) \\
\leq{}&\Lip_{d_\theta}(f^{(a)}) \sum_{k = 0}^{s - r - 1} \theta^{s - r - k} \leq \frac{T_0 \theta}{1 - \theta}.
\end{align*}
Hence, the lemma follows by comparing $\hat{\mu}$ and $\nu$.
\end{proof}

Now, returning to our goal of approximating the transfer operator, the following lemma is proved as in \cite[Lemma 4.21]{OW16}.

\begin{lemma}
\label{lem:TransferOperatorConvolutionApproximation}
For all $\xi \in \mathbb C$ with $|a| < a_0'$, for all ideals $\mathfrak{q} \subset \mathcal{O}_{\mathbb K}$ coprime to $\mathfrak{q}_0$, $x \in \Sigma^+$, integers $0 < r < s$, and $H \in \mathcal{V}_\mathfrak{q}(U)$, we have
\begin{multline*}
\left\|\mathcal{M}_{\xi, \mathfrak{q}}^s(H)(x) - \sum_{\alpha_{r + 1}, \alpha_{r + 2}, \dotsc, \alpha_s} \mu_{(\alpha_s, \alpha_{s - 1}, \dotsc, \alpha_{r + 1})}^{\xi, \mathfrak{q}, x} * \phi_{(\alpha_s, \alpha_{s - 1}, \dotsc, \alpha_{r + 1})}^{\mathfrak{q}, H}\right\|_2 \\
\leq C_f \Lip_{d_\theta}(H)\theta^{s - r}.
\end{multline*}
\end{lemma}

We will use this approximation to study the convolution, rather than dealing with the transfer operator directly, and obtain strong bounds. This is the objective of \cref{sec:L2FlatteningLemma} but first we need to establish an important fact in \cref{sec:ZariskiDensityAndTraceFieldOfTheReturnTrajectorySubgroups}.

\section{Zariski density and trace field of the return trajectory subgroups}
\label{sec:ZariskiDensityAndTraceFieldOfTheReturnTrajectorySubgroups}
The aim of this section is to prove that the return trajectory subgroups have finite index in $\Gamma$ as stated in \cref{thm:H(yz)FiniteIndexInGamma}. The important consequence is that the return trajectory subgroups are then Zariski dense and have full trace field $\mathbb K$. This will be required in \cref{sec:L2FlatteningLemma} in order to use the strong approximation theorem together with the expander machinery of Golsefidy--Varj\'{u} \cite{GV12}.

\begin{definition}[Return trajectory subgroup]
Let $(y, z) \in \mathcal{A}^2$. For all $p \in \mathbb N$, we define the \emph{return trajectory subgroup of level $p$} to be the subgroup $H^p(y, z) < \Gamma$ generated by the subset $S^p(y, z)$ which consists of the elements
\begin{align*}
\mathtt{c}^{p + 1}(\alpha)\mathtt{c}^{p + 1}(\tilde{\alpha})^{-1} = \prod_{j = 0}^p \mathtt{c}(\alpha_j, \alpha_{j + 1}) \prod_{j = 0}^p \mathtt{c}(\tilde{\alpha}_{p - j}, \tilde{\alpha}_{p + 1 - j})^{-1} \in \Gamma
\end{align*}
for all admissible sequences $\alpha = (\alpha_0, \alpha_1, \dotsc, \alpha_{p + 1})$ and $\tilde{\alpha} = (\tilde{\alpha}_0, \tilde{\alpha}_1, \dotsc, \tilde{\alpha}_{p + 1})$ such that $\alpha_0 = \tilde{\alpha}_0 = y$ and $\alpha_{p + 1} = \tilde{\alpha}_{p + 1} = z$. We simply call the subgroup $H(y, z) < \Gamma$ generated by the subset $S(y, z) = \bigcup_{p = 1}^\infty S^p(y, z)$ the \emph{return trajectory subgroup}.
\end{definition}

For all $p \in \mathbb N$ and $(y, z) \in \mathcal{A}^2$, we denote $\tilde{S}^p(y, z) = \tilde{\pi}^{-1}(S^p(y, z)) \subset \tilde{\Gamma}$ and $\tilde{H}^p(y, z) = \langle \tilde{S}^p(y, z) \rangle = \tilde{\pi}^{-1}(H^p(y, z)) < \tilde{\Gamma}$.

The following lemma gives a useful property of the return trajectory subgroups. Recall the constant $N_T \in \mathbb N$ associated with the topologically mixing property of $T$ from \cref{subsec:SymbolicDynamics}.

\begin{lemma}
\label{lem:ShortWordGeneratorAsLongWordGeneratorTrick}
For all $(y, z) \in \mathcal{A}^2$, $p_0 \in \mathbb N$, and $p \geq p_0 + N_T$, we have $S^{p_0}(y, z) \subset S^p(y, z)$.
\end{lemma}

\begin{proof}
Let $(y, z) \in \mathcal A^2$. Let $s = \mathtt{c}^{p_0 + 1}(\alpha)\mathtt{c}^{p_0 + 1}(\tilde{\alpha})^{-1} \in S^{p_0}(y, z)$ for some $p_0 \in \mathbb N$ and admissible sequences $\alpha = (\alpha_0, \alpha_1, \dotsc, \alpha_{p_0 + 1})$ and $\tilde{\alpha} = (\tilde{\alpha}_0, \tilde{\alpha}_1, \dotsc, \tilde{\alpha}_{p_0 + 1})$ such that $\alpha_0 = \tilde{\alpha}_0 = y$ and $\alpha_{p_0 + 1} = \tilde{\alpha}_{p_0 + 1} = z$. Let $p \geq p_0 + N_T$ be an integer. Then $p - p_0 \geq N_T$, and since $T$ is topologically mixing, we can fix an admissible sequence $\beta = (\beta_1, \beta_2, \dotsc, \beta_{p - p_0})$ such that $(z, \beta)$ is admissible and $\beta_{p - p_0} = z$. We can then extend the admissible sequences to $\alpha = (\alpha_0, \alpha_1, \dotsc, \alpha_{p_0 + 1}, \beta)$ and $\tilde{\alpha} = (\tilde{\alpha}_0, \tilde{\alpha}_1, \dotsc, \tilde{\alpha}_{p_0 + 1}, \beta)$. Due to cancellations of the middle factors, we have
\begin{align*}
s = \mathtt{c}^{p_0 + 1}(\alpha)\mathtt{c}^{p_0 + 1}(\tilde{\alpha})^{-1} &= \prod_{j = 0}^{p_0} \mathtt{c}(\alpha_j, \alpha_{j + 1}) \prod_{j = 0}^{p_0} \mathtt{c}(\tilde{\alpha}_{p_0 - j}, \tilde{\alpha}_{p_0 + 1 - j})^{-1} \\
&= \prod_{j = 0}^p \mathtt{c}(\alpha_j, \alpha_{j + 1}) \prod_{j = 0}^p \mathtt{c}(\tilde{\alpha}_{p - j}, \tilde{\alpha}_{p + 1 - j})^{-1} \\
&= \mathtt{c}^{p + 1}(\alpha)\mathtt{c}^{p + 1}(\tilde{\alpha})^{-1} \in S^p(y, z).
\end{align*}
\end{proof}

\begin{remark}
From \cref{lem:ShortWordGeneratorAsLongWordGeneratorTrick} we can derive that $H(y, z) = H(y, z')$ for all $y, z, z' \in \mathcal{A}$. Further using similar tricks, we can in fact derive that $H(y, z) = \gamma H(y', z')\gamma^{-1}$ for some $\gamma \in \Gamma$ for all $(y, z), (y', z') \in \mathcal{A}^2$.
\end{remark}

We now present the main theorem of this section and its corollaries.

\begin{theorem}
\label{thm:H(yz)FiniteIndexInGamma}
For all $(y, z) \in \mathcal A^2$, the subgroup $H(y, z) < \Gamma$ is of finite index.
\end{theorem}

\begin{corollary}
\label{cor:H^p(yz)FiniteIndexInGamma}
For all $(y, z) \in \mathcal A^2$, there exists $p_0 \in \mathbb N$ such that for all $p > p_0$, we have $H^p(y, z) = H(y, z)$; in particular, $H^p(y, z) < \Gamma$ is of finite index.
\end{corollary}

\begin{proof}
Let $(y, z) \in \mathcal A^2$. Since $H(y, z)$ is convex cocompact by \cref{thm:H(yz)FiniteIndexInGamma}, it is finitely generated. Thus, we can choose $h_1, h_2, \dotsc, h_{k_0} \in H(y, z)$ for some $k_0 \in \mathbb N$ such that $H(y, z) = \langle h_1, h_2, \dotsc, h_{k_0} \rangle$. For all $1 \leq k \leq k_0$, we have $h_k = s_{k, 1}s_{k, 2} \dotsb s_{k, j_k}$ as a product of generators in $S(y, z)$. But then $s_{k, j} \in S^{p_{k, j}}(y, z)$ for some $p_{k, j} \in \mathbb N$, for all $1 \leq k \leq k_0$ and $1 \leq j \leq j_k$. Fix $p_0 = \max_{1 \leq k \leq k_0, 1 \leq j \leq j_k} p_{k, j} + N_T$ and let $p > p_0$ be an integer. By \cref{lem:ShortWordGeneratorAsLongWordGeneratorTrick}, we can conclude that $s_{k, j} \in S^{p}(y, z)$ for all $1 \leq k \leq k_0$ and $1 \leq j \leq j_k$. Thus, $H(y, z) = \langle h_1, h_2, \dotsc, h_{k_0} \rangle \subset H^p(y, z)$ and so in fact $H^p(y, z) = H(y, z)$. The corollary now follows from \cref{thm:H(yz)FiniteIndexInGamma}.
\end{proof}

\begin{corollary}
\label{cor:Z-DenseInSimplyConnectedCoverGAndTraceFieldK}
There exists $p_0 \in \mathbb N$ such that for all $p > p_0$ and $(y, z) \in \mathcal A^2$, we have
\begin{enumerate}
\item $\tilde{H}^p(y, z)$ is Zariski dense in $\tilde{\mathbf{G}}$;
\item $\mathbb Q(\tr(\Ad(\tilde{H}^p(y, z)))) = \mathbb K$.
\end{enumerate}
\end{corollary}

\begin{proof}
For all $(y, z) \in \mathcal A^2$, let $p_{(y, z)} \in \mathbb N$ be the $p_0$ provided by \cref{cor:H^p(yz)FiniteIndexInGamma}. Fix $p_0 = \max_{(y, z) \in \mathcal A^2} p_{(y, z)}$ and let $p > p_0$ be an integer. Suppose $\tilde{H}^p(y, z)$ is not Zariski dense in $\tilde{\mathbf{G}}$. The image of the Zariski closure of $\tilde{H}^p(y, z)$ under $\tilde{\pi}$, which is a Zariski constructible set, cannot contain any Zariski open subset of $\mathbf{G}$ for dimensional reasons. Thus, it is contained in a finite union of proper subvarieties of $\mathbf{G}$ which contradicts the Zariski density of $H^p(y, z)$ in $\mathbf{G}$ guaranteed by \cref{cor:H^p(yz)FiniteIndexInGamma}. From \cref{cor:H^p(yz)FiniteIndexInGamma}, we derive that $\tilde{H}^p(y, z) < \tilde{\Gamma}$ is also a finite index subgroup and hence the trace field property follows from \cite[Theorem 3]{Vin71}.
\end{proof}

Our goal is now to prove \cref{thm:H(yz)FiniteIndexInGamma}. We prepare by first fixing some notations for the rest of the section. Through an isometry, we will view the hyperbolic space in the upper half space model $\mathbb H^n \cong \{(x_1, x_2, \dotsc, x_n) \in \mathbb R^n: x_n > 0\}$ with boundary at infinity $\partial_\infty\mathbb H^n \cong \{(x_1, x_2, \dotsc, x_n) \in \mathbb R^n: x_n = 0\} \cup \{\infty\} \cong \mathbb R^{n - 1} \cup \{\infty\}$. We also use the isometry $\T(\mathbb H^n) \cong \mathbb H^n \times \mathbb R^n$ and denote by $\pi_{\mathbb H^n}: \mathbb H^n \times \mathbb R^n \to \mathbb H^n$ the tangent bundle projection map. Let $(e_1, e_2, \dotsc, e_n)$ be the standard basis on $\mathbb R^n$ and $\pi_{\mathbb R^{n - 1}}: \mathbb R^n \to \mathbb R^{n - 1}$ be the orthogonal projection onto $\mathbb R^{n - 1} \cong \langle e_1, e_2, \dotsc, e_{n - 1}\rangle$. Let $B^{\mathrm{E}}_\epsilon(u) \subset \mathbb R^{n - 1} \subset \partial_\infty\mathbb H^n$ denote the open Euclidean ball of radius $\epsilon > 0$ centered at $u \in \mathbb R^{n - 1}$. We reserve the notation $B_\epsilon(V) \subset \mathbb H^n$ for the $\epsilon$-neighborhood of $V \subset \mathbb H^n$ for $\epsilon > 0$ with respect to the hyperbolic metric. For all $j \in \mathcal{A}$, denote by $\hat{\mathsf{R}}_j$, $\hat{\mathsf{U}}_j$, and $\hat{\mathsf{S}}_j$ the cores of $\mathsf{R}_j$, $\mathsf{U}_j$, and $\mathsf{S}_j$ respectively. Denote by $\mathsf{\Omega}$ the preimage of $\Omega$ under the covering map $\T^1(\mathbb H^n) \to \T^1(X)$. Recall the trajectory isomorphism $\psi$ from \cite[Definition 1.1]{Rat73} defined such that for all $y \in \mathcal{A}$ and $u, u' \in [U_y, u]$, we have $u' = \psi_u^{-1}(u')a_t$ for some unique $\psi_u^{-1}(u') \in W_{\epsilon_0}^{\mathrm{su}}(u)$ and $t \in (\underline{\tau}, \underline{\tau})$. Define the map $\Phi: \T^1(\mathbb H^n) \to \partial_\infty \mathbb H^n$ by $\Phi(u) = u^+$ for all $u \in \T^1(\mathbb H^n)$. Consider the set of unit tangent vectors $V = \{(u, -e_n) \in \T^1(\mathbb H^n): \langle u, e_n \rangle = 1\}$ perpendicular to a horosphere. Setting $C_{\mathrm{E}} > 0$ to be the image of the constant map $V \to \mathbb R$ defined by $u \mapsto \|(d\Phi)_u\|_{\mathrm{op}}$, we have $\frac{1}{C_{\mathrm{E}}}d_{\mathrm{su}}(u, v)\leq \|u^+ - v^+\| \leq C_{\mathrm{E}}d_{\mathrm{su}}(u, v)$ for all $u, v \in V$.

\begin{figure}
\begin{tikzpicture}[>=stealth]
\draw[thick] (-3, 0) to (3, 0);
\draw[very thin] (-3, 5) to (3, 5);
\draw[very thick] (-1.5, 5) to (2.5, 5);
\draw[thick, dotted] (0,5) to (0, 0);
\draw[thick, dotted] (1,5) to (1, 0);
\draw[thick, dotted] (2,5) to (2, 0);

\draw[very thin] (2, 0.5) circle  [radius = 0.5, fill = white];
\draw[very thick] ([shift=(20:0.5)]2,0.5) arc (20:160:0.5);
\draw[thick, dotted] (2, 0) arc(0:180:1);

\draw[very thin] (1, 0.25) circle  [radius = 0.25, fill = white];
\draw[very thick] ([shift=(20:0.25)]1,0.25) arc (20:160:0.25);
\draw[thick, dotted] (1, 0) arc(0:180:0.5);


\draw[->, thick] (0,5) -- (0,4);
\draw[->, thick] (1,5) -- (1,4);
\draw[->, thick] (2,5) -- (2,4);
\draw[->, thick] (-1,5) -- (-1,4);

\draw[->, thick] (2,1) -- (2,0.7);
\draw[->, thick] (2 - 0.4, 0.8) -- (2 - 0.16,0.62);
\draw[->, thick] (0.4, 0.8) -- (0.64,0.98);
\draw[->, thick] (2 + 0.235294, 0.941176) -- (2 + 0.094118, 0.67647);

\draw[->, thick] (1,0.5) -- (1,0.3);
\draw[->, thick] (0.8, 0.4) -- (0.8 + 0.16,0.4 - 0.12);
\draw[->, thick] (0.5, 0.5) -- (0.7, 0.5);

\fill[gray, fill opacity = 0.2] (0, 0) to (-3, 0.75) to (-3, 5) to (3, 5) to (3, 0.75);

\node[font=\tiny] at (-1.25, 5.12) {1};
\node[font=\tiny] at (-0.12, 4.5) {2};
\node[font=\tiny] at (2-0.12, 4.5) {3};
\node[font=\tiny] at ([shift=(35:0.6)]2, 0.5) {4};
\node[font=\tiny] at (2 - 0.14, 0.86) {5};
\node[font=\tiny] at (2- 0.38, 0.65) {6};
\node[font=\tiny] at (0.42, 0.94) {7};
\node[font=\tiny] at (0, -0.12) {8};
\node[font=\tiny] at (2, -0.12) {9};


\coordinate (Legend) at (current bounding box.west);

\fontsize{5}{2}\selectfont
\matrix[right] at ([shift={(0,-0.6)}]Legend) {
\node [label=right:$\mathsf{U}_y \subset W^{\mathrm{su}}(u)$] {1.}; \\
\node [label=right:$u$] {2.}; \\
\node [label=right:$u_q$] {3.}; \\
\node [label=right:{$[v_q, g_z\mathsf{S}_z] \subset W^{\mathrm{ss}}(v_q)$}] {4.}; \\
\node [label=right:$\mathcal{P}^{q + r}(u_q)$] {5.}; \\
\node [label=right:$v_q$] {6.}; \\
\node [label=right:$\mathcal{P}^{-(q + r)}(v_q)$] {7.}; \\
\node [label=right:$u^+$] {8.}; \\
\node [label=right:$u_q^+$] {9.}; \\
};
\end{tikzpicture}
\caption{This illustrates the idea of the proof of \cref{pro:LimitSetH(yy)RadialEqualsLimitSetH(yy)EqualsLimitSetGamma}. Note that the actual positions of $v_q$ and $\mathcal{P}^{-(q + r)}(v_q)$ are perturbations of what is shown.}
\label{fig:OrbitVectorCloseToOrigin}
\end{figure}
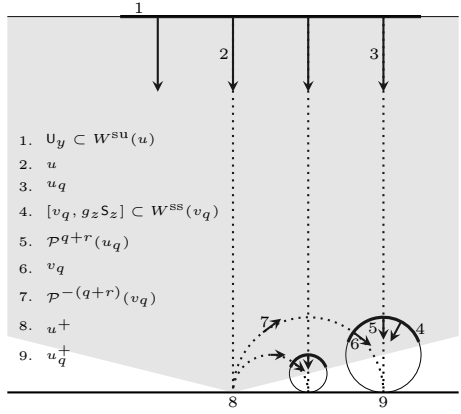

The following \cref{lem:ProducingElementsofH} gives a description of the elements in the generating set of the return trajectory subgroup using the Poincar\'{e} map.

\begin{lemma}
\label{lem:ProducingElementsofH}
Let $p \in \mathbb N$, and $h \in \Gamma$. If there exist $g_z \in \Gamma$, $u_0 \in \hat{\mathsf{R}}_y$, and $v_0 \in g_z\hat{\mathsf{R}}_z$ such that $\mathcal P^{p + 1}(u_0) \in g_z\hat{\mathsf{R}}_z$ and $\mathcal P^{-(p + 1)}(v_0) \in h\hat{\mathsf{R}}_y$, then $h \in S^p(y, z)$.
\end{lemma}

\begin{proof}
Let $g_z$, $u_0$, and $v_0$ be as in the lemma. Let $\tilde{u}_0 = h^{-1} \mathcal P^{-(p + 1)}(v_0) \in \hat{\mathsf{R}}_y$. Then $\mathcal P^{p + 1}(\tilde{u}_0) \in \tilde{g}_z\hat{\mathsf{R}}_z$ where $\tilde{g}_z = h^{-1}g_z$. Now from definitions, there exist admissible sequences $\alpha = (\alpha_0, \alpha_1, \dotsc, \alpha_{p + 1})$ and $\tilde{\alpha} = (\tilde{\alpha}_0, \tilde{\alpha}_1, \dotsc, \tilde{\alpha}_{p + 1})$ with $\alpha_0 = \tilde{\alpha}_0 = y$ and $\alpha_{p + 1} = \tilde{\alpha}_{p + 1} = z$ such that $g_z = \mathtt{c}^{p + 1}(\alpha)$ and $\tilde{g}_z = \mathtt{c}^{p + 1}(\tilde{\alpha})$. Thus, $h = g_z\tilde{g}_z^{-1} = \mathtt{c}^{p + 1}(\alpha) \mathtt{c}^{p + 1}(\tilde{\alpha})^{-1} \in S^p(y, z)$ by definition.
\end{proof}

\cref{lem:ProducingElementsofH} provides a procedure to produce elements in the return trajectory subgroup. Our strategy is to use this procedure to construct a sequence of orbit points which has limit $0 \in \partial_\infty \mathbb H^n$. Recalling \cref{def:RadialLimitSet}, we also ensure the crucial property that such a sequence can be constructed inside a cone, as depicted in \cref{fig:OrbitVectorCloseToOrigin}, to show that $0 \in \partial_\infty \mathbb H^n$ is a radial limit point.

\begin{lemma}
\label{lem:ChooseAFarAway_v_SuchThat0IsInView}
Let $z \in \mathcal{A}$. There exist $\delta_1 > 0$, $\delta_2 > 0$, and $r \in \mathbb N$ such that for all $y \in \mathcal{A}$ and $u \in \mathsf{R}_y$, there exists $v \in [\hat{\mathsf{U}}_y, u]$ such that
\begin{enumerate}
\item\label{itm:ChooseAFarAway_v} $d_{\mathrm{su}}(u, \psi_u^{-1}(v)) > \delta_1$;
\item\label{itm:0IsInView} $\mathcal{P}^r(v) \in g_z[\hat{\mathsf{U}}_z, \mathsf{S}_z]$ such that $W_{\delta_2}^{\mathrm{ss}}(\mathcal{P}^r(v)) \cap \mathsf{\Omega} \subset [\mathcal{P}^r(v), g_z\mathsf{S}_z]$ for some $g_z \in \Gamma$ and if $g \in G$ such that $gv = (e_n, -e_n)$, then we have $\Phi(W_{\delta_2}^{\mathrm{ss}}(g\mathcal{P}^r(v))) \supset \partial_\infty \mathbb H^n \setminus B_{C_{\mathrm{E}}^{-1}\delta_1e^{-\underline{\tau}}}^{\mathrm{E}}(0)$.
\end{enumerate}
\end{lemma}

\begin{proof}
Let $z \in \mathcal{A}$. Let $w_y \in \mathsf{R}_y$ be the center for all $y \in \mathcal{A}$. Using the fact that $\Gamma$ is nonelementary and $T$ is topologically mixing, we can fix two $\mathcal{P}$-periodic points $\omega_1, \omega_2 \in \hat{\mathsf{R}}_z$ with common period $r_0 \in \mathbb N$, by which we mean $\mathcal{P}^{r_0}(\omega_1) = \gamma_1\omega_1$ and $\mathcal{P}^{r_0}(\omega_2) = \gamma_2\omega_2$ for some $\gamma_1, \gamma_2 \in \Gamma$ (they correspond to closed geodesics on $X$), such that $[\omega_1, w_z] \neq [\omega_2, w_z]$. For all $y \in \mathcal{A}$, fix any admissible sequence $\alpha_y = (\alpha_{y, 0}, \alpha_{y, 1}, \dotsc, \alpha_{y, N_T})$ with $\len(\alpha_y) = N_T$ such that $\alpha_{y, 0} = y$ and $\alpha_{y, N_T} = z$. Define $\omega_{y, 1} = \sigma^{-\alpha_y}([\omega_1, w_z]) \in g_y\hat{\mathsf{U}}_y$ and $\omega_{y, 2} = \sigma^{-\alpha_y}([\omega_2, w_z]) \in g_y\hat{\mathsf{U}}_y$ for some $g_y \in \Gamma$ and $\delta_y' = \frac{1}{2} d_{\mathrm{su}}(\omega_{y, 1}, \omega_{y, 2})$, for all $y \in \mathcal{A}$. Since $[\cdot, \cdot]: W_{\epsilon_0}^{\mathrm{su}}(w_y) \times W_{\epsilon_0}^{\mathrm{ss}}(w_y) \to [W_{\epsilon_0}^{\mathrm{su}}(w_y), W_{\epsilon_0}^{\mathrm{ss}}(w_y)]$ is a homeomorphism and $\mathsf{U}_y$ and $\mathsf{S}_y$ are compact for all $y \in \mathcal{A}$, there exists $\delta_1 \in (0, \min_{y \in \mathcal{A}} \delta_y'')$ where
\begin{align*}
\delta_y'' = \inf\bigl\{d_{\mathrm{su}}\bigl([u', v'], \psi_{[u', v']}^{-1}([u'', v'])\bigr): u', u'' \in \mathsf{U}_y \text{ with } d_{\mathrm{su}}(u', u'') \geq \delta_y', v' \in \mathsf{S}_y\bigr\}
\end{align*}
for all $y \in \mathcal{A}$.

By properness of $\mathsf{S}_z$, there exists $\delta_2 > 0$ such that $W_{2\delta_2}^{\mathrm{ss}}(\omega_j) \cap \mathsf{\Omega} \subset [\omega_j, \mathsf{S}_z]$ for all $j \in \{1, 2\}$. Thus, for all $w \in W_{\delta_2}^{\mathrm{ss}}(\omega_j)$, we have $W_{\delta_2}^{\mathrm{ss}}(w) \cap \mathsf{\Omega} \subset [\omega_j, \mathsf{S}_z]$.

Let $u \in \mathsf{R}_y$ for some $y \in \mathcal{A}$. We now fix $v = [g_y^{-1}\omega_{y, j}, u]$ for some $j \in \{1, 2\}$ as appropriate since one of them satisfies \cref{itm:ChooseAFarAway_v} by construction. Since $\omega_1$ and $\omega_2$ are $\mathcal{P}$-periodic points, by the Markov property, there exists $k_1 \in \mathbb N$ independent of the choice of $u$ or $v$ such that for all $k > k_1$, we have $\mathcal{P}^{N_T + kr_0}(v) \in W_{\delta_2}^{\mathrm{ss}}(\omega_j)$ and hence $W_{\delta_2}^{\mathrm{ss}}(\mathcal{P}^{N_T + kr_0}(v)) \cap \mathsf{\Omega} \subset [\mathcal{P}^{N_T + kr_0}(v), g_z\mathsf{S}_z]$ for some $g_z \in \Gamma$. Now, suppose $g \in G$ such that $gv = (e_n, -e_n)$. Then, for all $k > k_1$, we also have $\iota_k g \mathcal{P}^{N_T + kr_0}(v) = (e_n, -e_n)$ where $\iota_k \in G$ is a dilation by a factor of $e^{\tau_{N_T + kr_0}(v)}$. In fact, $\iota_k g W_{\delta_2}^{\mathrm{ss}}(\mathcal{P}^{N_T + kr_0}(v)) = W_{\delta_2}^{\mathrm{ss}}(\iota_k g \mathcal{P}^{N_T + kr_0}(v)) = W_{\delta_2}^{\mathrm{ss}}((e_n, -e_n))$ are independent of $k > k_1$. Thus, $\Phi(\iota_k g W_{\delta_2}^{\mathrm{ss}}(\mathcal{P}^{N_T + kr_0}(v))) = \partial_\infty \mathbb H^n \setminus B_B^{\mathrm{E}}(0)$ for some constant $B > 0$. Since $\Phi$ is equivariant with respect to isometries, we can apply $\iota_k^{-1}$ to get $\Phi(gW_{\delta_2}^{\mathrm{ss}}(\mathcal{P}^{N_T + kr_0}(v))) = \partial_\infty \mathbb H^n \setminus B_{B_k}^{\mathrm{E}}(0)$ where $B_k = Be^{-\tau_{N_T + kr_0}(v)} \leq Be^{-(N_T + kr_0)\underline{\tau}}$, for all $k > k_1$. Thus, there exists an integer $k_2 > k_1$ such that we have the additional property $\Phi(gW_{\delta_2}^{\mathrm{ss}}(\mathcal{P}^{N_T + kr_0}(v))) \supset \partial_\infty \mathbb H^n \setminus B_{C_{\mathrm{E}}^{-1}\delta_1e^{-\underline{\tau}}}^{\mathrm{E}}(0)$ for all $k > k_2$. Fixing any such $k > k_2$ and $r = N_T + kr_0$, we can also guarantee \cref{itm:0IsInView}.
\end{proof}

\begin{proposition}
\label{pro:LimitSetH(yy)RadialEqualsLimitSetH(yy)EqualsLimitSetGamma}
For all $(y, z) \in \mathcal{A}^2$, we have $\Lambda_{\mathrm{r}}(H(y, z)) = \Lambda(H(y, z)) = \Lambda(\Gamma)$.
\end{proposition}

\begin{proof}
Let $(y, z) \in \mathcal{A}^2$. Clearly, $\Lambda_{\mathrm{r}}(H(y, z)) \subset \Lambda(H(y, z)) \subset \Lambda(\Gamma)$. Thus, it suffices to show that $\Lambda(\Gamma) \subset \Lambda_{\mathrm{r}}(H(y, z))$. Let $u \in g_y\mathsf{R}_y$ for some $g_y \in \Gamma$ be a tangent vector with an arbitrary forward limit point $u^+ \in \Lambda(\Gamma)$. Without loss of generality, we can assume $g_y = e \in \Gamma$ and also that $u \in \mathsf{U}_y$. There exists $g \in G$ such that $gu = (e_n, -e_n)$ and $gu^+ = 0 \in \partial_\infty \mathbb H^n$. Without loss of generality, we can assume $g = e \in G$. We will now construct a sequence of $H(y, z)$-orbit points of $\pi_{\mathbb H^n}(u) = e_n \in \mathbb H^n$ inside some cone whose limit is $u^+ = 0 \in \partial_\infty \mathbb H^n$.

Fix $\delta_1$, $\delta_2$, and $r \in \mathbb N$ from \cref{lem:ChooseAFarAway_v_SuchThat0IsInView}. Let $q \in \mathbb N$. Let $\mathtt{C}_q \subset \mathsf{U}_y$ be a cylinder and $\alpha$ be a corresponding admissible sequence with $\len(\mathtt{C}_q) = \len(\alpha) = q$ such that $u \in \overline{\mathtt{C}_q}$. Then, $\overline{\mathcal{P}^q(\mathtt{C}_q)} \subset g_{y'}\mathsf{R}_{y'}$ for some $y' \in \mathcal{A}$ and $g_{y'} \in \Gamma$. Note that $ua_{\tau_\alpha(u)} \in \overline{\mathcal{P}^q(\mathtt{C}_q)}$, corresponding to which we can fix $u_q' \in \mathcal{P}^q(\mathtt{C}_q) \cap g_{y'}[\hat{\mathsf{U}}_{y'}, \mathsf{S}_{y'}]$ provided by \cref{lem:ChooseAFarAway_v_SuchThat0IsInView}. Let $u_q = \mathcal{P}^{-q}(u_q') \in \hat{\mathsf{U}}_y$. By \cref{lem:ChooseAFarAway_v_SuchThat0IsInView}, we have $\mathcal{P}^{q + r}(u_q) \in g_z[\hat{\mathsf{U}}_z, \mathsf{S}_z]$ for some $g_z \in \Gamma$. Since $\|\pi_{\mathbb R^{n - 1}}(\pi_{\mathbb H^n}(\mathcal{P}^{q + r}(u_q)))\| = \|\pi_{\mathbb R^{n - 1}}(\pi_{\mathbb H^n}(u_q'))\|$, so $\pi_{\mathbb R^{n - 1}}(\pi_{\mathbb H^n}(\mathcal{P}^{q + r}(u_q)))$ satisfies the bound
\begin{align}
\label{eqn:u_q^+_HorizontalCoordinateBounds}
\frac{1}{C_{\mathrm{E}}} \delta_1 e^{-\tau_\alpha(u)} \leq \|\pi_{\mathbb R^{n - 1}}(\pi_{\mathbb H^n}(\mathcal{P}^{q + r}(u_q)))\| \leq C_{\mathrm{E}}\hat{\delta} e^{-\tau_\alpha(u)}
\end{align}
using $d_{\mathrm{su}}\bigl(ua_{\tau_\alpha(u)}, \psi_{ua_{\tau_\alpha(u)}}^{-1}(u_q')\bigr) > \delta_1$ from \cref{itm:ChooseAFarAway_v} of \cref{lem:ChooseAFarAway_v_SuchThat0IsInView}. Moreover, it lies at height
\begin{align*}
\langle \pi_{\mathbb H^n}(\mathcal{P}^{q + r}(u_q)), e_n \rangle = e^{-\tau_{q + r}(u_q)} = e^{-(\tau_\alpha(u_q) + \tau_r(u_q'))} \geq e^{-(\tau_\alpha(u) + \underline{\tau} + r\overline{\tau})}.
\end{align*}
We then have the calculation
\begin{align*}
\frac{\langle \pi_{\mathbb H^n}(\mathcal{P}^{q + r}(u_q)), e_n \rangle}{\|\pi_{\mathbb R^{n - 1}}(\pi_{\mathbb H^n}(\mathcal{P}^{q + r}(u_q)))\|} \geq \frac{e^{-(\tau_\alpha(u) + \underline{\tau} + r\overline{\tau})}}{C_{\mathrm{E}}\hat{\delta} e^{-\tau_\alpha(u)}} = \frac{e^{-(\underline{\tau} + r\overline{\tau})}}{C_{\mathrm{E}}\hat{\delta}}
\end{align*}
which is a constant. Thus, $\mathcal{P}^{q + r}(u_q) \in \mathcal{C}_0$ where we define the cone
\begin{align*}
\mathcal{C}_0 = \{(\tilde{w}, w_n) \in \mathbb R^{n - 1} \times \mathbb R: w_n > 2C_{\mathrm{E}}^{-1}\hat{\delta}^{-1}e^{-(\underline{\tau} + r\overline{\tau})} \|\tilde{w}\|\}.
\end{align*}
Let $\iota \in G$ be a translation by $-\pi_{\mathbb R^{n - 1}}(\pi_{\mathbb H^n}(u_q')) = -u_q^+$ followed by a dilation by a factor of $\langle \pi_{\mathbb H^n}(u_q'), e_n \rangle^{-1} = e^{\tau_\alpha(u_q)}$. Then, $\iota u_q' = (e_n, -e_n)$ and hence by \cref{itm:0IsInView} of \cref{lem:ChooseAFarAway_v_SuchThat0IsInView}, we have $W_{\delta_2}^{\mathrm{ss}}(\mathcal{P}^{q + r}(u_q)) \cap \mathsf{\Omega} \subset [\mathcal{P}^{q + r}(u_q), g_z\mathsf{S}_z]$ and $\Phi(W_{\delta_2}^{\mathrm{ss}}(\iota \mathcal{P}^{q + r}(u_q))) \supset \partial_\infty \mathbb H^n \setminus B_{C_{\mathrm{E}}^{-1}\delta_1e^{-\underline{\tau}}}^{\mathrm{E}}(0)$. Applying $\iota^{-1}$, we get
\begin{align*}
\Phi(W_{\delta_2}^{\mathrm{ss}}(\mathcal{P}^{q + r}(u_q))) \supset \partial_\infty \mathbb H^n \setminus B_{C_{\mathrm{E}}^{-1}\delta_1 e^{-\tau_\alpha(u_q) - \underline{\tau}}}^{\mathrm{E}}(u_q^+) \supset \partial_\infty \mathbb H^n \setminus B_{C_{\mathrm{E}}^{-1}\delta_1 e^{-\tau_\alpha(u)}}^{\mathrm{E}}(u_q^+).
\end{align*}
Recalling \cref{eqn:u_q^+_HorizontalCoordinateBounds}, we have $0 \in \Phi(W_{\delta_2}^{\mathrm{ss}}(\mathcal{P}^{q + r}(u_q)))$. Also recalling $0, u_q^+ \in \Lambda(\Gamma)$, there exists $v_q' \in W_{\delta_2}^{\mathrm{ss}}(\mathcal{P}^{q + r}(u_q)) \cap \mathsf{\Omega} \subset [\mathcal{P}^{q + r}(u_q), g_z\mathsf{S}_z]$ with $(v_q')^+ = u_q^+$ and $(v_q')^- = 0$. Then, $\pi_{\mathbb H^n}(v_q') \in B_{\hat{\delta}}(\mathcal{C}_0)$. It follows that $\gamma \subset B_{\hat{\delta}}(\mathcal{C}_0)$ for the geodesic ray $\gamma = \{\pi_{\mathbb H^n}(v_q' a_{-t}) \in \mathbb H^n: t \geq \underline{\tau}\}$ because $(v_q')^- = 0$. Since $\hat{\mathsf{S}}_y \subset \mathsf{S}_y$ is dense, using the Anosov property, there exists $v_q'' \in [\mathcal{P}^{q + r}(u_q), g_z\hat{\mathsf{S}}_z]$ such that $\pi_{\mathbb H^n}(w_q) \in B_{\hat{\delta}}(\gamma) \subset B_{2\hat{\delta}}(\mathcal{C}_0)$ where $w_q = \mathcal{P}^{-(q + r - N_T)}(v_q'') \in g_{z'}\mathsf{R}_{z'}$ for some $g_{z'} \in \Gamma$. By the topologically mixing property of $T$, there exists $w_q' \in g_{z'}\hat{\mathsf{R}}_{z'}$ such that $\mathcal{P}^{- N_T}(w_q') \in h_q\hat{\mathsf{R}}_y$ for some $h_q \in \Gamma$. Let $v_q = \mathcal{P}^{q + r - N_T}([w_q, w_q']) \in [\mathcal{P}^{q + r}(u_q), g_z\hat{\mathsf{S}}_z]$. Then, $\mathcal{P}^{-(q + r - N_T)}(v_q) \in g_{z'}\hat{\mathsf{R}}_{z'}$ and $\mathcal{P}^{-(q + r)}(v_q) \in h_q\hat{\mathsf{R}}_y$. The first fact we can conclude is $h_q \in H(y, z)$ by \cref{lem:ProducingElementsofH}. The second fact we can conclude is $h_q\pi_{\mathbb H^n}(u) \in B_{2\hat{\delta}}(\mathcal{P}^{-(q + r)}(v_q))$ and hence $h_q\pi_{\mathbb H^n}(u) \in B_{N_T\overline{\tau} + 4\hat{\delta}}(B_{\hat{\delta}}(\gamma)) \subset B_{N_T\overline{\tau} + 4\hat{\delta}}(B_{2\hat{\delta}}(\mathcal{C}_0))$.

Define the cone $\mathcal{C} = B_{N_T\overline{\tau} + 4\hat{\delta}}(B_{2\hat{\delta}}(\mathcal{C}_0))$. Then, $\{h_q\}_{q \in \mathbb N} \subset H(y, z)$ is a sequence such that $\{h_q\pi_{\mathbb H^n}(u)\}_{q \in \mathbb N} \subset \mathcal{C}$. In fact, since $h_q\pi_{\mathbb H^n}(u) \in B_{N_T\overline{\tau} + 4\hat{\delta}}(B_{\hat{\delta}}(\gamma))$ for all $q \in \mathbb N$ and $\lim_{q \to \infty} u_q^+ = 0$, we also have $\lim_{q \to \infty} h_q\pi_{\mathbb H^n}(u) = 0$. Hence, $u^+ = 0 \in \Lambda_{\mathrm{r}}(H(y, z))$.
\end{proof}

\begin{proof}[Proof of \cref{thm:H(yz)FiniteIndexInGamma}]
First, $\Lambda_{\mathrm{r}}(H(y, z)) = \Lambda(H(y, z))$ from \cref{pro:LimitSetH(yy)RadialEqualsLimitSetH(yy)EqualsLimitSetGamma} implies that $H(y, z)$ is convex cocompact by \cref{thm:ConvexCocompactIFFLimitSetIsRadial}. Since $\Gamma$ is nonelementary and $\Lambda(H(y, z)) = \Lambda(\Gamma)$ by \cref{pro:LimitSetH(yy)RadialEqualsLimitSetH(yy)EqualsLimitSetGamma}, we can apply \cite[Theorem 1]{SS92} to conclude that $H(y, z) < \Gamma$ is a finite index subgroup.
\end{proof}

\section{\texorpdfstring{$L^2$}{L-2}-flattening lemma}
\label{sec:L2FlatteningLemma}
In this section, we prove the following $L^2$-flattening lemma. We generalize the arguments in \cite[Appendix]{MOW17} due to Bourgain--Kontorovich--Magee.

\begin{lemma}
\label{lem:L2FlatteningLemma}
There exist $C > 0$, $C_0 > 0$, and $l \in \mathbb N$ such that for all $\xi \in \mathbb C$ with $|a| < a_0'$, square-free ideals $\mathfrak{q} \subset \mathcal{O}_{\mathbb K}$ coprime to $\mathfrak{q}_0$, $x \in \Sigma^+$, integers $C_0 \log(N_{\mathbb K}(\mathfrak{q})) \leq r < s$ with $r \in l\mathbb Z$, admissible sequences $(\alpha_s, \alpha_{s - 1}, \dotsc, \alpha_{r + 1})$, and $\phi \in E_\mathfrak{q}^\mathfrak{q}$ with $\|\phi\|_2 = 1$, we have
\begin{align*}
\left\|\mu_{(\alpha_s, \alpha_{s - 1}, \dotsc, \alpha_{r + 1})}^{\xi, \mathfrak{q}, x} * \phi\right\|_2 &\leq C N_{\mathbb K}(\mathfrak{q})^{-\frac{1}{3}} \left\|\nu_{(\alpha_s, \alpha_{s - 1}, \dotsc, \alpha_{r + 1})}^{a, \mathfrak{q}, x}\right\|_1; \\
\left\|\hat{\mu}_{(\alpha_s, \alpha_{s - 1}, \dotsc, \alpha_{r + 1})}^{a, \mathfrak{q}, x} * \phi\right\|_2 &\leq C N_{\mathbb K}(\mathfrak{q})^{-\frac{1}{3}} \left\|\nu_{(\alpha_s, \alpha_{s - 1}, \dotsc, \alpha_{r + 1})}^{a, \mathfrak{q}, x}\right\|_1.
\end{align*}
\end{lemma}

The proof uses two tools. The first is \cref{lem:ConvolutionBoundOnE_q^q} derived from lower bounds of nontrivial irreducible representations of Chevalley groups. The second is the expander machinery of Golsefidy--Varj\'{u} \cite{GV12} which we cannot use directly in its raw form but culminates in \cref{lem:ExpanderMachineryBound}. Due to \cref{lem:muHatLessThanCnu}, we focus on $\nu_0^{a, \mathfrak{q}, x, r}$ and our goal is to use the expander machinery to obtain bounds on the operator norm which requires the measure to be ``nearly flat''. Although $\nu_0^{a, \mathfrak{q}, x, r}$ is not nearly flat, it suffices to estimate $\nu_0^{a, \mathfrak{q}, x, r}$ by $\nu_1^{a, \mathfrak{q}, x, r}$ which \emph{breaks up} into convolutions of nearly flat measures. The following is the procedure to do exactly that.

In this section, we fix any $p > p_0$ from \cref{cor:Z-DenseInSimplyConnectedCoverGAndTraceFieldK} so that the corollary applies when it is needed in \cref{lem:GV_Expander}. For the purposes of proving \cref{lem:L2FlatteningLemma}, we will fix $\xi \in \mathbb C$ with $|a| < a_0'$, $x \in \Sigma^+$, $r \in \mathbb N$ with factorization $r = r'l$ for some fixed $r' \in \mathbb N$ and some fixed integer $l > p$ henceforth in this section. For all ideals $\mathfrak{q} \subset \mathcal{O}_{\mathbb K}$ coprime to $\mathfrak{q}_0$, for all admissible sequences $(\alpha_s, \alpha_{s - 1}, \dotsc, \alpha_{r + 1})$, denote the measures $\mu_{(\alpha_s, \alpha_{s - 1}, \dotsc, \alpha_{r + 1})}^{\xi, \mathfrak{q}, x}$, $\nu_0^{a, \mathfrak{q}, x, r}$, $\hat{\mu}_{(\alpha_s, \alpha_{s - 1}, \dotsc, \alpha_{r + 1})}^{a, \mathfrak{q}, x}$, and $\nu_{(\alpha_s, \alpha_{s - 1}, \dotsc, \alpha_{r + 1})}^{a, \mathfrak{q}, x}$ by $\mu_{(\alpha_s, \alpha_{s - 1}, \dotsc, \alpha_{r + 1})}^\mathfrak{q}$, $\nu_0^\mathfrak{q}$, $\hat{\mu}_{(\alpha_s, \alpha_{s - 1}, \dotsc, \alpha_{r + 1})}^\mathfrak{q}$, and $\nu_{(\alpha_s, \alpha_{s - 1}, \dotsc, \alpha_{r + 1})}^\mathfrak{q}$ respectively.

Let $\alpha^r$ be an admissible sequence. To better facilitate manipulations of sequences, we introduce the following additional notations. Define
\begin{align*}
\alpha_j^l &= (\alpha_{jl}, \alpha_{jl - 1}, \dotsc, \alpha_{(j - 1)l + 1}); \\
\alpha_j^{(p)_1} &= (\alpha_{jl}, \alpha_{jl - 1}, \dotsc, \alpha_{jl - p + 1}); \\
\alpha_j^{(l - p)_2} &= (\alpha_{jl - p}, \alpha_{jl - p - 1}, \dotsc, \alpha_{(j - 1)l + 1})
\end{align*}
for all $1 \leq j \leq r'$. For example, with these notations and conventions we have $\alpha^r = \big(\alpha_{r'}^l, \alpha_{r' - 1}^l, \dotsc,  \alpha_1^l\big) = (\alpha_r, \alpha_{r - 1}, \dotsc, \alpha_1)$ and $\alpha_j^l = \big(\alpha_j^{(p)_1}, \alpha_j^{(l - p)_2}\big)$ for all $1 \leq j \leq r'$. We also have $\sigma^k(\alpha^j) = \alpha^{j - k}$ for all $1 \leq j \leq r$ and $0 \leq k \leq j - 1$.

For all admissible sequences $\alpha^r$, we compute that
\begin{align*}
&f_r^{(a)}(\alpha^r, x) = \sum_{k = 0}^{r - 1} f^{(a)}(\sigma^k(\alpha^r, x)) = \sum_{k = 0}^{r - 1} f^{(a)}(\alpha^{r - k}, x) \\
={}&\sum_{k = 0}^{2l - p - 1} f^{(a)}(\alpha^{2l - p - k}, x) + \sum_{j = 0}^{r' - 3} \sum_{k = 0}^{l - 1} f^{(a)}(\alpha^{r - p - (jl + k)}, x) + \sum_{k = 0}^{p - 1} f^{(a)}(\alpha^{r - k}, x) \\
={}&\sum_{k = 0}^{2l - p - 1} f^{(a)}(\sigma^k(\alpha^{2l - p}, x)) + \sum_{j = 0}^{r' - 3} \sum_{k = 0}^{l - 1} f^{(a)}(\sigma^k(\alpha^{r - p - jl}, x)) + \sum_{k = 0}^{p - 1} f^{(a)}(\sigma^k(\alpha^r, x)) \\
={}&f_{2l - p}^{(a)}(\alpha^{2l - p}, x) + \sum_{j = 2}^{r' - 1} f_l^{(a)}(\alpha^{(j + 1)l - p}, x) + f_p^{(a)}(\alpha^r, x).
\end{align*}
We can estimate each term in the sum above so that in the $j$\textsuperscript{th} term, we remove dependence on $\alpha_k^{(p)_1}$ for all distinct $1 \leq j, k \leq r'$.

\begin{remark}
Such an estimate is not required for $j = 1$ since the first term does not have any dependence on $\alpha_k^{(p)_1}$ for all $2 \leq k \leq r'$.
\end{remark}

\begin{lemma}
\label{lem:Estimate_f_ToRemoveDependence}
There exists $C > 0$ such that for all admissible sequences $\alpha^r$, we have
\begin{align*}
\big|f_l^{(a)}(\alpha^{(j + 1)l - p}, x) - f_l^{(a)}\big(\alpha_{j + 1}^{(l - p)_2}, \alpha_j^l, \omega\big(\alpha_j^l\big)\big)\big| &\leq C \theta^l \qquad \text{for all $2 \leq j \leq r' - 1$}; \\
\big|f_p^{(a)}(\alpha^r, x) - f_p^{(a)}\big(\alpha_{r'}^l, \omega\big(\alpha_{r'}^l\big)\big)\big| &\leq C \theta^l
\end{align*}
where $C$ is independent of $|a| < a_0'$, $x \in \Sigma^+$, $r \in \mathbb N$ and its factorization $r = r'l$ with $l > p$.
\end{lemma}

\begin{proof}
Fix $C = \frac{T_0 \theta^{1 - p}}{1 - \theta}$. Let $\alpha^r$ be an admissible sequence. We make the estimate
\begin{align*}
&\big|f_l^{(a)}(\alpha^{(j + 1)l - p}, x) - f_l^{(a)}\big(\alpha_{j + 1}^{(l - p)_2}, \alpha_j^l, \omega\big(\alpha_j^l\big)\big)\big| \\
\leq{}&\sum_{k = 0}^{l - 1} \big|f^{(a)}\big(\sigma^k\big(\alpha_{j + 1}^{(l - p)_2}, \alpha_j^l, \alpha_{j - 1}^{l}, \dotsc, \alpha_1^{l}, x\big)\big) - f^{(a)}\big(\sigma^k\big(\alpha_{j + 1}^{(l - p)_2}, \alpha_j^l, \omega\big(\alpha_j^l\big)\big)\big)\big| \\
\leq{}&\sum_{k = 0}^{l - 1} \Lip_{d_\theta}(f^{(a)}) \cdot d_\theta\big(\sigma^k\big(\alpha_{j + 1}^{(l - p)_2}, \alpha_j^l, \alpha_{j - 1}^{l}, \dotsc, \alpha_1^{l}, x\big), \sigma^k\big(\alpha_{j + 1}^{(l - p)_2}, \alpha_j^l, \omega\big(\alpha_j^l\big)\big)\big) \\
\leq{}&\Lip_{d_\theta}(f^{(a)}) \sum_{k = 0}^{l - 1} \theta^{2l - p - k} \leq \frac{T_0 \theta^{l - p + 1}}{1 - \theta} = C \theta^l
\end{align*}
for all $2 \leq j \leq r' - 1$. The second inequality follows from a similar calculation.
\end{proof}

To make sense of the notations in what follows, we make the convention that $\alpha_j^{(l - p)_2}$ is the empty sequence for all admissible sequences $\alpha^r$ and $j \in \{0, r' + 1\}$. In light of the calculations and \cref{lem:Estimate_f_ToRemoveDependence} above, for all ideals $\mathfrak{q} \subset \mathcal{O}_{\mathbb K}$ coprime to $\mathfrak{q}_0$ and admissible sequences $\alpha^r$, define the coefficients
\begin{align*}
E\big(\alpha_{j + 1}^{(l - p)_2}, \alpha_j^l\big) =
\begin{cases}
e^{f_{2l - p}^{(a)}(\alpha^{2l - p}, x)}, & j = 1 \\
e^{f_l^{(a)}(\alpha_{j + 1}^{(l - p)_2}, \alpha_j^l, \omega(\alpha_j^l))}, & 2 \leq j \leq r' - 1 \\
e^{f_p^{(a)}(\alpha_{r'}^l, \omega(\alpha_{r'}^l))}, & j = r'
\end{cases}
\end{align*}
and the measures
\begin{align*}
\eta^\mathfrak{q}\big(\alpha_{j + 1}^{(l - p)_2}, \alpha_j^{(l - p)_2}\big) =
\begin{cases}
\delta_{\mathtt{c}_\mathfrak{q}(\alpha_1, x)}, & j = 0 \\
\sum_{\alpha_j^{(p)_1}} E\big(\alpha_{j + 1}^{(l - p)_2}, \alpha_j^l\big) \delta_{\mathtt{c}_\mathfrak{q}^l(\alpha_{jl + 1}, \alpha^{jl}, x)}, & 1 \leq j \leq r'
\end{cases}
\end{align*}
where we show the dependence of the admissible choices of $\alpha_j^{(p)_1}$ on $\alpha_{j + 1}^{(l - p)_2}$ and $\alpha_j^{(l - p)_2}$ (or more precisely only on the last entry of $\alpha_{j + 1}^{(l - p)_2}$ and the first entry of $\alpha_j^{(l - p)_2}$). The measures above satisfy a property as shown in \cref{lem:NearlyFlat} which we call \emph{nearly flat}.

\begin{lemma}
\label{lem:NearlyFlat}
There exists $C > 1$ such that for all $1 \leq j \leq r'$, and for all pairs of admissible sequences $\big(\alpha_{j + 1}^{(l - p)_2}, \alpha_j^l\big)$ and $\big(\tilde{\alpha}_{j + 1}^{(l - p)_2}, \tilde{\alpha}_j^l\big)$ with $\big(\alpha_{j + 1}^{(l - p)_2}, \alpha_j^{(l - p)_2}\big) = \big(\tilde{\alpha}_{j + 1}^{(l - p)_2}, \tilde{\alpha}_j^{(l - p)_2}\big)$, we have
\begin{align*}
\frac{E\big(\tilde{\alpha}_{j + 1}^{(l - p)_2}, \tilde{\alpha}_j^l\big)}{E\big(\alpha_{j + 1}^{(l - p)_2}, \alpha_j^l\big)} \leq C
\end{align*}
where $C$ is independent of $|a| < a_0'$, $x \in \Sigma^+$, $r \in \mathbb N$ and its factorization $r = r'l$ with $l > p$.
\end{lemma}

\begin{proof}
Fix $C = e^{T_0\left(\frac{\theta}{1 - \theta} + p\right)} > 1$. Let $1 \leq j \leq r'$ be an integer and $\big(\alpha_{j + 1}^{(l - p)_2}, \alpha_j^l\big)$ and $\big(\tilde{\alpha}_{j + 1}^{(l - p)_2}, \tilde{\alpha}_j^l\big)$ be two pairs of admissible sequences with $\big(\alpha_{j + 1}^{(l - p)_2}, \alpha_j^{(l - p)_2}\big) = \big(\tilde{\alpha}_{j + 1}^{(l - p)_2}, \tilde{\alpha}_j^{(l - p)_2}\big)$. For the $j = 1$ case, we have the calculation
\begin{align*}
&\big|f_{2l - p}^{(a)}(\alpha^{2l - p}, x) - f_{2l - p}^{(a)}(\tilde{\alpha}^{2l - p}, x)\big| \\
\leq{}&\sum_{k = 0}^{2l - p - 1} \big|f^{(a)}(\sigma^k(\alpha^{2l - p}, x)) - f^{(a)}(\sigma^k(\tilde{\alpha}^{2l - p}, x))\big| \\
\leq{}&\sum_{k = 0}^{l - p - 1} \Lip_{d_\theta}(f^{(a)}) \cdot d_\theta(\sigma^k(\alpha^{2l - p}, x), \sigma^k(\tilde{\alpha}^{2l - p}, x)) \\
{}&+ \sum_{k = l - p}^{l - 1} \Lip_{d_\theta}(f^{(a)}) \cdot d_\theta(\sigma^k(\alpha^{2l - p}, x), \sigma^k(\tilde{\alpha}^{2l - p}, x)) \\
{}&+ \sum_{k = l}^{2l - p - 1} \Lip_{d_\theta}(f^{(a)}) \cdot d_\theta(\sigma^k(\alpha^{2l - p}, x), \sigma^k(\tilde{\alpha}^{2l - p}, x)) \\
\leq{}&\Lip_{d_\theta}(f^{(a)}) \left(\sum_{k = 0}^{l - p - 1} \theta^{l - p - k} + \sum_{k = l - p}^{l - 1} 1\right) \leq T_0\left(\frac{\theta}{1 - \theta} + p\right) = \log(C).
\end{align*}
We have similar calculations for the $2 \leq j \leq r' - 1$ and $j = r'$ cases. The lemma now follows from definitions.
\end{proof}

For all ideals $\mathfrak{q} \subset \mathcal{O}_{\mathbb K}$ coprime to $\mathfrak{q}_0$, we also define the measure
\begin{align*}
\nu_1^\mathfrak{q} = \sum_{\alpha_1^{(l - p)_2}, \alpha_2^{(l - p)_2}, \dotsc, \alpha_{r'}^{(l - p)_2}} \mathop{\bigast}\limits_{j = 0}^{r'} \eta^\mathfrak{q}\big(\alpha_{j + 1}^{(l - p)_2}, \alpha_j^{(l - p)_2}\big)
\end{align*}
which in particular consists of convolutions of nearly flat measures. \cref{lem:EstimateNu} shows that we can estimate $\nu_0^\mathfrak{q}$ with $\nu_1^\mathfrak{q}$ up to a multiplicative constant and vice versa where the constant has exponential dependence on $r'$ in the factorization $r = r'l$.

\begin{lemma}
\label{lem:EstimateNu}
There exists $C > 0$ such that for all ideals $\mathfrak{q} \subset \mathcal{O}_{\mathbb K}$ coprime to $\mathfrak{q}_0$, we have $\nu_0^\mathfrak{q} \leq e^{r' C\theta^l}\nu_1^\mathfrak{q}$ and $\nu_1^\mathfrak{q} \leq e^{r' C\theta^l} \nu_0^\mathfrak{q}$ where $C$ is independent of $|a| < a_0'$, $x \in \Sigma^+$, $r \in \mathbb N$ and its factorization $r = r'l$ with $l > p$.
\end{lemma}

\begin{proof}
Fix $C > 0$ to be the one from \cref{lem:Estimate_f_ToRemoveDependence}. Let $\mathfrak{q} \subset \mathcal{O}_{\mathbb K}$ be an ideal coprime to $\mathfrak{q}_0$. First we have
\begin{align*}
\nu_0^\mathfrak{q} &= \sum_{\alpha^r} e^{f_r^{(a)}(\alpha^r, x)} \delta_{\mathtt{c}_\mathfrak{q}^{r + 1}(\alpha_{r + 1}, \alpha^r, x)} \\
&= \sum_{\alpha^r} e^{f_{2l - p}^{(a)}(\alpha^{2l - p}, x) + \sum_{j = 2}^{r' - 1} f_l^{(a)}(\alpha^{(j + 1)l - p}, x) + f_p^{(a)}(\alpha^r, x)} \delta_{\mathtt{c}_\mathfrak{q}^{r + 1}(\alpha_{r + 1}, \alpha^r, x)}.
\end{align*}
Now by construction, for fixed sequences $\alpha_1^{(l - p)_2}, \alpha_2^{(l - p)_2}, \dotsc, \alpha_{r'}^{(l - p)_2}$, the terms $E\big(\alpha_{j + 1}^{(l - p)_2}, \alpha_j^l\big)$ and $\mathtt{c}_\mathfrak{q}^l\big(\alpha_{jl + 1}, \alpha^{jl}, x\big) = \mathtt{c}_\mathfrak{q}^l\big(\alpha_{jl + 1}, \alpha_j^{(p)_1}, \alpha_j^{(l - p)_2}, \omega\big(\alpha_j^{(l - p)_2}\big)\big)$ depend only on the choice of $\alpha_j^{(p)_1}$ and not on $\alpha_k^{(p)_1}$ for all distinct $1 \leq j, k \leq r'$. We also note that $\mathtt{c}_\mathfrak{q}(\alpha_1, x)$ depends only on the choice of $\alpha_1^{(l - p)_2}$ since $l - p \geq 1$. So we can do the manipulations
\begin{align*}
\nu_1^\mathfrak{q} ={}&\sum_{\alpha_1^{(l - p)_2}, \alpha_2^{(l - p)_2}, \dotsc, \alpha_{r'}^{(l - p)_2}} \delta_{\mathtt{c}_\mathfrak{q}(\alpha_1, x)} * \left(\mathop{\bigast}\limits_{j = 1}^{r'} \sum_{\alpha_j^{(p)_1}} E\big(\alpha_{j + 1}^{(l - p)_2}, \alpha_j^l\big) \delta_{\mathtt{c}_\mathfrak{q}^l(\alpha_{jl + 1}, \alpha^{jl}, x)}\right) \\
={}&\sum_{\alpha_1^{(l - p)_2}, \alpha_2^{(l - p)_2}, \dotsc, \alpha_{r'}^{(l - p)_2}} \delta_{\mathtt{c}_\mathfrak{q}(\alpha_1, x)} \\
{}&* \sum_{\alpha_1^{(p)_1}, \alpha_2^{(p)_1}, \dotsc, \alpha_{r'}^{(p)_1}} \mathop{\bigast}\limits_{j = 1}^{r'} E\big(\alpha_{j + 1}^{(l - p)_2}, \alpha_j^l\big) \delta_{\mathtt{c}_\mathfrak{q}^l(\alpha_{jl + 1}, \alpha^{jl}, x)} \\
={}&\sum_{\alpha_1^{(l - p)_2}, \alpha_2^{(l - p)_2}, \dotsc, \alpha_{r'}^{(l - p)_2}} \left(\sum_{\alpha_1^{(p)_1}, \alpha_2^{(p)_1}, \dotsc, \alpha_{r'}^{(p)_1}} \left(\prod_{j = 1}^{r'} E\big(\alpha_{j + 1}^{(l - p)_2}, \alpha_j^l\big) \right)\right. \\
{}&\left.\cdot \delta_{\mathtt{c}_\mathfrak{q}(\alpha_1, x)} * \delta_{\mathtt{c}_\mathfrak{q}^l(\alpha_{l + 1}, \alpha^{l}, x)} * \delta_{\mathtt{c}_\mathfrak{q}^l(\alpha_{2l + 1}, \alpha^{2l}, x)} * \dotsb * \delta_{\mathtt{c}_\mathfrak{q}^l(\alpha_{r + 1}, \alpha^{r'l}, x)} \rule{0cm}{0.8cm}\right) \\
={}&\sum_{\alpha^r} \left(\prod_{j = 1}^{r'} E\big(\alpha_{j + 1}^{(l - p)_2}, \alpha_j^l\big)\right) \delta_{\mathtt{c}_\mathfrak{q}^{r + 1}(\alpha_{r + 1}, \alpha^r, x)}.
\end{align*}
Hence the lemma follows by comparing the above two expressions for $\nu_0^\mathfrak{q}$ and $\nu_1^\mathfrak{q}$ and using \cref{lem:Estimate_f_ToRemoveDependence}.
\end{proof}

Let $\mathfrak{q} \subset \mathcal{O}_{\mathbb K}$ be an ideal coprime to $\mathfrak{q}_0$, $1 \leq j \leq r'$ be an integer, and $\big(\alpha_{j + 1}^{(l - p)_2}, \alpha_j^l\big)$ and $\big(\tilde{\alpha}_{j + 1}^{(l - p)_2}, \tilde{\alpha}_j^l\big)$ be pairs of admissible sequences such that $\big(\alpha_{j + 1}^{(l - p)_2}, \alpha_j^{(l - p)_2}\big) = \big(\tilde{\alpha}_{j + 1}^{(l - p)_2}, \tilde{\alpha}_j^{(l - p)_2}\big)$. We have
\begin{align*}
\mathtt{c}_\mathfrak{q}^l\big(\alpha_{jl + 1}, \alpha^{jl}, x\big) = \prod_{k = 0}^p \mathtt{c}_\mathfrak{q}(\alpha_{jl + 1 - k}, \alpha_{jl - k}) \prod_{k = p + 1}^{l - 1} \mathtt{c}_\mathfrak{q}(\alpha_{jl + 1 - k}, \alpha_{jl - k})
\end{align*}
for the first sequence and similarly for the second one. We write it in this fashion because $\mathtt{c}_\mathfrak{q}(\alpha_{jl + 1 - k}, \alpha_{jl - k}) = \mathtt{c}_\mathfrak{q}(\tilde{\alpha}_{jl + 1 - k}, \tilde{\alpha}_{jl - k})$ for all $p + 1 \leq k \leq l - 1$. We then calculate that
\begin{align*}
&\mathtt{c}_\mathfrak{q}^l\big(\alpha_{jl + 1}, \alpha^{jl}, x\big) \mathtt{c}_\mathfrak{q}^l\big(\tilde{\alpha}_{jl + 1}, \tilde{\alpha}^{jl}, x\big)^{-1} \\
={}&\prod_{k = 0}^p \mathtt{c}_\mathfrak{q}(\alpha_{jl + 1 - k}, \alpha_{jl - k}) \prod_{k = p + 1}^{l - 1} \mathtt{c}_\mathfrak{q}(\alpha_{jl + 1 - k}, \alpha_{jl - k}) \\
{}&\cdot \left(\prod_{k = p + 1}^{l - 1} \mathtt{c}_\mathfrak{q}(\alpha_{jl + 1 - k}, \alpha_{jl - k})\right)^{-1} \left(\prod_{k = 0}^p \mathtt{c}_\mathfrak{q}(\tilde{\alpha}_{jl + 1 - k}, \tilde{\alpha}_{jl - k})\right)^{-1} \\
={}&\prod_{k = 0}^p \mathtt{c}_\mathfrak{q}(\alpha_{jl + 1 - k}, \alpha_{jl - k}) \prod_{k = 0}^p \mathtt{c}_\mathfrak{q}(\tilde{\alpha}_{jl - p + 1 + k}, \tilde{\alpha}_{jl - p + k})^{-1}.
\end{align*}

Now, by the Zariski density and full trace field properties of the return trajectory subgroups from \cref{cor:Z-DenseInSimplyConnectedCoverGAndTraceFieldK}, we can use the strong approximation theorem together with the expander machinery of Golsefidy--Varj\'{u} \cite{GV12} to obtain the following lemma regarding spectral gap.

\begin{lemma}
\label{lem:GV_Expander}
There exists $\epsilon \in (0, 1)$ such that for all $1 \leq j \leq r'$, pairs of admissible sequences $\big(\alpha_{j + 1}^{(l - p)_2}, \alpha_j^{(l - p)_2}\big)$, square-free ideals $\mathfrak{q} \subset \mathcal{O}_{\mathbb K}$ coprime to $\mathfrak{q}_0$, and $\phi \in L_0^2(\tilde{G}_\mathfrak{q})$ with $\|\phi\|_2 = 1$, there exist admissible sequences $\big(\beta_{jl + 1}, \beta_j^{(p)_1}, \beta_{jl - p}\big)$ and $\big(\tilde{\beta}_{jl + 1}, \tilde{\beta}_j^{(p)_1}, \tilde{\beta}_{jl - p}\big)$ with $\beta_{jl + 1} = \tilde{\beta}_{jl + 1} = \alpha_{jl + 1}$ and $\beta_{jl - p} = \tilde{\beta}_{jl - p} = \alpha_{jl - p}$ such that
\begin{align*}
\|\delta_g * \phi - \phi\|_2 \geq \epsilon
\end{align*}
where $g = \prod_{k = 0}^p \mathtt{c}_\mathfrak{q}(\beta_{jl + 1 - k}, \beta_{jl - k}) \prod_{k = 0}^p \mathtt{c}_\mathfrak{q}(\tilde{\beta}_{jl - p + 1 + k}, \tilde{\beta}_{jl - p + k})^{-1}$ and $\epsilon$ is independent of $r \in \mathbb N$ and its factorization $r = r'l$ with $l > p$.
\end{lemma}

\begin{proof}
Uniformity of $\epsilon$ with respect to $r \in \mathbb N$ with factorization $r = r'l$ with $l > p$, integers $1 \leq j \leq r'$, and pairs of admissible sequences $\big(\alpha_{j + 1}^{(l - p)_2}, \alpha_j^{(l - p)_2}\big)$ is trivial since $\epsilon$ only depends on the first entry $\alpha_{jl + 1} \in \mathcal A$ of $\alpha_{j + 1}^{(l - p)_2}$ and the last entry $\alpha_{jl - p} \in \mathcal A$ of $\alpha_j^{(l - p)_2}$ and there are only a finite number of such elements. So let $1 \leq j \leq r'$ be an integer and $\big(\alpha_{j + 1}^{(l - p)_2}, \alpha_j^{(l - p)_2}\big)$ be a pair of admissible sequences. Denote $\tilde{S}^p(\alpha_{jl + 1}, \alpha_{jl - p})$ by $\tilde{S}^p$ and $\tilde{H}^p(\alpha_{jl + 1}, \alpha_{jl - p})$ by $\tilde{H}^p$. For all ideals $\mathfrak{q} \subset \mathcal{O}_{\mathbb K}$, let $\tilde{S}^p_\mathfrak{q} = \pi_\mathfrak{q}(\ker(\tilde{\pi})) \backslash \pi_\mathfrak{q}(\tilde{S}^p)$ and $\tilde{H}^p_\mathfrak{q} = \pi_\mathfrak{q}(\ker(\tilde{\pi})) \backslash \pi_\mathfrak{q}(\tilde{H}^p) = \langle\tilde{S}^p_\mathfrak{q}\rangle$. Recalling the strong approximation theorem of Weisfeiler (see \cite[Theorem 10.1]{Wei84} and its proof), \cref{cor:Z-DenseInSimplyConnectedCoverGAndTraceFieldK} implies that $\tilde{H}^p_\mathfrak{q} = \tilde{G}_\mathfrak{q}$ for all ideals $\mathfrak{q} \subset \mathcal{O}_{\mathbb K}$ coprime to $\mathfrak{q}_0$. By \cite[Corollary 6]{GV12}, we can furthermore conclude that the Cayley graphs $\Cay(\tilde{G}_\mathfrak{q}, \tilde{S}^p_\mathfrak{q}) = \Cay(\pi_\mathfrak{q}(\ker(\tilde{\pi})) \backslash \tilde{\mathbf{G}}(\mathcal{O}_\mathbb{K}/\mathfrak{q}), \pi_\mathfrak{q}(\ker(\tilde{\pi})) \backslash \pi_\mathfrak{q}(\tilde{S}^p))$ form expanders with respect to \emph{square-free} ideals $\mathfrak{q} \subset \mathcal{O}_{\mathbb K}$ coprime to $\mathfrak{q}_0$. For all square-free ideals $\mathfrak{q} \subset \mathcal{O}_{\mathbb K}$ coprime to $\mathfrak{q}_0$, let $A_\mathfrak{q}: L^2(\tilde{G}_\mathfrak{q}) \to L^2(\tilde{G}_\mathfrak{q})$ be the self-adjoint adjacency operator defined by
\begin{align*}
A_\mathfrak{q}(\phi) = \sum_{h \in \tilde{S}^p_\mathfrak{q}} \delta_h * \phi \qquad \text{for all $\phi \in L^2(\tilde{G}_\mathfrak{q})$}.
\end{align*}
Its largest eigenvalue is $\lambda_1(A_\mathfrak{q}) = \#\tilde{S}^p_\mathfrak{q}$ with the corresponding eigenvectors being the constant functions. We can choose $\epsilon \in (0, 1)$ coming from the expander property so that for all square-free ideals $\mathfrak{q} \subset \mathcal{O}_{\mathbb K}$ coprime to $\mathfrak{q}_0$, the next largest eigenvalue is $\lambda_2(A_\mathfrak{q}) \leq (1 - \epsilon)\#\tilde{S}^p_\mathfrak{q}$. This corresponds to the graph Laplacian $\Delta_\mathfrak{q}  = \Id_{L^2(\tilde{G}_\mathfrak{q})} - \frac{1}{\#\tilde{S}^p_\mathfrak{q}}A_\mathfrak{q}$ having the smallest eigenvalue $\lambda_1(\Delta_\mathfrak{q}) = 0$ with the corresponding eigenvectors being the constant functions and having the next smallest eigenvalue $\lambda_2(\Delta_\mathfrak{q}) \geq \epsilon$ for all square-free ideals $\mathfrak{q} \subset \mathcal{O}_{\mathbb K}$ coprime to $\mathfrak{q}_0$. Thus, for all square-free ideals $\mathfrak{q} \subset \mathcal{O}_{\mathbb K}$ coprime to $\mathfrak{q}_0$ and $\phi \in L_0^2(\tilde{G}_\mathfrak{q})$ with $\|\phi\|_2 = 1$, we have $\|\Delta_\mathfrak{q}(\phi)\|_2 = \left\|\frac{1}{\#\tilde{S}^p_\mathfrak{q}}A_\mathfrak{q}(\phi) - \phi\right\|_2 \geq \epsilon$ which implies
\begin{align*}
\sum_{h \in \tilde{S}^p_\mathfrak{q}} \|\delta_h * \phi - \phi\|_2 \geq \epsilon \cdot \#\tilde{S}^p_\mathfrak{q}.
\end{align*}
So, there exists $g \in \tilde{S}^p_\mathfrak{q}$ such that $\|\delta_g * \phi - \phi\|_2 \geq \epsilon$. But $\tilde{S}^p \subset \tilde{H}^p < \tilde{\Gamma}$ and recall the induced isomorphisms $\overline{\pi_{\mathfrak{q}}|_{\tilde{\Gamma}}}: \tilde{\Gamma}_\mathfrak{q} \backslash \tilde{\Gamma} \to \tilde{G}_\mathfrak{q}$ and $\overline{\tilde{\pi}|_{\tilde{\Gamma}}}: \tilde{\Gamma}_\mathfrak{q} \backslash \tilde{\Gamma} \to \Gamma_\mathfrak{q} \backslash \Gamma$. Following these isomorphisms, we can see that $g$ is in fact of the form
\begin{align*}
g = \prod_{k = 0}^p \mathtt{c}_\mathfrak{q}(\beta_{jl + 1 - k}, \beta_{jl - k}) \prod_{k = 0}^p \mathtt{c}_\mathfrak{q}(\tilde{\beta}_{jl - p + 1 + k}, \tilde{\beta}_{jl - p + k})^{-1}.
\end{align*}
\end{proof}

\begin{remark}
Although in the above proof we only know a priori that the Cayley graphs $\Cay(\tilde{\mathbf{G}}(\mathcal{O}_\mathbb{K}/\mathfrak{q}), \pi_\mathfrak{q}(\tilde{S}^p))$ form expanders with respect to square-free ideals $\mathfrak{q} \subset \mathcal{O}_{\mathbb K}$ coprime to $\mathfrak{q}_0$, it is easy to see using the Cheeger constant formulation for expanders that left quotients by $\pi_\mathfrak{q}(\ker(\tilde{\pi}))$ preserve this property.
\end{remark}

Let $\mathfrak{q} \subset \mathcal{O}_{\mathbb K}$ be an ideal coprime to $\mathfrak{q}_0$. For all measures $\eta$ on $\tilde{G}_\mathfrak{q}$, we define $\tilde{\eta}$ to be the operator $\tilde{\eta}: L^2(\tilde{G}_\mathfrak{q}) \to L^2(\tilde{G}_\mathfrak{q})$ acting by the convolution $\tilde{\eta}(\phi) = \eta * \phi$ for all $\phi \in L^2(\tilde{G}_\mathfrak{q})$ and $\tilde{\eta}^*$ will be its adjoint operator as usual. We also define $\eta^*$ to be the measure on $\tilde{G}_\mathfrak{q}$ defined by $\eta^*(g) = \overline{\eta(g^{-1})}$ for all $g \in \tilde{G}_\mathfrak{q}$. It can be checked that $\tilde{\eta}^* = \widetilde{\eta^*}$, i.e., $\tilde{\eta}^*(\phi) = \eta^* * \phi$ for all $\phi \in L^2(\tilde{G}_\mathfrak{q})$.

\begin{lemma}
\label{lem:EtaOperatorBound}
There exists $C \in (0, 1)$ such that for all square-free ideals $\mathfrak{q} \subset \mathcal{O}_{\mathbb K}$ coprime to $\mathfrak{q}_0$, integers $1 \leq j \leq r'$, pairs of admissible sequences $\big(\alpha_{j + 1}^{(l - p)_2}, \alpha_j^{(l - p)_2}\big)$, and $\phi \in L_0^2(\tilde{G}_\mathfrak{q})$ with $\|\phi\|_2 = 1$, we have
\begin{align*}
\left\|\eta^\mathfrak{q}\big(\alpha_{j + 1}^{(l - p)_2}, \alpha_j^{(l - p)_2}\big) * \phi\right\|_2 \leq C \left\|\eta^\mathfrak{q}\big(\alpha_{j + 1}^{(l - p)_2}, \alpha_j^{(l - p)_2}\big)\right\|_1
\end{align*}
where $C$ is independent of $|a| < a_0'$, $x \in \Sigma^+$, $r \in \mathbb N$ and its factorization $r = r'l$ with $l > p$.
\end{lemma}

\begin{proof}
Fix $\epsilon \in (0, 1)$ to be the one from \cref{lem:GV_Expander} and $C_0 > 1$ to be the $C$ from \cref{lem:NearlyFlat}. Fix $C = \sqrt{1 - \frac{\epsilon^2}{2C_0^2N^{2p}}} \in (0, 1)$. Let $\mathfrak{q} \subset \mathcal{O}_{\mathbb K}$ be a square-free ideal coprime to $\mathfrak{q}_0$, $1 \leq j \leq r'$ be an integer, $\big(\alpha_{j + 1}^{(l - p)_2}, \alpha_j^{(l - p)_2}\big)$ be a pair of admissible sequences, and $\phi \in L_0^2(\tilde{G}_\mathfrak{q})$ with $\|\phi\|_2 = 1$. Denote $\eta^\mathfrak{q}\big(\alpha_{j + 1}^{(l - p)_2}, \alpha_j^{(l - p)_2}\big)$ by $\eta$ and $E\big(\alpha_{j + 1}^{(l - p)_2}, \alpha_j^l\big)$ by $E\big(\alpha_j^{(p)_1}\big)$ for all admissible sequences $\big(\alpha_{jl + 1}, \alpha_j^{(p)_1}, \alpha_{jl - p}\big)$. Recalling that $E$ is real-valued, we can define the measure
\begin{align*}
\Pi &= \eta^* * \eta = \sum_{\alpha_j^{(p)_1}, \tilde{\alpha}_j^{(p)_1}} E\big(\alpha_j^{(p)_1}\big) E\big(\tilde{\alpha}_j^{(p)_1}\big) \delta_{\mathtt{c}_\mathfrak{q}^l(\tilde{\alpha}_{jl + 1}, \tilde{\alpha}^{jl}, x)^{-1}} * \delta_{\mathtt{c}_\mathfrak{q}^l(\alpha_{jl + 1}, \alpha^{jl}, x)} \\
&= \sum_{\alpha_j^{(p)_1}, \tilde{\alpha}_j^{(p)_1}} E\big(\alpha_j^{(p)_1}\big) E\big(\tilde{\alpha}_j^{(p)_1}\big) \delta_{\mathtt{c}_\mathfrak{q}^l(\alpha_{jl + 1}, \alpha^{jl}, x) \mathtt{c}_\mathfrak{q}^l(\tilde{\alpha}_{jl + 1}, \tilde{\alpha}^{jl}, x)^{-1}}
\end{align*}
where $\big(\tilde{\alpha}_{jl + 1}, \tilde{\alpha}_j^{(p)_1}, \tilde{\alpha}_j^{(l - p)_2}, \tilde{\alpha}^{(j - 1)l}\big) = \big(\alpha_{jl + 1}, \tilde{\alpha}_j^{(p)_1}, \alpha_j^{(l - p)_2}, \alpha^{(j - 1)l}\big)$ henceforth. To begin estimating $\|\eta * \phi\|_2$, we calculate that $0 \leq \|\eta * \phi\|_2^2 = \langle \eta * \phi, \eta * \phi \rangle = \langle \eta^* * \eta * \phi, \phi \rangle = \langle\tilde{\Pi}(\phi), \phi \rangle$. The calculations also show that $\tilde{\Pi} = \tilde{\eta}^* \tilde{\eta}$ is a self-adjoint positive semidefinite operator. It suffices to show $\langle\tilde{\Pi}(\phi), \phi \rangle \leq C^2\|\Pi\|_1$ since
\begin{align*}
\|\Pi\|_1 = \sum_{\alpha_j^{(p)_1}, \tilde{\alpha}_j^{(p)_1}} E\big(\alpha_j^{(p)_1}\big) E\big(\tilde{\alpha}_j^{(p)_1}\big) = \left(\sum_{\alpha_j^{(p)_1}} E\big(\alpha_j^{(p)_1}\big)\right)^2 = \|\eta\|_1^2.
\end{align*}
Before we estimate $\langle\tilde{\Pi}(\phi), \phi \rangle$, we first use \cref{lem:GV_Expander} to obtain an element
\begin{align*}
g &= \mathtt{c}_\mathfrak{q}^l\big(\beta_{jl + 1}, \beta^{jl}, x\big) \mathtt{c}_\mathfrak{q}^l\big(\tilde{\beta}_{jl + 1}, \tilde{\beta}^{jl}, x\big)^{-1} \\
&= \prod_{k = 0}^p \mathtt{c}_\mathfrak{q}(\beta_{jl + 1 - k}, \beta_{jl - k}) \prod_{k = 0}^p \mathtt{c}_\mathfrak{q}(\tilde{\beta}_{jl - p + 1 + k}, \tilde{\beta}_{jl - p + k})^{-1}
\end{align*}
for some admissible sequences $\big(\beta_{jl + 1}, \beta_j^{(p)_1}, \beta_{jl - p}\big)$ and $\big(\tilde{\beta}_{jl + 1}, \tilde{\beta}_j^{(p)_1}, \tilde{\beta}_{jl - p}\big)$ with $\beta_{jl + 1} = \tilde{\beta}_{jl + 1} = \alpha_{jl + 1}$ and $\beta_{jl - p} = \tilde{\beta}_{jl - p} = \alpha_{jl - p}$ such that $\|\delta_g * \phi - \phi\|_2 \geq \epsilon$. Expanding the norm, we have $\|\delta_g * \phi\|_2^2 - 2\Re \langle \delta_g * \phi, \phi\rangle + \|\phi\|_2^2 \geq \epsilon^2$. Thus
\begin{align*}
\Re \langle \delta_g * \phi, \phi\rangle \leq 1 - \frac{\epsilon^2}{2} \in (0, 1)
\end{align*}
using the fact that $\|\delta_g * \phi\|_2 = \|\phi\|_2 = 1$. Now we use this inequality to begin estimating $\langle\tilde{\Pi}(\phi), \phi \rangle$. Since $\tilde{\Pi}$ is self-adjoint, we have
\begin{align*}
&\langle \tilde{\Pi}(\phi), \phi\rangle = \Re \langle \tilde{\Pi}(\phi), \phi\rangle \\
={}& \Re \left\langle\sum_{\alpha_j^{(p)_1}, \tilde{\alpha}_j^{(p)_1}} E\big(\alpha_j^{(p)_1}\big) E\big(\tilde{\alpha}_j^{(p)_1}\big) \delta_{\mathtt{c}_\mathfrak{q}^l(\alpha_{jl + 1}, \alpha^{jl}, x) \mathtt{c}_\mathfrak{q}^l(\tilde{\alpha}_{jl + 1}, \tilde{\alpha}^{jl}, x)^{-1}} * \phi, \phi \right\rangle \\
={}&\sum_{\alpha_j^{(p)_1}, \tilde{\alpha}_j^{(p)_1}} E\big(\alpha_j^{(p)_1}\big) E\big(\tilde{\alpha}_j^{(p)_1}\big) \Re \langle \delta_{\mathtt{c}_\mathfrak{q}^l(\alpha_{jl + 1}, \alpha^{jl}, x) \mathtt{c}_\mathfrak{q}^l(\tilde{\alpha}_{jl + 1}, \tilde{\alpha}^{jl}, x)^{-1}} * \phi, \phi\rangle \\
\leq{}&E\big(\beta_j^{(p)_1}\big) E\big(\tilde{\beta}_j^{(p)_1}\big) \Re \langle \delta_g * \phi, \phi\rangle + \sum_{(\alpha_j^{(p)_1}, \tilde{\alpha}_j^{(p)_1}) \neq (\beta_j^{(p)_1}, \tilde{\beta}_j^{(p)_1})} E\big(\alpha_j^{(p)_1}\big) E\big(\tilde{\alpha}_j^{(p)_1}\big) \\
&{}\cdot |\langle \delta_{\mathtt{c}_\mathfrak{q}^l(\alpha_{jl + 1}, \alpha^{jl}, x) \mathtt{c}_\mathfrak{q}^l(\tilde{\alpha}_{jl + 1}, \tilde{\alpha}^{jl}, x)^{-1}} * \phi, \phi\rangle| \\
\leq{}&\left(1 - \frac{\epsilon^2}{2}\right) E\big(\beta_j^{(p)_1}\big) E\big(\tilde{\beta}_j^{(p)_1}\big) + \sum_{(\alpha_j^{(p)_1}, \tilde{\alpha}_j^{(p)_1}) \neq (\beta_j^{(p)_1}, \tilde{\beta}_j^{(p)_1})} E\big(\alpha_j^{(p)_1}\big) E\big(\tilde{\alpha}_j^{(p)_1}\big) \\
&{}\cdot \|\delta_{\mathtt{c}_\mathfrak{q}^l(\alpha_{jl + 1}, \alpha^{jl}, x) \mathtt{c}_\mathfrak{q}^l(\tilde{\alpha}_{jl + 1}, \tilde{\alpha}^{jl}, x)^{-1}} * \phi\|_2 \|\phi\|_2 \\
={}&\left(1 - \frac{\epsilon^2}{2}\right) E\big(\beta_j^{(p)_1}\big) E\big(\tilde{\beta}_j^{(p)_1}\big) + \sum_{(\alpha_j^{(p)_1}, \tilde{\alpha}_j^{(p)_1}) \neq (\beta_j^{(p)_1}, \tilde{\beta}_j^{(p)_1})} E\big(\alpha_j^{(p)_1}\big) E\big(\tilde{\alpha}_j^{(p)_1}\big) \\
={}&\sum_{\alpha_j^{(p)_1}, \tilde{\alpha}_j^{(p)_1}} E\big(\alpha_j^{(p)_1}\big) E\big(\tilde{\alpha}_j^{(p)_1}\big) - \frac{\epsilon^2}{2} E\big(\beta_j^{(p)_1}\big) E\big(\tilde{\beta}_j^{(p)_1}\big) \\
={}&\|\Pi\|_1 - \frac{\epsilon^2}{2} E\big(\beta_j^{(p)_1}\big) E\big(\tilde{\beta}_j^{(p)_1}\big).
\end{align*}
Finally, \cref{lem:NearlyFlat} gives
\begin{align*}
\|\Pi\|_1 &= \sum_{\alpha_j^{(p)_1}, \tilde{\alpha}_j^{(p)_1}} E\big(\alpha_j^{(p)_1}\big) E\big(\tilde{\alpha}_j^{(p)_1}\big) \leq C_0^2 N^{2p}E\big(\beta_j^{(p)_1}\big) E\big(\tilde{\beta}_j^{(p)_1}\big)
\end{align*}
and thus $\langle \tilde{\Pi}(\phi), \phi\rangle \leq \left(1 - \frac{\epsilon^2}{2C_0^2N^{2p}}\right)\|\Pi\|_1 = C^2\|\Pi\|_1$.
\end{proof}

\begin{lemma}
\label{lem:ExpanderMachineryBound}
There exists $l_0 \in \mathbb N$ such that if $l > l_0$, then there exists $C \in (0, 1)$ such that for all square-free ideals $\mathfrak{q} \subset \mathcal{O}_{\mathbb K}$ coprime to $\mathfrak{q}_0$, integers $s > r$, admissible sequences $(\alpha_s, \alpha_{s - 1}, \dotsc, \alpha_{r + 1})$, and $\phi \in L_0^2(\tilde{G}_\mathfrak{q})$ with $\|\phi\|_2 = 1$, we have
\begin{align*}
\left\|\nu_{(\alpha_s, \alpha_{s - 1}, \dotsc, \alpha_{r + 1})}^\mathfrak{q} * \phi\right\|_2 \leq C^r \left\|\nu_{(\alpha_s, \alpha_{s - 1}, \dotsc, \alpha_{r + 1})}^\mathfrak{q}\right\|_1
\end{align*}
where $C$ is independent of $|a| < a_0'$, $x \in \Sigma^+$, and $r \in \mathbb N$, but dependent on the factorization $r = r'l$.
\end{lemma}

\begin{proof}
Fix $C_1 > 0$ to be the $C$ from \cref{lem:EstimateNu} and $C_2 \in (0, 1)$ to be the $C$ from \cref{lem:EtaOperatorBound}. Fix $C_3 = -\log(C_2) > 0$. There exists an integer $l_0 \geq p$ such that $2C_1 \theta^l - C_3 < 0$ for all $l > l_0$. Suppose $l > l_0$. Fix $C = e^{\frac{1}{l}(2C_1 \theta^l - C_3)} \in (0, 1)$. Let $\mathfrak{q} \subset \mathcal{O}_{\mathbb K}$ be a square-free ideal coprime to $\mathfrak{q}_0$, $s > r$ be integers, $(\alpha_s, \alpha_{s - 1}, \dotsc, \alpha_{r + 1})$ be an admissible sequence, and $\phi \in L_0^2(\tilde{G}_\mathfrak{q})$ with $\|\phi\|_2 = 1$. Using \cref{lem:EstimateNu} and then \cref{lem:EtaOperatorBound} repeatedly $r'$ times, we have
\begin{align*}
&\|\nu_0^\mathfrak{q} * \phi\|_2 \leq e^{r'C_1 \theta^l}\|\nu_1^\mathfrak{q} * \phi\|_2 \\
\leq{}&e^{r'C_1 \theta^l}\sum_{\alpha_1^{(l - p)_2}, \alpha_2^{(l - p)_2}, \dotsc, \alpha_{r'}^{(l - p)_2}} \left\|\mathop{\bigast}\limits_{j = 0}^{r'} \eta^\mathfrak{q}\big(\alpha_{j + 1}^{(l - p)_2}, \alpha_j^{(l - p)_2}\big) * \phi \right\|_2 \\
\leq{}&e^{r'C_1 \theta^l} C_2^{r'}\sum_{\alpha_1^{(l - p)_2}, \alpha_2^{(l - p)_2}, \dotsc, \alpha_{r'}^{(l - p)_2}} \prod_{j = 1}^{r'} \left\|\eta^\mathfrak{q}\big(\alpha_{j + 1}^{(l - p)_2}, \alpha_j^{(l - p)_2}\big)\right\|_1 \\
={}&e^{r'(C_1 \theta^l - C_3)} \sum_{\alpha_1^{(l - p)_2}, \alpha_2^{(l - p)_2}, \dotsc, \alpha_{r'}^{(l - p)_2}} \prod_{j = 1}^{r'} \sum_{\alpha_j^{(p)_1}} E\big(\alpha_{j + 1}^{(l - p)_2}, \alpha_j^l\big).
\end{align*}
Note that in the above calculations $\eta^\mathfrak{q}\big(\alpha_1^{(l - p)_2}, \alpha_0^{(l - p)_2}\big) = \delta_{\mathtt{c}_\mathfrak{q}(\alpha_1, x)}$ which preserves the norm when taking convolutions. As mentioned earlier, for fixed sequences $\alpha_1^{(l - p)_2}, \alpha_2^{(l - p)_2}, \dotsc, \alpha_{r'}^{(l - p)_2}$, the term $E\big(\alpha_{j + 1}^{(l - p)_2}, \alpha_j^l\big)$ depends only on the choice of $\alpha_j^{(p)_1}$ and not on $\alpha_k^{(p)_1}$ for all distinct $1 \leq j, k \leq r'$. Hence we can commute the inner sum and product to get
\begin{align*}
&\|\nu_0^\mathfrak{q} * \phi\|_2 \\
\leq{}&e^{\frac{r}{l}(C_1 \theta^l - C_3)} \sum_{\alpha_1^{(l - p)_2}, \alpha_2^{(l - p)_2}, \dotsc, \alpha_{r'}^{(l - p)_2}} \left(\sum_{\alpha_1^{(p)_1}, \alpha_2^{(p)_1}, \dotsc, \alpha_{r'}^{(p)_1}} \prod_{j = 1}^{r'} E\big(\alpha_{j + 1}^{(l - p)_2}, \alpha_j^l\big)\right) \\
={}&e^{\frac{r}{l}(C_1 \theta^l - C_3)} \sum_{\alpha^r} \prod_{j = 1}^{r'} E\big(\alpha_{j + 1}^{(l - p)_2}, \alpha_j^l\big).
\end{align*}
Recognizing this sum to be $\|\nu_1^\mathfrak{q}\|_1$, we use \cref{lem:EstimateNu} once again to get
\begin{align*}
\|\nu_0^\mathfrak{q} * \phi\|_2 \leq e^{\frac{r}{l}(C_1 \theta^l - C_3)} \|\nu_1^\mathfrak{q}\|_1 \leq e^{\frac{r}{l}(2C_1 \theta^l - C_3)} \|\nu_0^\mathfrak{q}\|_1 = C^r \|\nu_0^\mathfrak{q}\|_1.
\end{align*}
Since $e^{f_{s - r}^{(a)}(\alpha_s, \alpha_{s - 1}, \dotsc, \alpha_{r + 1}, \omega(\alpha_{r + 1}))} > 0$, the lemma follows.
\end{proof}

Let $\mathfrak{q} \subset \mathcal{O}_{\mathbb K}$ be a square-free ideal coprime to $\mathfrak{q}_0$ and $\mu$ be a complex measure on $\tilde{G}_\mathfrak{q}$. We note that $E_\mathfrak{q}^\mathfrak{q}$ is a $\tilde{\mu}$-invariant submodule of the left $\mathbb C[\tilde{G}_\mathfrak{q}]$-module $L^2(\tilde{G}_\mathfrak{q})$. Let $\dot{\mu}$ denote the measure on $\tilde{\mathbf{G}}(\mathcal{O}_{\mathbb K}/\mathfrak{q})$ defined by $\dot{\mu}(g) = \mu(\pi_\mathfrak{q}(\ker(\tilde{\pi}))g)$ for all $g \in \tilde{\mathbf{G}}(\mathcal{O}_{\mathbb K}/\mathfrak{q})$ and $\tilde{\dot{\mu}}: L^2(\tilde{\mathbf{G}}(\mathcal{O}_{\mathbb K}/\mathfrak{q})) \to L^2(\tilde{\mathbf{G}}(\mathcal{O}_{\mathbb K}/\mathfrak{q}))$ denote the operator acting by convolution as before. Similar to above, $\dot{E}_\mathfrak{q}^\mathfrak{q}$ is a $\tilde{\dot{\mu}}$-invariant submodule of the left $\mathbb C[\tilde{\mathbf{G}}(\mathcal{O}_{\mathbb K}/\mathfrak{q})]$-module $L^2(\tilde{\mathbf{G}}(\mathcal{O}_{\mathbb K}/\mathfrak{q}))$.

\begin{lemma}
\label{lem:ConvolutionBoundOnE_q^q}
There exists $C > 0$ such that for all square-free ideals $\mathfrak{q} \subset \mathcal{O}_{\mathbb K}$ coprime to $\mathfrak{q}_0$ and complex measures $\mu$ on $\tilde{G}_\mathfrak{q}$, we have
\begin{align*}
\|\tilde{\mu}|_{E_\mathfrak{q}^\mathfrak{q}}\|_{\mathrm{op}} \leq C N_{\mathbb K}(\mathfrak{q})^{-\frac{1}{3}} (\#\tilde{G}_\mathfrak{q})^{\frac{1}{2}} \|\mu\|_2.
\end{align*}
\end{lemma}

\begin{proof}
Let $\mathfrak{q} \subset \mathcal{O}_{\mathbb K}$ be a square-free ideal coprime to $\mathfrak{q}_0$ and $\mu$ be a complex measure on $\tilde{G}_\mathfrak{q}$. It will be fruitful to first work with $\tilde{\dot{\mu}}$. One way to calculate the operator norm is using the formula
\begin{align*}
\|\tilde{\dot{\mu}}|_{\dot{E}_\mathfrak{q}^\mathfrak{q}}\|_{\mathrm{op}} = \max_{\lambda \in \Lambda(\tilde{\dot{\mu}}^*\tilde{\dot{\mu}}|_{\dot{E}_\mathfrak{q}^\mathfrak{q}})} \sqrt{\lambda}
\end{align*}
where $\Lambda\big(\tilde{\dot{\mu}}^*\tilde{\dot{\mu}}|_{\dot{E}_\mathfrak{q}^\mathfrak{q}}\big)$ is the set of eigenvalues of the self-adjoint positive semidefinite operator $\tilde{\dot{\mu}}^*\tilde{\dot{\mu}}|_{\dot{E}_\mathfrak{q}^\mathfrak{q}}$ which is diagonalizable with nonnegative eigenvalues. Since $\tilde{\dot{\mu}}^*\tilde{\dot{\mu}}|_{\dot{E}_\mathfrak{q}^\mathfrak{q}}: \dot{E}_\mathfrak{q}^\mathfrak{q} \to \dot{E}_\mathfrak{q}^\mathfrak{q}$ is a left $\mathbb C[\tilde{\mathbf{G}}(\mathcal{O}_{\mathbb K}/\mathfrak{q})]$-module homomorphism, its eigenspaces are submodules of $\dot{E}_\mathfrak{q}^\mathfrak{q}$ and hence must contain at least one irreducible submodule $V$ which corresponds to an irreducible representation $\rho: \tilde{\mathbf{G}}(\mathcal{O}_{\mathbb K}/\mathfrak{q}) \to \GL(V)$. Now suppose we have the prime ideal factorization $\mathfrak{q} = \prod_{j = 1}^k \mathfrak{p}_j$ for some $k \in \mathbb N$ and for some prime ideals $\mathfrak{p}_1, \mathfrak{p}_2, \dotsc, \mathfrak{p}_k \subset \mathcal{O}_{\mathbb K}$. Using the Chinese remainder theorem, we have $\tilde{\mathbf{G}}(\mathcal{O}_{\mathbb K}/\mathfrak{q}) \cong \prod_{j = 1}^k \tilde{\mathbf{G}}(\mathcal{O}_{\mathbb K}/\mathfrak{p}_j) \cong \prod_{j = 1}^k \tilde{\mathbf{G}}(\mathbb F_{N_{\mathbb K}(\mathfrak{p}_j)})$. Thus we have $\rho = \bigotimes_{j = 1}^k \rho_j$ where $\rho_j: \tilde{\mathbf{G}}(\mathbb F_{N_{\mathbb K}(\mathfrak{p}_j)}) \to \GL(V_j)$ for some complex vector space $V_j$ is an irreducible representation of Chevalley groups for all $1 \leq j \leq k$. The significance of using $\dot{E}_\mathfrak{q}^\mathfrak{q}$ is precisely that $V \subset \dot{E}_\mathfrak{q}^\mathfrak{q}$ forces $\rho_j$ to be nontrivial for all $1 \leq j \leq k$. Now we directly use \cite[Proposition 4.2]{KS13} to get $\dim(\rho_j) \geq C_1 N_{\mathbb K}(\mathfrak{p}_j)^{\frac{2}{3}}$ for all $1 \leq j \leq k$, for some constant $C_1 > 0$. So, $\dim(\rho) \geq C_1 N_{\mathbb K}(\mathfrak{q})^{\frac{2}{3}}$. Thus, for all $\lambda \in \Lambda\big(\tilde{\dot{\mu}}^*\tilde{\dot{\mu}}|_{\dot{E}_\mathfrak{q}^\mathfrak{q}}\big)$, we have
\begin{align*}
C_1 N_{\mathbb K}(\mathfrak{q})^{\frac{2}{3}} \lambda &\leq \tr(\tilde{\dot{\mu}}^*\tilde{\dot{\mu}}) = \sum_{g \in \tilde{\mathbf{G}}(\mathcal{O}_{\mathbb K}/\mathfrak{q})} \langle \tilde{\dot{\mu}}^*\tilde{\dot{\mu}}(\delta_g), \delta_g \rangle = \sum_{g \in \tilde{\mathbf{G}}(\mathcal{O}_{\mathbb K}/\mathfrak{q})} \|\dot{\mu} * \delta_g\|_2^2 \\
&= \#\tilde{\mathbf{G}}(\mathcal{O}_{\mathbb K}/\mathfrak{q}) \cdot \|\dot{\mu}\|_2^2.
\end{align*}
Hence, we have
\begin{align*}
\|\tilde{\dot{\mu}}|_{\dot{E}_\mathfrak{q}^\mathfrak{q}}\|_{\mathrm{op}}^2 = \max_{\lambda \in \Lambda(\tilde{\dot{\mu}}^*\tilde{\dot{\mu}}|_{\dot{E}_\mathfrak{q}^\mathfrak{q}})} \lambda \leq C_1^{-1} N_{\mathbb K}(\mathfrak{q})^{-\frac{2}{3}} \#\tilde{\mathbf{G}}(\mathcal{O}_{\mathbb K}/\mathfrak{q}) \cdot \|\dot{\mu}\|_2^2.
\end{align*}
Now we convert this to a bound for $\|\tilde{\mu}|_{E_\mathfrak{q}^\mathfrak{q}}\|_{\mathrm{op}}$. Let $\phi \in E_\mathfrak{q}^\mathfrak{q}$. Then $\dot{\phi} \in \dot{E}_\mathfrak{q}^\mathfrak{q}$ and also $\tilde{\dot{\mu}}(\dot{\phi}) = \#\pi_\mathfrak{q}(\ker(\tilde{\pi})) \cdot \dot{\psi} \leq \#\ker(\tilde{\pi}) \cdot \dot{\psi}$ where $\psi = \tilde{\mu}(\phi)$. So, $\|\tilde{\mu}(\phi)\|_2 = \|\psi\|_2 \leq \|\dot{\psi}\|_2 \leq \|\tilde{\dot{\mu}}(\dot{\phi})\|_2$. Now the above bound gives
\begin{align*}
\|\tilde{\mu}(\phi)\|_2 \leq \|\tilde{\dot{\mu}}(\dot{\phi})\|_2 &\leq C_1^{-\frac{1}{2}} N_{\mathbb K}(\mathfrak{q})^{-\frac{1}{3}} (\#\tilde{\mathbf{G}}(\mathcal{O}_{\mathbb K}/\mathfrak{q}))^{\frac{1}{2}} \|\dot{\mu}\|_2  \|\dot{\phi}\|_2 \\
&\leq C_1^{-\frac{1}{2}} N_{\mathbb K}(\mathfrak{q})^{-\frac{1}{3}} (\#\ker(\tilde{\pi}))^{\frac{3}{2}} (\#\tilde{G}_\mathfrak{q})^{\frac{1}{2}} \|\mu\|_2 \|\phi\|_2 \\
&\leq C N_{\mathbb K}(\mathfrak{q})^{-\frac{1}{3}} (\#\tilde{G}_\mathfrak{q})^{\frac{1}{2}} \|\mu\|_2 \|\phi\|_2
\end{align*}
where $C = C_1^{-\frac{1}{2}} (\#\ker(\tilde{\pi}))^{\frac{3}{2}}$. Hence, the lemma follows.
\end{proof}

\begin{remark}
The hypothesis that the ideal $\mathfrak{q} \subset \mathcal{O}_{\mathbb K}$ be square-free is not required in \cref{lem:ConvolutionBoundOnE_q^q}. In general, we do not get Chevalley groups but \cite[Proposition 4.2]{KS13} still holds.
\end{remark}

Now we prove \cref{lem:L2FlatteningLemma} by starting with \cref{lem:ConvolutionBoundOnE_q^q} obtained from the lower bounds of nontrivial irreducible representations of Chevalley groups, and then using \cref{lem:ExpanderMachineryBound} obtained from the exapansion machinery to continue to bound the right hand side by the $L^1$ norm and also remove the growth contributed by $\#\tilde{G}_\mathfrak{q}$ essentially by fiat.

\begin{proof}[Proof of \cref{lem:L2FlatteningLemma}]
Fix $C_1 > 0$ to be the $C$ from \cref{lem:ConvolutionBoundOnE_q^q} and $C_2 > 0$ to be the $C$ from \cref{lem:muHatLessThanCnu}. Suppose $l > l_0$ where $l_0 \in \mathbb N$ is the constant provided by \cref{lem:ExpanderMachineryBound} and fix $C_3 \in (0, 1)$ to be the corresponding $C$ from the same lemma. Fix $C = 2C_1C_2 > 0$ and $C_0 = -\frac{c}{2\log(C_3)} > 0$ where $c > 0$ depends on $n$ and is such that $\#\tilde{G}_\mathfrak{q} \leq N_{\mathbb K}(\mathfrak{q})^{c}$ for all nontrivial ideals $\mathfrak{q} \subset \mathcal{O}_{\mathbb K}$. Let $\mathfrak{q} \subset \mathcal{O}_{\mathbb K}$ be a square-free ideal coprime to $\mathfrak{q}_0$. Suppose $r \geq C_0 \log(N_{\mathbb K}(\mathfrak{q}))$. Let $s > r$ be an integer, $(\alpha_s, \alpha_{s - 1}, \dotsc, \alpha_{r + 1})$ be an admissible sequence, and $\phi \in E_\mathfrak{q}^\mathfrak{q}$ with $\|\phi\|_2 = 1$. Let $\mu$ denote either $\mu_{(\alpha_s, \alpha_{s - 1}, \dotsc, \alpha_{r + 1})}^\mathfrak{q}$ or $\hat{\mu}_{(\alpha_s, \alpha_{s - 1}, \dotsc, \alpha_{r + 1})}^\mathfrak{q}$ and $\nu$ denote $\nu_{(\alpha_s, \alpha_{s - 1}, \dotsc, \alpha_{r + 1})}^\mathfrak{q}$. Applying \cref{lem:ConvolutionBoundOnE_q^q} to $\mu$ and then using \cref{lem:muHatLessThanCnu} gives
\begin{align*}
\|\mu * \phi\|_2 \leq C_1C_2 N_{\mathbb K}(\mathfrak{q})^{-\frac{1}{3}}(\#\tilde{G}_\mathfrak{q})^{\frac{1}{2}} \|\nu\|_2.
\end{align*}
By choice of $r$ and $\varphi = \delta_e - \frac{1}{\#\tilde{G}_\mathfrak{q}}\chi_{\tilde{G}_\mathfrak{q}} \in L_0^2(\tilde{G}_\mathfrak{q})$, which satisfies $\|\varphi\|_2 \leq 1$, we can use \cref{lem:ExpanderMachineryBound} to get
\begin{align*}
\|\nu\|_2 \leq \left\|\nu * \frac{1}{\#\tilde{G}_\mathfrak{q}}\chi_{\tilde{G}_\mathfrak{q}}\right\|_2 +\|\nu * \varphi\|_2 \leq \frac{\|\nu\|_1}{(\#\tilde{G}_\mathfrak{q})^{\frac{1}{2}}} + C_3^r\|\nu\|_1 \leq 2\frac{\|\nu\|_1}{(\#\tilde{G}_\mathfrak{q})^{\frac{1}{2}}}.
\end{align*}
Combining the two inequalities, we have $\|\mu * \phi\|_2 \leq C N_{\mathbb K}(\mathfrak{q})^{-\frac{1}{3}} \|\nu\|_1$.
\end{proof}

\section{\texorpdfstring{$L^\infty$}{L-infinity} and Lipschitz bounds and proof of \texorpdfstring{\cref{thm:ReducedTheoremSmall|b|}}{\autoref{thm:ReducedTheoremSmall|b|}}}
\label{sec:SupremumAndLipschitzBounds}
In this section we use \cref{lem:L2FlatteningLemma} to prove the $L^\infty$ and Lipschitz bounds in \cref{lem:ReducedTheoremEstimate1,lem:ReducedTheoremEstimate2}. We then use them to prove \cref{thm:ReducedTheoremSmall|b|} by induction.

We start with fixing some notations and writing some easy bounds. Let $\mathfrak{q} \subset \mathcal{O}_{\mathbb K}$ be a nontrivial proper ideal. Fix integers
\begin{align*}
r_\mathfrak{q} \in [C_0 \log(N_{\mathbb K}(\mathfrak{q})), C_0 \log(N_{\mathbb K}(\mathfrak{q})) + l)
\end{align*}
with $r_\mathfrak{q} \in l\mathbb Z$ and
\begin{align*}
s_\mathfrak{q} \in \left(r_\mathfrak{q} - \frac{\log(N_{\mathbb K}(\mathfrak{q})) + \log(4C_1C_f)}{\log(\theta)}, C_s\log(N_{\mathbb K}(\mathfrak{q}))\right)
\end{align*}
where we fix $C_0$ and $l$ to be constants from \cref{lem:L2FlatteningLemma}, $C_1$ to be the constant from \cref{lem:Small|b|Bound} and $C_s = C_0 - \frac{1}{\log(\theta)} + \frac{l}{\log(2)} - \frac{\log(4C_1C_f)}{\log(\theta)\log(2)} + \frac{1}{\log(2)}$ so that there is enough room for the integer $s_\mathfrak{q}$ to exist. These definitions of constants ensure that $C_0 \log(N_{\mathbb K}(\mathfrak{q})) \leq r_\mathfrak{q} < s_\mathfrak{q}$ and $4C_1C_f \theta^{s_\mathfrak{q} - r_\mathfrak{q}} \leq N_{\mathbb K}(\mathfrak{q})^{-1}$. For all $\xi \in \mathbb C$ with $|a| < a_0'$, for all square-free ideals $\mathfrak{q} \subset \mathcal{O}_{\mathbb K}$ coprime to $\mathfrak{q}_0$, $x \in \Sigma^+$, integers $C_0 \log(N_{\mathbb K}(\mathfrak{q})) \leq r < s$ with $r \in l\mathbb Z$, and admissible sequences $(\alpha_s, \alpha_{s - 1}, \dotsc, \alpha_{r + 1})$, we have
\begin{align*}
\left\|\nu_{(\alpha_s, \alpha_{s - 1}, \dotsc, \alpha_{r + 1})}^{a, \mathfrak{q}, x}\right\|_1 &= e^{f_{s - r}^{(a)}(\alpha_s, \alpha_{s - 1}, \dotsc, \alpha_{r + 1}, \omega(\alpha_{r + 1}))} \sum_{\alpha^r} e^{f_r^{(a)}(\alpha^r, x)} \\
&\leq C_f e^{f_{s - r}^{(a)}(\alpha_s, \alpha_{s - 1}, \dotsc, \alpha_{r + 1}, \omega(\alpha_{r + 1}))}
\end{align*}
by \cref{lem:SumExpf^aBound} and hence \cref{lem:L2FlatteningLemma} implies that for all $\phi \in E_\mathfrak{q}^\mathfrak{q}$ we have
\begin{align*}
\|\mu * \phi\|_2 \leq CC_f N_{\mathbb K}(\mathfrak{q})^{-\frac{1}{3}} e^{f_{s - r}^{(a)}(\alpha_s, \alpha_{s - 1}, \dotsc, \alpha_{r + 1}, \omega(\alpha_{r + 1}))} \|\phi\|_2
\end{align*}
where $\mu$ denotes either $\mu_{(\alpha_s, \alpha_{s - 1}, \dotsc, \alpha_{r + 1})}^{\xi, \mathfrak{q}, x}$ or $\hat{\mu}_{(\alpha_s, \alpha_{s - 1}, \dotsc, \alpha_{r + 1})}^{a, \mathfrak{q}, x}$ and $C$ is the constant from the same lemma. We will use this in \cref{lem:ReducedTheoremEstimate1,lem:ReducedTheoremEstimate2}. We now start with the $L^\infty$ bound.

\begin{lemma}
\label{lem:ReducedTheoremEstimate1}
There exist $\kappa_1 \in (0, 1)$ and $q_{1,1} \in \mathbb N$ such that for all $\xi \in \mathbb C$ with $|a| < a_0'$, square-free ideals $\mathfrak{q} \subset \mathcal{O}_{\mathbb K}$ coprime to $\mathfrak{q}_0$ with $N_{\mathbb K}(\mathfrak{q}) > q_{1,1}$, and $H \in \mathcal{W}_\mathfrak{q}^\mathfrak{q}(U)$, we have
\begin{align*}
\big\|\mathcal{M}_{\xi, \mathfrak{q}}^{s_\mathfrak{q}}(H)\big\|_\infty \leq \frac{1}{2}N_{\mathbb K}(\mathfrak{q})^{-\kappa_1}(\|H\|_\infty + \Lip_{d_\theta}(H)).
\end{align*}
\end{lemma}

\begin{proof}
There exist $q_{1,1} \in \mathbb N$ and $\epsilon \in (0, 1)$ such that $\frac{1}{3} - \frac{\log(2CC_f^2)}{\log(q)} > \epsilon$ for all $q > q_{1,1}$. Fix any $\kappa_1 \in (0, \epsilon)$. Let $\xi \in \mathbb C$ with $|a| < a_0'$, $\mathfrak{q} \subset \mathcal{O}_{\mathbb K}$ be a square-free ideal coprime to $\mathfrak{q}_0$ with $N_{\mathbb K}(\mathfrak{q}) > q_{1,1}$, $H \in \mathcal{W}_\mathfrak{q}^\mathfrak{q}(U)$, and $x \in \Sigma^+$. Denote $r_\mathfrak{q}$ by $r$ and $s_\mathfrak{q}$ by $s$. Now using the approximation from \cref{lem:TransferOperatorConvolutionApproximation} and then \cref{lem:L2FlatteningLemma}, we have
\begin{align*}
&\left\|\mathcal{M}_{\xi, \mathfrak{q}}^s(H)(x)\right\|_2 \\
\leq{}&\left\|\sum_{\alpha_{r + 1}, \alpha_{r + 2}, \dotsc, \alpha_s} \mu_{(\alpha_s, \alpha_{s - 1}, \dotsc, \alpha_{r + 1})}^{\xi, \mathfrak{q}, x} * \phi_{(\alpha_s, \alpha_{s - 1}, \dotsc, \alpha_{r + 1})}^{\mathfrak{q}, H}\right\|_2 + C_f \Lip_{d_\theta}(H)\theta^{s - r} \\
\leq{}&\sum_{\alpha_{r + 1}, \alpha_{r + 2}, \dotsc, \alpha_s} \left\|\mu_{(\alpha_s, \alpha_{s - 1}, \dotsc, \alpha_{r + 1})}^{\xi, \mathfrak{q}, x} * \phi_{(\alpha_s, \alpha_{s - 1}, \dotsc, \alpha_{r + 1})}^{\mathfrak{q}, H}\right\|_2 + C_f \Lip_{d_\theta}(H)\theta^{s - r} \\
\leq{}&\sum_{\alpha_{r + 1}, \alpha_{r + 2}, \dotsc, \alpha_s} CC_f N_{\mathbb K}(\mathfrak{q})^{-\frac{1}{3}} e^{f_{s - r}^{(a)}(\alpha_s, \alpha_{s - 1}, \dotsc, \alpha_{r + 1}, \omega(\alpha_{r + 1}))} \left\|\phi_{(\alpha_s, \alpha_{s - 1}, \dotsc, \alpha_{r + 1})}^{\mathfrak{q}, H}\right\|_2 \\
{}&+ C_f \Lip_{d_\theta}(H)\theta^{s - r} \\
\leq{}&CC_f N_{\mathbb K}(\mathfrak{q})^{-\frac{1}{3}} \|H\|_\infty \sum_{\alpha_{r + 1}, \alpha_{r + 2}, \dotsc, \alpha_s} e^{f_{s - r}^{(a)}(\alpha_s, \alpha_{s - 1}, \dotsc, \alpha_{r + 1}, \omega(\alpha_{r + 1}))} \\
{}&+ C_f \Lip_{d_\theta}(H)\theta^{s - r} \\
\leq{}&CC_f^2 N_{\mathbb K}(\mathfrak{q})^{-\frac{1}{3}} \|H\|_\infty + C_f \Lip_{d_\theta}(H)\theta^{s - r} \\
\leq{}&\frac{1}{2}N_{\mathbb K}(\mathfrak{q})^{-\kappa_1}(\|H\|_\infty + \Lip_{d_\theta}(H)).
\end{align*}
since $CC_f^2 N_{\mathbb K}(\mathfrak{q})^{-\frac{1}{3}} \leq \frac{1}{2}N_{\mathbb K}(\mathfrak{q})^{-\kappa_1}$ and
\begin{align*}
C_f \theta^{s - r} \leq C_1C_f \theta^{s - r} \leq \frac{1}{4} N_{\mathbb K}(\mathfrak{q})^{-1} \leq \frac{1}{2} N_{\mathbb K}(\mathfrak{q})^{-\kappa_1}
\end{align*}
by definitions of the constants.
\end{proof}

Recalling that we already fixed $b_0 = 1$, we now record an estimate.

\begin{lemma}
\label{lem:Small|b|Bound}
There exists $C > 1$ such that for all $\xi \in \mathbb C$ with $|a| < a_0'$ and $|b| \leq b_0$, elements $x, y \in \Sigma^+$, $s \in \mathbb N$, and admissible sequences $\alpha^s$, we have
\begin{align*}
\left|1 - e^{(f_s^{(a)} + ib\tau_s)(\alpha^s, y) - (f_s^{(a)} + ib\tau_s)(\alpha^s, x)}\right| \leq C d_\theta(x, y).
\end{align*}
\end{lemma}

\begin{proof}
Fix $C > \max\left(1, (1 + b_0)\frac{T_0 \theta}{1 - \theta}e^{\frac{T_0 \theta}{1 - \theta}}\right)$. Let $\xi \in \mathbb C$ with $|a| < a_0'$ and $|b| \leq b_0$. Let $x, y \in \Sigma^+$, $s \in \mathbb N$, and $\alpha^s$ be an admissible sequence. We calculate that
\begin{align*}
\big|f_s^{(a)}(\alpha^s, y) - f_s^{(a)}(\alpha^s, x)\big| &\leq \sum_{j = 0}^{s - 1} \big|f^{(a)}(\sigma^j(\alpha^s, y)) - f^{(a)}(\sigma^j(\alpha^s, x))\big| \\
&\leq \sum_{j = 0}^{s - 1} \Lip_{d_\theta}(f^{(a)}) \cdot d_\theta(\sigma^j(\alpha^s, y), \sigma^j(\alpha^s, x)) \\
&\leq \Lip_{d_\theta}(f^{(a)}) \sum_{j = 0}^{s - 1} \theta^{s - j} d_\theta(x, y) \leq \frac{T_0 \theta}{1 - \theta}d_\theta(x, y).
\end{align*}
In the same way, we have a similar bound $|\tau_s(\alpha^s, y) - \tau_s(\alpha^s, x)| \leq \frac{T_0 \theta}{1 - \theta}d_\theta(x, y)$. Thus, using $d_\theta(x, y) \leq 1$, we have
\begin{align*}
&\left|1 - e^{(f_s^{(a)} + ib\tau_s)(\alpha^s, y) - (f_s^{(a)} + ib\tau_s)(\alpha^s, x)}\right| \\
\leq{}&e^{\big|f_s^{(a)}(\alpha^s, y) - f_s^{(a)}(\alpha^s, x)\big|} \big|(f_s^{(a)} + ib\tau_s)(\alpha^s, y) - (f_s^{(a)} + ib\tau_s)(\alpha^s, x)\big| \\
\leq{}&e^{\frac{T_0 \theta}{1 - \theta}}\left(\frac{T_0 \theta}{1 - \theta}d_\theta(x, y) + \frac{b_0T_0 \theta}{1 - \theta}d_\theta(x, y)\right) \leq Cd_\theta(x, y).
\end{align*}
\end{proof}

\begin{remark}
This is the reason the approach of Bourgain--Gamburd--Sarnak is restricted to small $|b|$.
\end{remark}

Now we can take care of the Lipschitz bound and although \cref{lem:TransferOperatorConvolutionApproximation} cannot be used directly, we can use similar albeit more intricate estimates.

\begin{lemma}
\label{lem:ReducedTheoremEstimate2}
There exist $\kappa_2 \in (0, 1)$ and $q_{1,2} \in \mathbb N$ such that for all $\xi \in \mathbb C$ with $|a| < a_0'$ and $|b| \leq b_0$, square-free ideals $\mathfrak{q} \subset \mathcal{O}_{\mathbb K}$ coprime to $\mathfrak{q}_0$ with $N_{\mathbb K}(\mathfrak{q}) > q_{1,2}$, and $H \in \mathcal{W}_\mathfrak{q}^\mathfrak{q}(U)$, we have
\begin{align*}
\Lip_{d_\theta}(\mathcal{M}_{\xi, \mathfrak{q}}^{s_\mathfrak{q}}(H)) \leq \frac{1}{2}N_{\mathbb K}(\mathfrak{q})^{-\kappa_2}(\|H\|_\infty + \Lip_{d_\theta}(H)).
\end{align*}
\end{lemma}

\begin{proof}
There exist $q_{1,2} \in \mathbb N$ and $\epsilon \in (0, 1)$ such that $\frac{1}{3} - \frac{\log(4CC_1C_f^2)}{\log(q)} > \epsilon$ for all $q > q_{1,2}$. Fix any $\kappa_2 \in (0, \epsilon)$. Let $\xi \in \mathbb C$ with $|a| < a_0'$ and $|b| \leq b_0$. Let $\mathfrak{q} \subset \mathcal{O}_{\mathbb K}$ be a square-free ideal coprime to $\mathfrak{q}_0$ with $N_{\mathbb K}(\mathfrak{q}) > q_{1,2}$, $H \in \mathcal{W}_\mathfrak{q}^\mathfrak{q}(U)$, and $x, y \in \Sigma^+$. Denote $r_\mathfrak{q}$ by $r$ and $s_\mathfrak{q}$ by $s$. First suppose that $d_\theta(x, y) = 1$. Then from the proof of \cref{lem:ReducedTheoremEstimate1}, we can simply estimate as
\begin{align*}
&\left\|\mathcal{M}_{\xi, \mathfrak{q}}^s(H)(x) - \mathcal{M}_{\xi, \mathfrak{q}}^s(H)(y)\right\|_2 \\
\leq{}&\left\|\mathcal{M}_{\xi, \mathfrak{q}}^s(H)(x)\right\|_2 + \left\|\mathcal{M}_{\xi, \mathfrak{q}}^s(H)(y)\right\|_2 \\
\leq{}&\big(2CC_f^2 N_{\mathbb K}(\mathfrak{q})^{-\frac{1}{3}} \|H\|_\infty + 2C_f \Lip_{d_\theta}(H)\theta^{s - r}\big)d_\theta(x, y).
\end{align*}
Now suppose $d_\theta(x, y) < 1$. Then of course $x_0 = y_0$ and hence all the sums which will appear are over the same set of admissible sequences and moreover $\delta_{\mathtt{c}_\mathfrak{q}^s(\alpha^s, x)} = \delta_{\mathtt{c}_\mathfrak{q}^s(\alpha^s, y)}$. Thus we have
\begin{align*}
&\left\|\mathcal{M}_{\xi, \mathfrak{q}}^s(H)(x) - \mathcal{M}_{\xi, \mathfrak{q}}^s(H)(y)\right\|_2 \\
\leq{}&\left\|\sum_{\alpha^s} e^{(f_s^{(a)} + ib\tau_s)(\alpha^s, x)} \delta_{\mathtt{c}_\mathfrak{q}^s(\alpha^s, x)} * H(\alpha^s, x)\right. \\
{}&\left.- \sum_{\alpha^s} e^{(f_s^{(a)} + ib\tau_s)(\alpha^s, y)} \delta_{\mathtt{c}_\mathfrak{q}^s(\alpha^s, y)} * H(\alpha^s, y)\right\|_2 \\
\leq{}& \left\|\sum_{\alpha^s} e^{(f_s^{(a)} + ib\tau_s)(\alpha^s, y)}\delta_{\mathtt{c}_\mathfrak{q}^s(\alpha^s, x)} * \left(H(\alpha^s, x) - H(\alpha^s, y)\right)\right\|_2 \\
&{}+ \left\|\sum_{\alpha^s} \left(e^{(f_s^{(a)} + ib\tau_s)(\alpha^s, x)} - e^{(f_s^{(a)} + ib\tau_s)(\alpha^s, y)}\right)\right. \\
{}&\left.\cdot\delta_{\mathtt{c}_\mathfrak{q}^s(\alpha^s, x)} * (H(\alpha^s, x) - H(\alpha_s, \alpha_{s - 1}, \dotsc, \alpha_{r + 1}, \omega(\alpha_{r + 1})))\rule{0cm}{0.6cm}\right\|_2 \\
&{}+ \left\|\sum_{\alpha^s} \left(e^{(f_s^{(a)} + ib\tau_s)(\alpha^s, x)} - e^{(f_s^{(a)} + ib\tau_s)(\alpha^s, y)}\right)\right. \\
{}&\left.\cdot\delta_{\mathtt{c}_\mathfrak{q}^s(\alpha^s, x)} * H(\alpha_s, \alpha_{s - 1}, \dotsc, \alpha_{r + 1}, \omega(\alpha_{r + 1}))\rule{0cm}{0.6cm}\right\|_2 \\
={}&K_1 + K_2 + K_3.
\end{align*}
We easily estimate the first term $K_1$ as
\begin{align*}
K_1 &\leq \sum_{\alpha^s} e^{f_s^{(a)}(\alpha^s, x)} \|H(\alpha^s, x) - H(\alpha^s, y)\|_2 \\
&\leq \Lip_{d_\theta}(H) \theta^s d_\theta(x, y) \sum_{\alpha^s} e^{f_s^{(a)}(\alpha^s, x)} \leq C_f \Lip_{d_\theta}(H) \theta^s d_\theta(x, y).
\end{align*}
Next we estimate the second term $K_2$ as
\begin{align*}
K_2 &\leq \sum_{\alpha^s} \left|e^{(f_s^{(a)} + ib\tau_s)(\alpha^s, x)}\right| \cdot \left|1 - e^{(f_s^{(a)} + ib\tau_s)(\alpha^s, y) - (f_s^{(a)} + ib\tau_s)(\alpha^s, x)}\right| \\
&\cdot \|H(\alpha^s, x) - H(\alpha_s, \alpha_{s - 1}, \dotsc, \alpha_{r + 1}, \omega(\alpha_{r + 1}))\|_2 \\
&\leq C_1 \Lip_{d_\theta}(H) \theta^{s - r} d_\theta(x, y) \sum_{\alpha^s} e^{f_s^{(a)}(\alpha^s, x)} \leq C_1C_f \Lip_{d_\theta}(H) \theta^{s - r} d_\theta(x, y).
\end{align*}
Finally, using \cref{lem:L2FlatteningLemma}, we estimate the third and last term $K_3$ as
\begin{align*}
K_3 \leq{}&\left\|\sum_{\alpha^s} \left|e^{(f_s^{(a)} + ib\tau_s)(\alpha^s, x)}\right| \cdot \left|1 - e^{(f_s^{(a)} + ib\tau_s)(\alpha^s, y) - (f_s^{(a)} + ib\tau_s)(\alpha^s, x)}\right|\right. \\
{}&\left.\cdot\delta_{\mathtt{c}_\mathfrak{q}^s(\alpha^s, x)} * |H|(\alpha_s, \alpha_{s - 1}, \dotsc, \alpha_{r + 1}, \omega(\alpha_{r + 1}))\rule{0cm}{0.6cm}\right\|_2 \\
\leq{}&C_1 d_\theta(x, y) \sum_{\alpha_{r + 1}, \alpha_{r + 2}, \dotsc, \alpha_s} \left\|\sum_{\alpha^r} e^{f_s^{(a)}(\alpha^s, x)} \delta_{\mathtt{c}_\mathfrak{q}^{r + 1}(\alpha_{r + 1}, \alpha^r, x)} * \phi_{(\alpha_s, \alpha_{s - 1}, \dotsc, \alpha_{r + 1})}^{\mathfrak{q}, |H|}\right\|_2 \\
\leq{}&C_1 d_\theta(x, y) \sum_{\alpha_{r + 1}, \alpha_{r + 2}, \dotsc, \alpha_s} \left\|\hat{\mu}_{(\alpha_s, \alpha_{s - 1}, \dotsc, \alpha_{r + 1})}^{a, \mathfrak{q}, x} * \phi_{(\alpha_s, \alpha_{s - 1}, \dotsc, \alpha_{r + 1})}^{\mathfrak{q}, |H|}\right\|_2 \\
\leq{}&C_1 d_\theta(x, y) \sum_{\alpha_{r + 1}, \alpha_{r + 2}, \dotsc, \alpha_s} CC_f N_{\mathbb K}(\mathfrak{q})^{-\frac{1}{3}} e^{f_{s - r}^{(a)}(\alpha_s, \alpha_{s - 1}, \dotsc, \alpha_{r + 1}, \omega(\alpha_{r + 1}))} \\
{}&\cdot\left\|\phi_{(\alpha_s, \alpha_{s - 1}, \dotsc, \alpha_{r + 1})}^{\mathfrak{q}, |H|}\right\|_2 \\
\leq{}&CC_1C_f N_{\mathbb K}(\mathfrak{q})^{-\frac{1}{3}} \|\, |H| \,\|_\infty d_\theta(x, y) \sum_{\alpha_{r + 1}, \alpha_{r + 2}, \dotsc, \alpha_s} e^{f_{s - r}^{(a)}(\alpha_s, \alpha_{s - 1}, \dotsc, \alpha_{r + 1}, \omega(\alpha_{r + 1}))} \\
\leq{}&CC_1C_f^2 N_{\mathbb K}(\mathfrak{q})^{-\frac{1}{3}} \|H\|_\infty d_\theta(x, y).
\end{align*}
So combining all three estimates, we have
\begin{align*}
&\left\|\mathcal{M}_{\xi, \mathfrak{q}}^s(H)(x) - \mathcal{M}_{\xi, \mathfrak{q}}^s(H)(y)\right\|_2 \\
\leq{}&\big(CC_1C_f^2 N_{\mathbb K}(\mathfrak{q})^{-\frac{1}{3}} \|H\|_\infty + (C_f \theta^s + C_1C_f \theta^{s - r})\Lip_{d_\theta}(H)\big) d_\theta(x, y).
\end{align*}
Thus, for all $x, y \in \Sigma^+$, we have
\begin{align*}
\Lip_{d_\theta}(\mathcal{M}_{\xi, \mathfrak{q}}^s(H)) &\leq 2CC_1C_f^2 N_{\mathbb K}(\mathfrak{q})^{-\frac{1}{3}} \|H\|_\infty + 2C_1C_f\Lip_{d_\theta}(H)\theta^{s - r} \\
&\leq \frac{1}{2}N_{\mathbb K}(\mathfrak{q})^{-\kappa_2}(\|H\|_\infty + \Lip_{d_\theta}(H))
\end{align*}
since $2CC_1C_f^2 N_{\mathbb K}(\mathfrak{q})^{-\frac{1}{3}} \leq \frac{1}{2}N_{\mathbb K}(\mathfrak{q})^{-\kappa_2}$ and $2C_1C_f\theta^{s - r} \leq \frac{1}{2} N_{\mathbb K}(\mathfrak{q})^{-1} \leq \frac{1}{2} N_{\mathbb K}(\mathfrak{q})^{-\kappa_2}$ by definitions of the constants.
\end{proof}

\begin{proof}[Proof of \cref{thm:ReducedTheoremSmall|b|}]
Fix $\kappa_1, \kappa_2 \in (0, 1)$ and $q_{1,1}, q_{1,2} \in \mathbb N$ to be the constants from \cref{lem:ReducedTheoremEstimate1,lem:ReducedTheoremEstimate2}. Recall the constant $C_s$ and that we already fixed $b_0 = 1$. Fix $a_0 = a_0'$, $\kappa = \min(\kappa_1, \kappa_2) \in (0, 1)$, and $q_1 = \max(q_{1,1}, q_{1,2}) \in \mathbb N$. Let $\xi \in \mathbb C$ with $|a| < a_0$ and $|b| \leq b_0$. Let $\mathfrak{q} \subset \mathcal{O}_{\mathbb K}$ be a square-free ideal coprime to $\mathfrak{q}_0$ with $N_{\mathbb K}(\mathfrak{q}) > q_1$. Denote $s_\mathfrak{q}$ by $s$. Let $j \in \mathbb Z_{\geq 0}$ and $H \in \mathcal{W}_\mathfrak{q}^\mathfrak{q}(U)$. Note that \cref{lem:ReducedTheoremEstimate1,lem:ReducedTheoremEstimate2} together give
\begin{align*}
\left\|\mathcal{M}_{\xi, \mathfrak{q}}^s(H)\right\|_\infty + \Lip_{d_\theta}(\mathcal{M}_{\xi, \mathfrak{q}}^s(H)) \leq N_{\mathbb K}(\mathfrak{q})^{-\kappa}(\|H\|_\infty + \Lip_{d_\theta}(H))
\end{align*}
for all $H \in \mathcal{W}_\mathfrak{q}^\mathfrak{q}(U)$. Now let $j \in \mathbb Z_{\geq 0}$ and $H \in \mathcal{W}_\mathfrak{q}^\mathfrak{q}(U)$. Then by induction we have
\begin{align*}
\big\|\mathcal{M}_{\xi, \mathfrak{q}}^{js}(H)\big\|_2 &\leq \big\|\mathcal{M}_{\xi, \mathfrak{q}}^{js}(H)\big\|_\infty \leq N_{\mathbb K}(\mathfrak{q})^{-j\kappa}(\|H\|_\infty + \Lip_{d_\theta}(H)) \\
&= N_{\mathbb K}(\mathfrak{q})^{-j\kappa}\|H\|_{\Lip(d_\theta)}.
\end{align*}
\end{proof}

\section{Proofs of Theorems \ref{thm:TheoremLarge|b|OrNontrivial_rho} and \ref{thm:TheoremUniformExponentialMixingOfFrameFlow}}
\label{sec:Dolgopyat'sMethod}
In this section, we will prove \cref{thm:TheoremLarge|b|OrNontrivial_rho} using Dolgopyat's method \cite{Dol98}. Then we describe how \cref{thm:TheoremUniformExponentialMixingOfFrameFlow} is derived from \cref{thm:TheoremFrameFlow}.

First, we reduce \cref{thm:TheoremLarge|b|OrNontrivial_rho} to \cref{thm:FrameFlowDolgopyat} which is the main technical theorem in our setting associated to Dolgopyat's method. Similar theorems have appeared in \cite{Dol98,Sto11,OW16,SW20} but the difference here is that we deal with both the holonomy and the uniformity in ideals $\mathfrak{q} \subset \mathcal{O}_{\mathbb K}$ simultaneously.

We define the cone
\begin{align*}
K_B(\tilde{U}) = \{h \in C^1(\tilde{U}, \mathbb R): h > 0, \|(dh)_u\|_{\mathrm{op}} \leq Bh(u) \text{ for all } u \in \tilde{U}\}.
\end{align*}

\begin{remark}
It is useful to note that we can easily derive the equivalent $\log$-Lipschitz characterization given by $K_B(\tilde{U}) = \{h \in C^1(\tilde{U}, \mathbb R): h > 0, |\log \circ h|_{C^1} \leq B\}$.
\end{remark}

\begin{theorem}
\label{thm:FrameFlowDolgopyat}
There exist $m \in \mathbb N$, $\eta \in (0, 1)$, $E > \max\left(1, \frac{1}{b_0}, \frac{1}{\delta_\varrho}\right)$, $a_0 > 0$, $b_0 > 0$, and a set of operators $\{\mathcal{N}_{a, J}^H: C^1(\tilde{U}, \mathbb R) \to C^1(\tilde{U}, \mathbb R): H \in \mathcal{V}_{\mathfrak{q}, \rho}(\tilde{U}), |a| < a_0, J \in \mathcal{J}(b, \rho), \text{ for some } (b, \rho) \in \widehat{M}_0(b_0) \text{ and nontrivial ideal } \mathfrak{q} \subset \mathcal{O}_{\mathbb K}\}$, where $\mathcal{J}(b, \rho)$ is some finite set for all $(b, \rho) \in \widehat{M}_0(b_0)$, such that
\begin{enumerate}
\item\label{itm:FrameFlowDolgopyatProperty1}	$\mathcal{N}_{a, J}^H(K_{E\|\rho_b\|}(\tilde{U})) \subset K_{E\|\rho_b\|}(\tilde{U})$ for all $H \in \mathcal{V}_{\mathfrak{q}, \rho}(\tilde{U})$, $|a| < a_0$, $J \in \mathcal{J}(b, \rho)$, $(b, \rho) \in \widehat{M}_0(b_0)$, and nontrivial ideals $\mathfrak{q} \subset \mathcal{O}_{\mathbb K}$;
\item\label{itm:FrameFlowDolgopyatProperty2}	$\|\mathcal{N}_{a, J}^H(h)\|_2 \leq \eta \|h\|_2$ for all $h \in K_{E\|\rho_b\|}(\tilde{U})$, $H \in \mathcal{V}_{\mathfrak{q}, \rho}(\tilde{U})$, $|a| < a_0$, $J \in \mathcal{J}(b, \rho)$, $(b, \rho) \in \widehat{M}_0(b_0)$, and nontrivial ideals $\mathfrak{q} \subset \mathcal{O}_{\mathbb K}$;
\item\label{itm:FrameFlowDolgopyatProperty3}	for all $\xi \in \mathbb C$ with $|a| < a_0$, if $(b, \rho) \in \widehat{M}_0(b_0)$, then for all nontrivial ideals $\mathfrak{q} \subset \mathcal{O}_{\mathbb K}$, if $H \in \mathcal{V}_{\mathfrak{q}, \rho}(\tilde{U})$ and $h \in K_{E\|\rho_b\|}(\tilde{U})$ satisfy
\begin{enumerate}[label=(1\alph*), ref=\theenumi(1\alph*)]
\item\label{itm:FrameFlowDominatedByh}	$\|H(u)\|_2 \leq h(u)$ for all $u \in \tilde{U}$;
\item\label{itm:FrameFlowLogLipschitzh}	$\|(dH)_u\|_{\mathrm{op}} \leq E\|\rho_b\|h(u)$ for all $u \in \tilde{U}$;
\end{enumerate}
then there exists $J \in \mathcal{J}(b, \rho)$ such that
\begin{enumerate}[label=(2\alph*), ref=\theenumi(2\alph*)]
\item\label{itm:FrameFlowDominatedByDolgopyat}	$\big\|\tilde{\mathcal{M}}_{\xi, \mathfrak{q}, \rho}^m(H)(u)\big\|_2 \leq \mathcal{N}_{a, J}^H(h)(u)$ for all $u \in \tilde{U}$;
\item\label{itm:FrameFlowLogLipschitzDolgopyat}	$\left\|\left(d\tilde{\mathcal{M}}_{\xi, \mathfrak{q}, \rho}^m(H)\right)_u\right\|_{\mathrm{op}} \leq E\|\rho_b\|\mathcal{N}_{a, J}^H(h)(u)$ for all $u \in \tilde{U}$.
\end{enumerate}
\end{enumerate}
\end{theorem}

\begin{proof}[Proof that \cref{thm:FrameFlowDolgopyat} implies \cref{thm:TheoremLarge|b|OrNontrivial_rho}]
Fix $m \in \mathbb N, a_0 > 0, b_0 > 0, E > 0$ to be the ones from \cref{thm:FrameFlowDolgopyat} and $\tilde{\eta} \in (0, 1)$ to be the $\eta$ from \cref{thm:FrameFlowDolgopyat}. Fix
\begin{align*}
B = \sup_{|a| \leq a_0, \{0\} \subsetneq \mathfrak{q} \subset \mathcal{O}_{\mathbb K}, \rho \in \widehat{M}} \big\|\tilde{\mathcal{M}}_{\xi, \mathfrak{q}, \rho}\big\|_{\mathrm{op}} \leq \sup_{|a| \leq a_0} \big\|\tilde{\mathcal{L}}_\xi\big\|_{\mathrm{op}} \leq Ne^{T_0}
\end{align*}
where we use operator norms for operators on $L^2(\tilde{U}, L^2(F_\mathfrak{q}) \otimes V_\rho^\oplus)$ and $L^2(\tilde{U}, \mathbb R)$ respectively. Fix $\eta = \frac{-\log(\tilde{\eta})}{m} > 0$ and $C = B^m \tilde{\eta}^{-1}$. Let $\xi \in \mathbb C$ with $|a| < a_0$. Suppose $(b, \rho) \in \widehat{M}_0(b_0)$. Let $\mathfrak{q} \subset \mathcal{O}_{\mathbb K}$ be a nontrivial ideal, $k \in \mathbb N$, and $H \in \mathcal{V}_{\mathfrak{q}, \rho}(\tilde{U})$. The theorem is trivial if $H = 0$, so suppose that $H \neq 0$. First set $h_0 \in K_{E\|\rho_b\|}(\tilde{U})$ to be the positive constant function defined by $h_0(u) = \|H\|_{1, \|\rho_b\|}$ for all $u \in \tilde{U}$. Then $H$ and $h_0$ satisfy \cref{itm:FrameFlowDominatedByh,itm:FrameFlowLogLipschitzh} in \cref{thm:FrameFlowDolgopyat}. Thus, given $h_j \in K_{E\|\rho_b\|}(\tilde{U})$ for any $j \in \mathbb Z_{\geq 0}$, \cref{thm:FrameFlowDolgopyat} provides a $J_j \in \mathcal{J}(b)$ and we inductively obtain $h_{j + 1} = \mathcal{N}_{a, J_j}(h_j) \in K_{E\|\rho_b\|}(\tilde{U})$. Then $\big\|\tilde{\mathcal{M}}_{\xi, \mathfrak{q}, \rho}^{jm}(H)(u)\big\|_2 \leq h_j(u)$ for all $u \in \tilde{U}$ and hence $\big\|\tilde{\mathcal{M}}_{\xi, \mathfrak{q}, \rho}^{jm}(H)\big\|_2 \leq \|h_j\|_2 \leq \tilde{\eta}^j\|h_0\|_2 = \tilde{\eta}^j\|H\|_{1, \|\rho_b\|}$ for all $j \in \mathbb Z_{\geq 0}$. Then writing $k = jm + l$ for some $j \in \mathbb Z_{\geq 0}$ and $0 \leq l < m$, we have
\begin{align*}
\big\|\tilde{\mathcal{M}}_{\xi, \mathfrak{q}, \rho}^k(H)\big\|_2 \leq B^l\big\|\tilde{\mathcal{M}}_{\xi, \mathfrak{q}, \rho}^{jm}(H)\big\|_2 \leq B^l\tilde{\eta}^j\|H\|_{1, \|\rho_b\|} \leq Ce^{-\eta k} \|H\|_{1, \|\rho_b\|}.
\end{align*}
\end{proof}

The proof of \cref{thm:FrameFlowDolgopyat} is almost a verbatim repeat of the proof of \cite[Theorem 5.4]{SW20} but with the extra congruence aspect. Thus, instead of repeating all of it, we list here the required cosmetic changes and then later describe in more detail the required changes in some proofs where the cocycle plays a significant role.
\begin{enumerate}
\item Almost all corollaries, lemmas, propositions, and theorems need to be changed so that they are stated uniformly with respect to the nontrivial ideals $\mathfrak{q} \subset \mathcal{O}_{\mathbb K}$.
\item The transfer operators with holonomy need to be changed to congruence transfer operators with holonomy. Accordingly, many terms related to the operators appearing in the arguments need to have an action by both the cocycle $\mathtt{c}$ and an irreducible representation $\rho \in \widehat{M}$.
\item Occurrences of $H \in \mathcal{V}_\rho(\tilde{U})$ should be replaced by $H \in \mathcal{V}_{\mathfrak{q}, \rho}(\tilde{U})$.
\item Many of the proofs are almost the same remembering that both the cocycle $\mathtt{c}$ and the irreducible representations $\rho \in \widehat{M}$ act unitarily.
\end{enumerate}

\subsection{Changes required for \texorpdfstring{\cite[Lemma 7.3]{SW20}}{[\cite{SW20}, Lemma 7.3]}}
\label{subsec:ChangesRequiredForLemma7.3}
We use the same notation as in the proof of \cite[Lemma 7.3]{SW20}. Taking the differential and using the product rule, we have instead
\begin{align*}
&\left(d\tilde{\mathcal{M}}_{\xi, \mathfrak{q}, \rho}^k(H)\right)_u(z) \\
={}&\sum_{\alpha: \len(\alpha) = k} e^{f_\alpha^{(a)}(\sigma^{-\alpha}(u))} d(f_\alpha^{(a)} \circ \sigma^{-\alpha})_u(z) \\
&{}\cdot ((\mathtt{c}^\alpha)^{-1} \otimes \rho_b(\Phi^\alpha(\sigma^{-\alpha}(u)))^{-1}) H(\sigma^{-\alpha}(u)) \\
{}&- \sum_{\alpha: \len(\alpha) = k} e^{f_\alpha^{(a)}(\sigma^{-\alpha}(u))} (\Id_{L^2(F_{\mathfrak{q}})} \otimes d(\rho_b \circ \Phi^\alpha \circ \sigma^{-\alpha})_u(z)) H(\sigma^{-\alpha}(u)) \\
{}&+ \sum_{\alpha: \len(\alpha) = k} e^{f_\alpha^{(a)}(\sigma^{-\alpha}(u))} ((\mathtt{c}^\alpha)^{-1} \otimes \rho_b(\Phi^\alpha(\sigma^{-\alpha}(u)))^{-1}) d(H \circ \sigma^{-\alpha})_u(z).
\end{align*}
Now, in each of the summation, we have
\begin{enumerate}
\item $\|((\mathtt{c}^\alpha)^{-1} \otimes \rho_b(\Phi^\alpha(\sigma^{-\alpha}(u)))^{-1}) H(\sigma^{-\alpha}(u))\|_2 = \|H(\sigma^{-\alpha}(u))\|_2$;
\item $\|\Id_{L^2(F_{\mathfrak{q}})} \otimes d(\rho_b \circ \Phi^\alpha \circ \sigma^{-\alpha})_u(z)\|_{\mathrm{op}} = \|d(\rho_b \circ \Phi^\alpha \circ \sigma^{-\alpha})_u(z)\|_{\mathrm{op}}$;
\item $\|((\mathtt{c}^\alpha)^{-1} \otimes \rho_b(\Phi^\alpha(\sigma^{-\alpha}(u)))^{-1}) d(H \circ \sigma^{-\alpha})_u(z)\|_2 = \|d(H \circ \sigma^{-\alpha})_u(z)\|_2$.
\end{enumerate}
Thus, the rest of the proof proceeds as in \cite[Lemma 7.3]{SW20}.

\subsection{Changes required for \texorpdfstring{\cite[Lemma 8.1]{SW20}}{[\cite{SW20}, Lemma 8.1]}}
\label{subsec:ChangesRequiredForLemma8.1}
We simply use \cref{lem:maActionLowerBound} instead of \cite[Lemma 4.4]{SW20}.

\subsection{Changes required for \texorpdfstring{\cite[Lemma 9.10]{SW20}}{[\cite{SW20}, Lemma 9.10]}}
\label{subsec:ChangesRequiredForLemma9.10}
First, according to the changes above, we have the following definitions. For all $\xi \in \mathbb C$ with $|a| < a_0'$, if $(b, \rho) \in \widehat{M}_0(b_0)$, then for all nontrivial ideals $\mathfrak{q} \subset \mathcal{O}_{\mathbb K}$, $H \in \mathcal{V}_{\mathfrak{q}, \rho}(\tilde{U})$, $h \in K_{E\|\rho_b\|}(\tilde{U})$, and $1 \leq j \leq j_{\mathrm{m}}$, we define the functions $\chi_1^j[\xi, \rho, H, h], \chi_2^j[\xi, \rho, H, h]: \tilde{U}_1 \to \mathbb C$ by
\begin{align*}
\chi_1^j[\xi, \rho, H, h](u) ={}&\left\|e^{f_{\alpha_0}^{(a)}(v_0(u))} ((\mathtt{c}^{\alpha_0})^{-1} \otimes \rho_b(\Phi^{\alpha_0}(v_0(u)))^{-1})H(v_0(u))\right. \\
&\left.{}+ e^{f_{\alpha_j}^{(a)}(v_j(u))} ((\mathtt{c}^{\alpha_j})^{-1} \otimes \rho_b(\Phi^{\alpha_j}(v_j(u)))^{-1})H(v_j(u))\right\|_2 \\
&{}\cdot \left((1 - N\mu)e^{f_{\alpha_0}^{(a)}(v_0(u))}h(v_0(u)) + e^{f_{\alpha_j}^{(a)}(v_j(u))}h(v_j(u))\right)^{-1}
\end{align*}
and
\begin{align*}
\chi_2^j[\xi, \rho, H, h](u) ={}&\left\|e^{f_{\alpha_0}^{(a)}(v_0(u))} ((\mathtt{c}^{\alpha_0})^{-1} \otimes \rho_b(\Phi^{\alpha_0}(v_0(u)))^{-1})H(v_0(u))\right. \\
&\left.{}+ e^{f_{\alpha_j}^{(a)}(v_j(u))} ((\mathtt{c}^{\alpha_j})^{-1} \otimes \rho_b(\Phi^{\alpha_j}(v_j(u)))^{-1})H(v_j(u))\right\|_2 \\
&{}\cdot \left(e^{f_{\alpha_0}^{(a)}(v_0(u))}h(v_0(u)) + (1 - N\mu)e^{f_{\alpha_j}^{(a)}(v_j(u))}h(v_j(u))\right)^{-1}
\end{align*}
for all $u \in \tilde{U}_1$.

Now we discuss the changes required in the proof of \cite[Lemma 9.10]{SW20}. We use the same notation as in the proof of \cite[Lemma 9.10]{SW20}. We have instead
\begin{align*}
V_\ell(u) &= e^{f_{\alpha_\ell}^{(a)}(v_\ell(u))}((\mathtt{c}^{\alpha_\ell})^{-1} \otimes \rho_b(\phi_\ell(u))^{-1})H(v_\ell(u)); \\
\hat{V}_\ell(u) &= \frac{V_\ell(u)}{\|V_\ell(u)\|_2} = ((\mathtt{c}^{\alpha_\ell})^{-1} \otimes \rho_b(\phi_\ell(u))^{-1})\omega_\ell(u)
\end{align*}
for all $u \in \tilde{U}_1$ and $\ell \in \{0, j\}$. Since $\omega_0$ and $\omega_j$ are Lipschitz on $\hat{D}_p$ with Lipschitz constant $\delta_1 \|\rho_b\|$ and $d(x_1, x_2) \leq \frac{\epsilon_1}{2\|\rho_b\|}$, we have
\begin{align*}
&\big\|\hat{V}_0(x_2) - \hat{V}_j(x_2)\big\|_2 \\
={}&\|((\mathtt{c}^{\alpha_0})^{-1} \otimes \rho_b(\phi_0(x_2))^{-1})\omega_0(x_2) - ((\mathtt{c}^{\alpha_j})^{-1} \otimes \rho_b(\phi_j(x_2))^{-1})\omega_j(x_2)\|_2 \\
={}&\|((\mathtt{c}^{\alpha_0})^{-1} \otimes \rho_b(\phi_j(x_2)\phi_0(x_2)^{-1}))\omega_0(x_2) - ((\mathtt{c}^{\alpha_j})^{-1} \otimes \Id_{V_\rho^\oplus})\omega_j(x_2)\|_2 \\
\geq{}&\|((\mathtt{c}^{\alpha_0})^{-1} \otimes \rho_b(\phi_j(x_2)\phi_0(x_2)^{-1}))\omega_0(x_1) - ((\mathtt{c}^{\alpha_j})^{-1} \otimes \Id_{V_\rho^\oplus})\omega_j(x_1)\|_2 \\
{}&- \|((\mathtt{c}^{\alpha_0})^{-1} \otimes \rho_b(\phi_j(x_2)\phi_0(x_2)^{-1}))\omega_0(x_2) \\
&{}- ((\mathtt{c}^{\alpha_0})^{-1} \otimes \rho_b(\phi_j(x_2)\phi_0(x_2)^{-1}))\omega_0(x_1)\|_2 \\
{}&- \|((\mathtt{c}^{\alpha_j})^{-1} \otimes \Id_{V_\rho^\oplus})\omega_j(x_2) - ((\mathtt{c}^{\alpha_j})^{-1} \otimes \Id_{V_\rho^\oplus})\omega_j(x_1)\|_2 \\
={}&\|((\mathtt{c}^{\alpha_0})^{-1} \otimes \rho_b(\phi_j(x_2)\phi_0(x_2)^{-1}))\omega_0(x_1) - ((\mathtt{c}^{\alpha_j})^{-1} \otimes \Id_{V_\rho^\oplus})\omega_j(x_1)\|_2 \\
&{}- \|\omega_0(x_2) - \omega_0(x_1)\|_2 - \|\omega_j(x_2) - \omega_j(x_1)\|_2 \\
\geq{}&\|((\mathtt{c}^{\alpha_0})^{-1} \otimes \rho_b(\phi_j(x_2)\phi_0(x_2)^{-1}))\omega_0(x_1) \\
&{}- ((\mathtt{c}^{\alpha_0})^{-1} \otimes \rho_b(\phi_j(x_1)\phi_0(x_1)^{-1}))\omega_0(x_1)\|_2 - \delta_1\epsilon_1 \\
{}&- \|((\mathtt{c}^{\alpha_0})^{-1} \otimes \rho_b(\phi_j(x_1)\phi_0(x_1)^{-1}))\omega_0(x_1) - ((\mathtt{c}^{\alpha_j})^{-1} \otimes \Id_{V_\rho^\oplus})\omega_j(x_1)\|_2 \\
={}&\|(\Id_{L^2(F_{\mathfrak{q}})} \otimes \rho_b(\phi_0(x_1))^{-1})\omega_0(x_1) \\
&{}- (\Id_{L^2(F_{\mathfrak{q}})} \otimes \rho_b(\phi_0(x_1)^{-1}\phi_0(x_2)\phi_j(x_2)^{-1}\phi_j(x_1)\phi_0(x_1)^{-1}))\omega_0(x_1)\|_2 - \delta_1\epsilon_1 \\
{}&- \|((\mathtt{c}^{\alpha_0})^{-1} \otimes \rho_b(\phi_0(x_1))^{-1})\omega_0(x_1) - ((\mathtt{c}^{\alpha_j})^{-1} \otimes \rho_b(\phi_j(x_1))^{-1})\omega_j(x_1)\|_2 \\
\geq{}&\|\rho_{b, \mathfrak{q}}(\phi_0(x_1)^{-1})\omega_0(x_1) - \rho_{b, \mathfrak{q}}(\BP_j(x_2, x_1))\rho_{b, \mathfrak{q}}(\phi_0(x_1)^{-1})\omega_0(x_1)\|_2 \\
&{}- \delta_1\epsilon_1 - \big\|\hat{V}_0(x_1) - \hat{V}_j(x_1)\big\|_2.
\end{align*}
Denote $\omega = \rho_{b, \mathfrak{q}}(\phi_0(x)^{-1})\omega_0(x)$ and $Z = d(\BP_{j, x_1} \circ \Psi)_{\check{x}_1}(z)$ where $z = (\check{x}_1, \check{x}_2 - \check{x}_1) \in \T_{\check{x}_1}(\mathbb R^{n - 1})$. Recall the curve $\varphi^{\mathrm{BP}}_{j, x_1, z}: [0, 1] \to AM$ defined by $\varphi^{\mathrm{BP}}_{j, x_1, z}(t) = \BP_{j, x_1}(\Psi(\check{x}_1 + tz))$ for all $t \in [0, 1]$. Recall that ${\varphi^{\mathrm{BP}}_{j, x_1, z}}'(0) = Z$ and $\varphi^{\mathrm{BP}}_{j, x_1, z}(0) = \BP_{j, x_1}(x_1) = e$ and $\varphi^{\mathrm{BP}}_{j, x_1, z}(1) = \BP_{j, x_1}(x_2) = \BP_j(x_2, x_1)$. Now, applying \cite[Lemma 7.1 and Eq. (14)]{SW20} and the modified version of \cite[Lemma 8.1]{SW20} from \cref{subsec:ChangesRequiredForLemma8.1}, we bound the first term above as
\begin{align*}
&\|\omega - \rho_{b, \mathfrak{q}}(\BP_j(x_2, x_1))(\omega)\|_2 \\
\geq{}&\|\omega - \rho_{b, \mathfrak{q}}(\exp(Z))(\omega)\|_2 - \left\|\rho_{b, \mathfrak{q}}(\exp(Z))(\omega) - \rho_{b, \mathfrak{q}}\left(\varphi^{\mathrm{BP}}_{j, x_1, z}(1)\right)(\omega)\right\|_2 \\
\geq{}&\|\omega - \exp(d\rho_{b, \mathfrak{q}}(Z))(\omega)\|_2 - \|\rho_{b, \mathfrak{q}}\| \cdot d_{AM}\left(\exp(Z), \varphi^{\mathrm{BP}}_{j, x_1, z}(1)\right) \\
\geq{}&\|d\rho_{b, \mathfrak{q}}(Z)(\omega)\|_2 - \|\rho_{b, \mathfrak{q}}\|^2 \|Z\|^2 - \|\rho_{b, \mathfrak{q}}\| \cdot d_{AM}\left(\exp(Z), \varphi^{\mathrm{BP}}_{j, x_1, z}(1)\right) \\
\geq{}&\|d\rho_{b, \mathfrak{q}}(Z)(\omega)\|_2 - \|\rho_b\|^2 (C_{\BP, \Psi} C_{\Psi})^2 d(x_1, x_2)^2 - C_{\exp, \BP} \cdot \|\rho_b\| \cdot d(x_1, x_2)^2 \\
\geq{}&7\delta_1\epsilon_1 - \delta_1\epsilon_1 - \delta_1\epsilon_1 \geq 5\delta_1\epsilon_1.
\end{align*}
Hence, we have
\begin{align*}
\big\|\hat{V}_0(x_1) - \hat{V}_j(x_1)\big\|_2 + \big\|\hat{V}_0(x_2) - \hat{V}_j(x_2)\big\|_2 \geq 4\delta_1\epsilon_1.
\end{align*}
Now, the rest of the proof proceeds as in \cite[Lemma 9.10]{SW20}, remembering that the cocycle acts unitarily.

\subsection{Proof that \texorpdfstring{\cref{thm:TheoremFrameFlow}}{Theorem \ref{thm:TheoremFrameFlow}} implies \texorpdfstring{\cref{thm:TheoremUniformExponentialMixingOfFrameFlow}}{Theorem \ref{thm:TheoremUniformExponentialMixingOfFrameFlow}}}
\cref{thm:TheoremUniformExponentialMixingOfFrameFlow} is proved using \cref{thm:TheoremFrameFlow} via Paley--Wiener theory analogous to proofs in \cite{OW16,SW20} so we only mention the differences here.

In fact, the treatment for the frame flow case in \cite[Section 10]{SW20} can be closely followed. The first minor difference is that since our function spaces have tensor factors of $L^2(F_\mathfrak{q})$ involved, the notations and definitions in \cite[Section 10]{SW20} need to be modified accordingly. The second major difference is that similar to \cite[Section 5]{OW16}, we need to keep track of the factors of powers of $N_{\mathbb K}(\mathfrak{q})$ throughout the argument uniformly in the nontrivial ideals $\mathfrak{q} \subset \mathcal{O}_{\mathbb K}$. To do so, firstly, we use \cref{thm:TheoremFrameFlow} to prove an analogue of \cite[Proposition 5.5]{OW16} and \cite[Lemma 10.3]{SW20}. Here we interject that unlike \cite[Lemma 10.3]{SW20}, we should omit the analogous condition $\int_M \tilde{\psi}(u, g, m, t) \, dm = 0$ where $\tilde{\psi} \in C^1(\tilde{U}^{\mathfrak{q}, \vartheta, \tau}, \mathbb R)$ and rather include the condition $\sum_{g \in \tilde{G}_{\mathfrak{q}}} \tilde{\psi}(u, g, m, t) = 0$ analogous to \cite[Proposition 5.5]{OW16}. Secondly, we use the growth estimate $\#\tilde{G}_\mathfrak{q} \leq N_{\mathbb K}(\mathfrak{q})^{c}$ for all nontrivial ideals $\mathfrak{q} \subset \mathcal{O}_{\mathbb K}$, for some $c > 0$ depending on $n$, wherever required. Namely, it should be used in the analogues of \cite[Lemma 10.2 and Corollary 10.5]{SW20}. Similar to \cite[Section 5]{OW16} and \cite[Section 10]{SW20}, \cref{thm:TheoremUniformExponentialMixingOfFrameFlow} then follows from the analogues of \cite[Lemma 10.3 and Corollary 10.5]{SW20} and using the result of \cite{SW20} for the base manifold $X$.


\nocite{*}
\bibliographystyle{alpha_name-year-title}
\bibliography{References}
\end{document}